\documentclass[a4paper]{amsart}
\usepackage[utf8]{inputenc}
\usepackage[T1]{fontenc}
\usepackage{lmodern}
\usepackage{amssymb}
\usepackage[all]{xy}
\usepackage{nicefrac,mathtools,enumitem}
\usepackage{microtype}

\usepackage[pdftitle={Fell bundles over inverse semigroups and twisted étale groupoids},
  pdfauthor={Alcides Buss and Ruy Exel},
  pdfsubject={Mathematics; ...}
]{hyperref}
\usepackage[lite]{amsrefs}
\newcommand*{\MRref}[2]{ \href{http://www.ams.org/mathscinet-getitem?mr=#1}{MR #1}}
\newcommand*{\arxiv}[1]{ \href{http://www.arxiv.org/abs/#1}{arXiv:#1}}

\numberwithin{equation}{section}
\theoremstyle{plain}
\newtheorem{theorem}[equation]{Theorem}
\newtheorem{lemma}[equation]{Lemma}
\newtheorem{proposition}[equation]{Proposition}

\theoremstyle{definition}
\newtheorem{definition}[equation]{Definition}

\theoremstyle{remark}
\newtheorem{remark}[equation]{Remark}
\newtheorem{example}[equation]{Example}

\DeclareMathOperator*{\dom}{dom}
\DeclareMathOperator*{\ran}{ran}
\DeclareMathOperator*{\spn}{span}
\DeclareMathOperator*{\cspn}{\overline{span}}
\DeclareMathOperator*{\Rep}{Rep}
\DeclareMathOperator*{\supp}{supp}

\newcommand*{\nb}{\nobreakdash}
\newcommand*{\Star}{\texorpdfstring{$^*$\nobreakdash-\hspace{0pt}}{*-}}
\newcommand*{\cstar}{\texorpdfstring{$C^*$\nobreakdash-\hspace{0pt}}{*-}}
\newcommand*{\CstarRed}{C^*_{\rm r}}

\newcommand*{\N}{\mathbb N} 
\newcommand*{\R}{\mathbb R} 
\newcommand*{\C}{\mathbb C} 
\newcommand*{\Torus}{\mathbb T} 

\newcommand*{\mfI}{\mathfrak{I}}
\newcommand*{\mfs}{\mathfrak{s}}

\newcommand*{\sbe}{\subseteq}
\newcommand*{\bound}{\mathbb B} 

\newcommand*{\dd}{\textup d} 

\newcommand*{\cont}{\mathcal C} 
\newcommand*{\contc}{\cont_\textup{c}} 
\newcommand*{\contz}{\cont_0} 
\newcommand*{\id}{\textup{id}} 
\newcommand*{\Ind}{\textup{Ind}} 
\newcommand*{\locs}{\Gamma} 
\newcommand*{\hils}{\mathcal H} 

\newcommand*{\A}{\mathcal{A}}     
\newcommand*{\At}{\A^{(2)}}       
\newcommand*{\Az}{\A^{(0)}}       
\newcommand*{\pA}{p}              
\newcommand*{\B}{\mathcal{B}}     
\newcommand*{\CC}{\mathcal{C}}    

\newcommand*{\G}{\mathcal{G}} 
\newcommand*{\s}{\mathrm{d}} 
\renewcommand*{\r}{\mathrm{r}} 
\newcommand*{\domain}{\mathrm{d}} 
\newcommand*{\range}{\mathrm{r}} 
\newcommand*{\g}{\gamma}
\newcommand*{\Gz}{\G^{(0)}}
\newcommand*{\Gt}{\G^{(2)}}

\newcommand*{\E}{\mathcal{E}} 
\newcommand*{\F}{\mathcal{F}} 
\newcommand*{\U}{\mathcal{U}} 
\newcommand*{\V}{\mathcal{V}}
\newcommand*{\defeq}{\mathrel{\vcentcolon=}}

\newcommand*{\into}{\hookrightarrow}
\newcommand*{\onto}{\twoheadrightarrow}
\newcommand*{\congto}{\xrightarrow\sim}
\renewcommand*{\O}{\mathcal{O}} 
\newcommand*{\pb}[1]{\mathcal I(#1)}
\newcommand*{\I}{\mathcal I} 
\newcommand*{\X}{\mathcal X} 
\newcommand*{\rest}[1]{|_{#1}} 
\newcommand*{\ep}{\epsilon} 

\newcommand*{\inv}{^{-1}}

\newcommand*{\Hx}{\hils_x}
\newcommand*{\calcat}[1]{\,{\vrule height8pt depth4pt}_{\,#1}}
\newcommand*{\overeq}[1]{\buildrel{#1}\over =}
\newcommand*{\eq}[1]{\buildrel{#1}\over =}
\newcommand*{\trip}[3]{(#1,#2,#3)}
\newcommand*{\qtrip}[3]{[#1,#2,#3]}
\newcommand*{\qstrip}[3]{[\![#1,#2,#3]\!]}
\newcommand*{\germ}[2]{[#1, #2]}
\newcommand*{\half}{^\frac{1}{2}}
\newcommand*{\usc}{upper-semicontinuous}
\font \eightsc =cmcsc8
\newcommand*{\pf}[1]{{\medskip \noindent({\eightsc #1}):}}
\renewcommand*{\t}{\theta}
\newcommand*{\braket}[2]{\left\langle#1\!\mid\!#2\right\rangle} 

\begin{document}
\title[Fell bundles over inverse semigroups]{Fell bundles over inverse semigroups \\ and twisted étale groupoids}

\author{Alcides Buss}
\email{alcides@mtm.ufsc.br}

\author{Ruy Exel}
\email{exel@mtm.ufsc.br}
\address{Departamento de Matemática\\
  Universidade Federal de Santa Catarina\\
  88.040-900 Florianópolis-SC\\
  Brasil}

\begin{abstract} Given a saturated Fell bundle $\A$ over an inverse
semigroup $S$ which is semi-abelian in the sense that the fibers over
the idempotents of $S$ are commutative, we construct a
twisted étale groupoid $(\G,\Sigma)$ such that $\A$ can be
recovered from $(\G,\Sigma)$ in a canonical way. As an application we
recover most of Renault's recent result on the classification of
Cartan subalgebras of C*-algebras through twisted étale groupoids.
\end{abstract}



\thanks{The first author was supported by CNPq/PNPD Grant Number
558420/2008-7,  and the second author was partially supported by CNPq.}

\maketitle

\tableofcontents

\section{Introduction}
\label{sec:introduction}

Among the most interesting examples of \cstar{algebras} one finds the
``dynamical \cstar{algebras}'', meaning \cstar{algebras} constructed out of some
dynamical system.  The interest in their study lies in the fact that
they are algebraic representations of their accompanying systems,
sometimes revealing features which are not immediately seen with a
naked eye.

The classical notion of a group action on a topological space, the
most basic form of a dynamical system, leads to the
\emph{crossed product}\/ or \emph{covariance \cstar{algebra}} which, over the
years, has proven to be an invaluable tool in the study of group
actions.

Dynamical systems often take slightly more sophisticated forms, such
as semigroup actions, pseudogroups, or topological groupoids and, in
most cases, crossed-product-like constructions may be performed
providing \cstar{algebras} mirroring dynamical features algebraically.

The huge variety of dynamical \cstar{algebras} has prompted some to reverse
the order of things and to look for dynamical data attached to
\cstar{algebras} which are not necessarily born from a dynamical system.
This point of view has been enormously successful, so much
so that it is now a standard tool in the study of \cstar{algebras}.  When
available it often gives important information on simplicity, the
structure of ideals, faithfulness of representations and K-theory,
among others.

One of the earliest attempts at uncovering dynamical data beneath
unsuspecting algebraic systems is Feldman and Moore's
\cite{feldman_more:Cartan.subalgebrasII} description of Cartan subalgebras of von Neumann
algebras via twisted measured equivalence relations. Kumjian
\cite{Kumjian:cstar.diagonals} and Renault \cite{RenaultCartan} have later made these
ideas to work in the context of \cstar{algebras} and, thanks to them, we
now know that Cartan subalgebras of \cstar{algebras} may be described via
twisted, essentially principal, \'etale groupoids.

Motivated by these developments, the second named author has recently
found a generalized notion of (non-commutative) Cartan subalgebras, for
which a similar characterization may be given \cite{Exel:noncomm.cartan}.
The \emph{dynamical} object underneath this characterization comes in the
form of a \emph{Fell bundle over an inverse semigroup}, a concept
introduced by Sieben in a talk given at the Groupoid Fest in
1998 (see \cite{SiebenFellBundles,Exel:noncomm.cartan}).
But, perhaps due to the fact that this concept is still deeply rooted
in Algebra, it may not immediately appear to deserve the label of a
\emph{dynamical system}.

The present work intends to bridge this gap, clarifying the
relationship between such Fell bundles and dynamics proper.

By definition a Fell bundle over an inverse semigroup $S$ consists of
a family $\A=\{\A_s\}_{s\in S}$ of Banach spaces, equipped with
bilinear \emph{multiplication} operations $\A_s\times \A_t \to \A_{st}$,
conjugate-linear \emph{involution} operations $\A_s \to \A_{s^*}$, and
\emph{inclusion} maps $\A_s\hookrightarrow \A_t$, whenever $s\leq t$.
All these data are required to satisfy certain natural axioms (see Definition~\ref{def:Fell bundles over ISG} below).


The expression \emph{Fell bundle}
has its roots in Fell's pioneering work \cite{fell_doran} and
is used in this work primarily to
refer to Fell bundles over inverse semigroups, as briefly defined
above, but the concept of Fell bundles over groupoids, as introduced by Kumjian and
Yamagami \cite{YamagamiFellBundles,Kumjian:fell.bundles.over.groupoids}, also plays a crucial role since the
latter provides examples of the former: given a Fell bundle $\B$ over
an étale groupoid $\G$, let $S$ be any inverse semigroup consisting of
bissections of $\G$, and for each $U\in S$, let $\A_U$ be the space of
all continuous sections of $\B$ over $U$ vanishing at infinity.  The
operations on $\B$ may be used to give the collection $\A=
\{\A_U\}_{U\in S}$ the structure of a Fell bundle over $S$.
Under mild conditions, we prove that the cross sectional \cstar{algebras}
of $\B$ and of $\A$ are isomorphic (see Theorem~\ref{teor:fell bundles over groupoids and ISG isomorphic}).
In case the groupoid $\G$ is Hausdorff, we follow a partition-of-unit argument appearing in
\cite[Theorem 7.1]{Quigg.Sieben.C.star.actions.r.discrete.groupoids.and.inverse.semigroups}.
This idea also appears in the unpublished preprint \cite[Proposition 3.5]{SiebenFellBundles} by Nánbor Sieben.
We also deal with the non-Hausdorff case by using some ideas and results of
\cite{Exel:inverse.semigroups.comb.C-algebras,Muhly.Williams.Equivalence.and.Disintegration}.

A special subcase of this construction is obtained when, starting from
a twisted groupoid $(\G,\Sigma)$, we form the associated Fell line
bundle $\B$ over $\G$.

Given any Fell bundle $\A$ over an inverse semigroup $S$, and given an
element $e$ in the idempotent semilattice $E(S)$, the
\emph{fiber over} $e$, namely $\A_e$, is always a \cstar{algebra}.  In the
special case of the above Fell bundle constructed from a twisted
\'etale groupoid $(\G, \Sigma)$, a bissection $U$ is idempotent if and
only if it is contained in the unit space $\Gz$.  In this case the
fiber over $U$ is just the algebra $\contz(U)$ of continuous
complex-valued functions \footnote{The twist over the unit space is always trivial so it may
be disregarded when $U\subseteq \Gz$.} vanishing at infinity on $U$,
which is obviously a commutative \cstar{algebra}.  This suggests the
terminology \emph{semi-abelian}, referring to Fell bundles for which
$\A_e$ is commutative for every idempotent element~$e$.

Our main result shows that every semi-abelian Fell bundle arises from
a twisted \'etale groupoid in the above fashion.  This gives substance
to the statement that Fell bundles over inverse semigroups are indeed
dynamical objects,  and it also supports the claim that general Fell
bundles (not necessarily semi-abelian ones) should be considered as
\emph{twisted groupoids with non-commutative unit space}.

Our techniques borrow lavishly from Kumjian \cite{Kumjian:cstar.diagonals} and Renault
\cite{RenaultCartan}, especially when constructing a groupoid from a
semi-abelian Fell bundle.  Our construction of the twist is also
heavily inspired by these works, although we have found it more
economical to construct the associated line bundle directly, without
first passing through the twist itself.  Should the twist be needed,
it can be easily recovered as the circle bundle associated to our line
bundle.


In the last section we apply our result to Cartan subalgebras. Given
a (commutative) Cartan subalgebra $B$ of a \cstar{algebra} $A$, the results
of \cite{Exel:noncomm.cartan} yield a Fell bundle $\A$ over an inverse
semigroup $S$, and an isomorphism $A \cong \CstarRed(\A)$, sending $B$
onto $\CstarRed(\E)$, where $\E$ is the restriction of $\A$ to the
idempotent semilattice of $S$. Since $B$ is commutative, $\A$ must be
semi-abelian, so we may apply our results in order to obtain a twisted
étale groupoid $(\G,\Sigma)$ together with a canonical isomorphism
$A\cong \CstarRed(\G,\Sigma)$, which sends $B$ onto
$\contz(\Gz)$. This proves most of Renault's main result in
\cite{RenaultCartan}.


It should be stressed that the groupoids that come out of our
construction are not necessarily Hausdorff, as opposed to the groupoids considered in \cite{Kumjian:cstar.diagonals}
and \cite{RenaultCartan}.  But, starting with Connes' work on foliation
groupoids \cite{Connes:Survey_foliations}, the recent literature on non-Hausdorff
groupoids has increased significantly providing efficient techniques
which often only require small changes in relation to the Hausdorff case.

Returning to the Hausdorff question, in our proof of Renault's
characterization of Cartan subalgebras, we have not seen how to prove,
without appealing to Renault's ideas, that the underlying groupoid is
Hausdorff.

Since Renault's proof relies on the existence of a conditional
expectation, we were led to conjecture that, if $\A$ is a semi-abelian
Fell bundle over an inverse semigroup such that there exists a
conditional expectation from $\CstarRed(\A)$ onto $\CstarRed(\E)$,
then the underlying groupoid should be Hausdorff.  However this
is not true as shown by the example given in Proposition~\ref{prop:Non-Hausdorff groupoid with conditional expectation}.
We are therefore forced to accept that Hausdorffness is not only a consequence of
the existence of the conditional expectation, but that it  also depends on
maximal abeliannes, as in Renault's result.

\section{Preliminaries}
\label{sec:preliminaries}

\subsection{Banach bundles}
\label{sec:Banach bundles}

In this section we will establish basic facts about {\usc} Banach
bundles that we need in the sequel.  Even though this is a well known
theory, covered in detail in several references (see for instance \cite{Dupre.Gillette.Banach.Bundles}, \cite{fell_doran},
\cite{Muhly.Williams.Renault's.Equivalence.Theorem} and references therein),
its applications to non-Hausdorff spaces have not been widely considered up to now.
Even though large parts of the theory generalize nicely to certain
non-Hausdorff spaces, most of the classical texts deal exclusively
with the Hausdorff case.

\begin{definition} (\cite{Dupre.Gillette.Banach.Bundles}, \cite{Muhly.Williams.Renault's.Equivalence.Theorem}) \label{DefineUSCBundle}
  Let $X$ be a (not necessarily Hausdorff) topological space.
  An \emph{{\usc}-Banach bundle} over $X$ is a pair $\A=(A,p)$ consisting of a topological space
  $A$ together with a continuous, open surjection
  $p\colon A\to X$.  It is moreover assumed that for each $x\in X$, the
  \emph{fiber  over $x$}, namely $\A_x\defeq p\inv(x)$, has the structure of a complex Banach space
satisfying:
\begin{enumerate}
\item[\textup{(i)}]\label{NormIsUSC}
The map $v\mapsto\|v\|$ is {\usc} from  $A$ to $\R^+$.
\item[\textup{(ii)}]\label{SumIsContinuous}
 The map \ $(v, w) \mapsto v+w$ \  is continuous from $\{(v,w)\in A\times A\colon p(v)=p(w)\}$
  (seen as a topological subspace of $A\times A$) to $A$.
\item[\textup{(iii)}]\label{ScalarProdIsCont} For each $\lambda\in\C$, the map $v\mapsto \lambda v$ is
    continuous from $A$ to $A$.
\item[\textup{(iv)}]\label{CompWithTop} If $\{v_i\}_i$ is a net in $A$ such that $p(v_i)\to
    x$ and $\|v_i\|\to 0$, then $\{v_i\}_i$ converges to $0_x$, the zero element in $\A_x$.
\end{enumerate}
If the map $v\mapsto \|v\|$ is continuous from $A$ to $\R^+$, we say that $\A$ is a \emph{continuous Banach bundle}.
\end{definition}

\begin{remark}
Although the norm on $\A$ need not be continuous, upper-semicontinuity forces $\|v_i\|$ to converge to $0$
whenever $v_i$ converges to $0_x$ for some $x\in X$. With this in mind, as observed in
\cite{Muhly.Williams.Renault's.Equivalence.Theorem} (see comments after Definition 3.1),
the same proof of \cite[Proposition II.13.10]{fell_doran}
can still be applied to an {\usc} Banach bundle $\A$ in order to get a stronger version of property (iii):
\begin{enumerate}
\item[\textup{(iii)'}] The map $(\lambda,v)\mapsto \lambda v$ is continuous from $\C\times A$ to $A$.
\end{enumerate}
\end{remark}

In what follows we are going to omit the bundle projection $p$ and the total space $A$
from our notation and identify the latter with the bundle $\A$ itself. Moreover, we also usually identify $\A$
with the collection $\{\A_x\}_{x\in X}$ of all fibers and actually write $\A=\{\A_x\}_{x\in X}$ to denote all the data.
We believe that this will not cause any confusion.

As already mentioned, we need efficient methods of constructing Banach
bundles, and we shall now devote ourselves to generalizing Fell and
Doran's main such tool \cite[Theorem II.13.18]{fell_doran}.

As a first step we suppose we are given a
(not necessarily Hausdorff) topological space $X$  and
a pairwise disjoint collection of Banach spaces $\{\A_x\}_{x\in X}$.
Denote by $\A$ the disjoint union of the $\A_x$ and let $p\colon\A\to X$
be the function which assigns $x$ to every element of $\A_x$.

\begin{definition} If $U$ is any subset of $X$, and if $\xi\colon U\to \A$ is a
function such that $\xi(x)\in\A_x$ for every $x\in U$, we say that
$\xi$ is a \emph{local section of $\A$ over $U$}. In this case, we denote the \emph{domain} of $\xi$ by $\dom(\xi)\defeq U$.
\end{definition}

In \cite[Theorem II.13.18]{fell_doran} one chooses a collection of globally defined
sections satisfying certain assumptions and one constructs a topology
on $\A$ with respect to which the given sections are continuous.

However, there are examples of locally Hausdorff spaces $X$, for which the
trivial one-dimensional bundle admits no global continuous compactly
supported section (see \cite[Example 1.2]{Khoshkam_Skandalis:regular.representation.groupoid}).
To remedy this situation we work here with local sections defined on open subsets of $X$.

Given a subset $\locs$ of local sections of $\A$, we shall write $\spn\locs$ for the set of the local sections of
the form
$$\sum\limits_{i=1}^n\lambda_i\xi_i\quad\mbox{with }\lambda_i\in\C,\, \xi_i\in\locs\mbox{ and }n\in \N,$$
where by definition,
$$\dom\left(\sum\limits_{i=1}^n\lambda_i\xi_i\right)\defeq\bigcap\limits_{i=1}^n\dom(\xi_i)\quad\mbox{and}\quad
\left(\sum\limits_{i=1}^n\lambda_i\xi_i\right)(x)\defeq\sum\limits_{i=1}^n\lambda_i\xi_i(x).$$

Note that the set of all local sections of $\A$ is \emph{not} a vector space with respect to the sum and scalar product defined above
because the sum fails to have additive inverses. However, it is a so called \emph{semi-vector space}, that is, all the axioms of a vector space
are satisfied, except for the existence of additive inverses.

The following result, which is a non-Hausdorff version of \cite[Theorem II.13.18]{fell_doran}, is certainly well-known for specialists and is essentially the same as Proposition~3.6 in \cite{Hofmann:Bundles}. However we have chosen to include the proof here for reader's convenience.

\begin{proposition}\label{T:topology on Banach bundles from predefined sections}
Let $\locs$ be a set of local sections of $\A$ whose domains $\dom(\xi)$ are open subsets of $X$ for all $\xi\in \locs$.
Suppose that\textup:
\begin{enumerate}
\item[\textup{(i)}] Given $v\in \A$, there exists $\xi\in\locs$ such that
$p(v)\in\dom(\xi)$ and $v=\xi(p(v))$.
\item[\textup{(ii)}] The map $x\mapsto \|\xi(x)\|$ is {\usc} from $\dom(\xi)$ to $\R^+$ for all $\xi\in \spn\locs$, that is,
if $\xi\in \spn\locs$ and if $\alpha$ is a positive real number, then
  $$
  \left\{x\in \dom(\xi)\colon\|\xi(x)\|<\alpha\right\}
  $$
  is open in $X$.
\end{enumerate}
Then there exists a unique topology on $\A$ making it an
{\usc} Banach bundle over $X$ and such that all the local sections $\xi$ of
$\spn\locs$, viewed as functions $\xi\colon\dom(\xi)\to\A$, are continuous.
A basis of open sets for this topology is given by the sets of the form
  $$
  \Omega(U,\xi,\ep) = \{v\in \A \colon p(v)\in U,\; \|v-\xi(p(v))\|<\ep\},
  $$
  where $\xi\in \locs$, $U$ is an open subset of $\dom(\xi)$,  and
$\ep>0$.

Moreover, $\A$ is a continuous Banach bundle with this topology if and only if the maps $\dom(\xi)\ni x\mapsto \|\xi(x)\|\in \R^+$
are continuous for all $\xi\in \locs$.
\end{proposition}
\begin{proof}
  In order to see that the above sets do indeed form the basis of some
topology on $\A$, let
  $$
  v_0\in \Omega(U_1,\xi_1,\ep_1) \cap \Omega(U_2,\xi_2,\ep_2),
  $$
  where $\ep_i>0$,  $\xi_i\in \locs$, and $U_i\sbe\dom(\xi_i)$, for $i=1,2$.
  Put $x_0=p(v_0)$, so $x_0\in U_i$, and $\|v_0-\xi_i(x_0)\|<\ep_i$.
Choose $\ep'_i>0$ such that
  $$
  \|v_0-\xi_i(x_0)\|<\ep'_i<\ep_i,
  $$
  and let $\eta\in\locs$, such that $\eta(x_0)=v_0$.  Put
  $$
  V_i= U_i\cap \{x\in \dom(\eta)\cap\dom(\xi_i)\colon \|\eta(x)-\xi_i(x)\| <\ep'_i\}.
  $$
  and observe that $V_i$ is open by (ii).  Notice that $x_0\in V_i$.
Letting $\delta>0$ be such that $\ep'_i+\delta <\ep_i$, we   claim
that
\begin{equation}\label{IsBasis}
  v_0\in\Omega( V_1\cap V_2,\eta,\delta) \sbe
  \Omega(U_1,\xi_1,\ep_1) \cap \Omega(U_2,\xi_2,\ep_2).
\end{equation}
  To show that $v_0$ is in the indicated set, notice that $p(v_0) = x_0\in V_i$, and
$$\|v_0-\eta(p(v_0))\|=0<\delta.$$

  To show that $\Omega( V_1\cap V_2,\eta,\delta) \sbe
\Omega(U_i,\xi_i,\ep_i)$, pick any $v$ belonging to the set in the
left-hand-side. So $p(v)\in V_i\sbe U_i$.  Moreover
  $$
  \|v- \xi_i(p(v)) \| \leq
  \|v- \eta(p(v)) \| + \|\eta(p(v)) -\xi_i(p(v)) \| \leq
  \delta + \ep'_i < \ep_i.
  $$
  This shows that $v\in \Omega(U_i,\xi_i,\ep_i)$, and hence concludes
the proof of Equation~\eqref{IsBasis}.  To see that the collection
of all the $\Omega(U,\xi,\ep)$ does indeed form the basis for a
topology on $\A$ it is therefore enough to check that its union equals
$\A$, but this is clear from (i) because any $v\in \A$ lies in
  $
  \Omega(\dom(\xi),\xi,\ep),
  $
  as long as $v=\xi(p(v))$.

\pf{$p$ is open} To see this it is enough to show that $p$ sends basic
open sets to open sets, but
this follows immediately from the fact that
$p\big(\Omega(U,\xi,\ep)\big) = U$.

\pf{the norm is \usc} We need to show that the set
  $$
  N_\alpha= \{v\in \A\colon \|v\|<\alpha\}
  $$
  is open for every $\alpha>0$.  So
let $v_0\in N_\alpha$, and
choose $\xi\in\locs$, such that $\xi(x_0)=v_0$, where $x_0=p(v_0)$.
Pick $\alpha'$ such that
  $
  \|v_0\|<\alpha'<\alpha,
  $
  and set
  $$
  U = \{x\in \dom(\xi)\colon \|\xi(x)\|<\alpha'\},
  $$
  so $U$ is open by (ii) and $x_0\in U$.  Choose $\ep>0$, such that
$\alpha'+\ep<\alpha$, and let us show that
\begin{equation}\label{FindBall}
  v_0\in\Omega(U,\xi,\ep)\sbe N_\alpha.
\end{equation}
On the one hand $p(v_0) = x_0\in U$, and on the other
  $\|v_0-\xi(p(v_0))\| = 0 <\ep$, so $v_0\in
\Omega(U,\xi,\ep)$.  Moreover, given any $v\in
\Omega(U,\xi,\ep)$, we have $p(v)\in U$, and
  $$
  \|v\| \leq
  \|v-\xi(p(v))\| + \|\xi(p(v))\| <
  \ep + \alpha' <\alpha.
  $$
  This proves Equation~\eqref{FindBall}, and hence we see that $N_\alpha$ is
indeed open.

\pf{the sum is continuous} Let $v_0,w_0\in \A$, with
$p(v_0)=p(w_0)$, and suppose that $v_0+w_0$ lies in some basic
open set $\Omega(U,\xi,\ep)$.  We need to provide an open subset
$\Delta$ of $\A\times\A$, containing $(v_0, w_0)$, and such that the
sum operation sends $\Delta$ into $\Omega(U,\xi,\ep)$, meaning that
  for every $(v,w)\in \Delta$ such that $p(v)=p(w)$, one has
$v+w\in \Omega(U,\xi,\ep)$.

Set $x_0=p(v_0)=p(w_0)$, so we have $x_0\in U$, and we may pick
$\ep'>0$ such that
  $$
  \|v_0+w_0 - \xi(x_0)\| < \ep' < \ep.
  $$
  Choose $\eta,\zeta\in\locs$, such that $v_0=\eta(x_0)$, and
$w_0=\zeta(x_0)$, and put
  $$
  V = U\cap \{x\in \dom(\eta)\cap\dom(\zeta)\cap\dom(\xi)\colon
  \|\eta(x)+\zeta(x) - \xi(x)\|<\ep'\},
  $$
  and notice that $V$ is open by (ii) and
$x_0\in V$.
  Let $\delta>0$ be such that $\ep'+2\delta<\ep$, and observe that
  $$
  (v_0,w_0)\in \Omega(V,\eta,\delta) \times \Omega(V,\zeta,\delta).
  $$
  We claim that the set in the right-hand side fulfills the task
assigned to $\Delta$, above.

  In fact, pick
  $
  (v,w)\in \Omega(V,\eta,\delta) \times \Omega(V,\zeta,\delta),
  $
  with $p(v)=p(w)$. In order to show that $v+w\in\Omega(U,\xi,\ep)$,
notice that $x\defeq p(v+w)\in V\subseteq U$, and
moreover
  $$
  \|v+w-\xi(x)\| =   \|v-\eta(x)\| + \|w-\zeta(x)\| +
\|\eta(x)+\zeta(x)-\xi(x)\| <
  2\delta + \ep' < \ep.
  $$
  This proves the continuity of the sum operation.

\pf{scalar multiplication is continuous} Given $\lambda\in\C$, and $v_0\in\A$ such that $\lambda v_0$ lies in some basic open set
$\Omega(U,\xi,\ep)$, we have $x_0\defeq  p(\lambda v_0) = p(v_0)\in U$, and there exists $\ep'$ such that
  $$
  \|\lambda v_0 -\xi(x_0)\|<\ep'<\ep.
  $$
  Choose $\eta\in\locs$ such that $\eta(x_0)=v_0$, and consider the open set
  $$
  V = U \cap \{x\in\dom(\eta)\cap\dom(\xi) \colon \|\lambda\eta(x)-\xi(x)\| < \ep'\},
  $$
  which clearly contains $x_0$.  Choosing $\delta>0$ such that
  $$
  \ep' + |\lambda|\, \delta < \ep,
  $$
  we claim that if $w\in \Omega(V,\eta,\delta)$, then $\lambda w\in
\Omega(U,\xi,\ep)$.  In fact, by assumption we have $p(\lambda w)=p(w) \in V\sbe U$, and
\begin{align*}
  \|\lambda w - \xi(p(w))\| & \leq \|\lambda w - \lambda\eta(p(w))\| + \|\lambda\eta(p(w)) - \xi(p(w))\| \\
  & < |\lambda|\, \|w - \eta(p(w))\| + \ep' < |\lambda|\, \delta + \ep' < \ep,
\end{align*}
  so indeed $\lambda w\in \Omega(U,\xi,\ep)$.

\pf{proof of axiom {\rm \ref{DefineUSCBundle}.(iv)}}
  Let $\{v_i\}_i$ and $x$ be as in \ref{DefineUSCBundle}.(iv).
  Denoting by $0_x$ the zero element of $\A_x$, let $\Omega(U,\xi,\ep)$ be a basic neighborhood of $0_x$, so $x\in
U$, and $\|\xi(x)\|<\ep$.  Pick $\ep'$ such that
  $$
  \|\xi(x)\|<\ep'<\ep,
  $$
  and let
  $$
  V = U\cap \{y\in \dom(\xi)\colon \|\xi(y)\|<\ep'\}.
  $$
  Clearly $V$ is open and $x\in V$.
  Choose $\delta>0$, such that $\ep'+\delta < \ep$.
For all $i$ bigger than or equal to some $i_0$, one therefore has
$p(v_i)\in V$, and $\|v_i\|<\delta$.  For $i\geq i_0$ one then has $p(v_i)\in U$, and
  $$
  \|v_i-\xi(p(v_i))\| \leq \|v_i \| + \| \xi(p(v_i))\| < \delta + \ep' < \ep,
  $$
  so $v_i\in \Omega(U,\xi,\ep)$, proving that $v_i\to 0_x$.
We conclude that $\A$ is an {\usc} Banach bundle with the topology which has the sets $\Omega(U,\xi,\ep)$ as basis.
Now, to see that each $\xi\in \spn\locs$ is continuous as a function $\xi\colon\dom(\xi)\to \A$,
take any $\eta\in \locs$, suppose that $U\sbe\dom(\eta)$ is an open subset, and take $\ep>0$. Then we have
\begin{align*}
\{x\in \dom(\xi)\colon\xi(x)\in \Omega(U,\eta,\ep)\}&=\{x\in \dom(\xi)\colon x\in U\mbox{ and }\|\xi(x)-\eta(x)\|<0\}\\
    &=U\cap\{x\in \dom(\xi)\cap\dom(\eta)\colon\|\xi(x)-\eta(x)\|<\ep\},
\end{align*}
which is open by (ii) since $\xi-\eta\in \spn\locs$. This says that $\xi\colon\dom(\xi)\to \A$ is continuous.

To see that the topology on $\A$ is uniquely determined, assume that $\A$ has a topology making it an {\usc} Banach bundle
and such that all the local sections in $\spn\locs$ are continuous. Since the norm is {\usc}, and since the map $p\inv(U)\ni v\mapsto v-\xi(p(v))\in \A$ is continuous for every $\xi\in \locs$ and every open subset $U\sbe\dom(\xi)$, the sets $\Omega(U,\xi,\ep)$ must be open in $\A$.
Moreover, given any open subset $\V\sbe\A$ and $v_0\in \V$,
we claim that there are $\xi\in \locs$, $U\sbe\dom(\xi)$ open and $\ep>0$ such that $v_0\in \Omega(U,\xi,\ep)\sbe\V$.
By (i), there is  a local section $\xi\in \locs$ such that $x_0\defeq p(v_0)\in \dom(\xi)$ and  $\xi(x_0)=v_0$. Suppose, by contradiction,
that for any open subset $U\sbe\dom(\xi)$ containing $x_0$ and for any $\ep>0$, the set $\Omega(U,\xi,\ep)$ is not
contained in $\V$. This yields an element $v_{U,\ep}\in \A$ which is not contained in $\V$ and satisfies $p(v_{U,\ep})\in U$
and $\|v_{U,\ep}-\xi(p(v_{U,\ep}))\|<\ep$. Then the net $\{p(v_{U,\ep}\}_{U,\ep}$ converges to $x_0$ (where the pairs $U,\ep$ are directed in the canonical way) and the net $w_{U,\ep}\defeq v_{U,\ep}-\xi(p(v_{U,\ep}))$ converges in norm to $0$.
By Definition~\ref{DefineUSCBundle}(iv), the net $w_{U,\ep}$ converges to $0_{x_0}$ in $\A$. Since $\xi$ and $p$ are continuous,
$\xi(p(v_{U,\ep}))$ converges to $\xi(p(x_0))=v_0$, and hence $v_{U,\ep}$ converges to $v_0$. This is a contradiction because
$\V$ is open, $v_0\in \V$ and $v_{U,\ep}\notin\V$ for all $U,\ep$. This proves our claim. As a consequence,
the sets $\Omega(U,\xi,\ep)$ form a basis for the topology on $\A$, and therefore it must be the same topology we constructed above.

Finally, suppose that $\A$ is a continuous Banach bundle with the
above topology. Since the $\xi\in \locs$ are continuous and the norm on $\A$ is continuous,
the maps $\dom(\xi)\ni x\mapsto \|\xi(x)\|\in \R^+$ must be continuous for all $\xi\in \locs$.
Conversely, suppose these maps are continuous. To show that the map $\A\ni v\mapsto\|v\|\in \R^+$ is continuous,
it is enough to prove lower semi-continuity since we already know it is {\usc}. Thus we have to show that the sets
$M_\alpha\defeq \{v\in \A\colon\|v\|>\alpha\}$ are open in $\A$ for all $\alpha>0$. Take $v_0\in M_\alpha$ and choose $\xi\in \locs$
with $x_0=p(v_0)\in \dom(\xi)$ and $\xi(x_0)=v_0$. Pick $\alpha'$ satisfying $\alpha<\alpha'<\|v_0\|$ and define
$U\defeq \{x\in \dom(\xi)\colon\|\xi(x)\|>\alpha'\}$.
If $0<\ep<\alpha'-\alpha$, then every $v\in \Omega(U,\xi,\ep)$ satisfies
$$\|v\|\geq\|\xi(p(v))\|-\|\xi(p(v))-v\|>\alpha'-\ep>\alpha,$$
so that $v_0\in \Omega(U,\xi,\ep)\sbe M_\alpha$. This concludes the proof.
\end{proof}

\subsection{Fell bundles over groupoids and inverse semigroups}
\label{sec:FellBundles}

Let us begin by precisely identifying the class of groupoids that will
interest us.

\begin{definition}[\cite{RenaultThesis}]\label{DefineGroupoid}
A groupoid $\G$ is said to be a \emph{topological groupoid} if it is
equipped with a (not necessarily Hausdorff) topology relative to which
the multiplication and inversion operations are continuous. We say
that $\G$ is \emph{\'etale} if, in addition, the unit space $\Gz$ is
locally compact and Hausdorff in the relative topology, and the range map
  $
  \range\colon \G\to\Gz
  $
  (and consequently also the domain map
  $
  \domain\colon \G\to\Gz
  $)
  is a local homeomorphism.
\end{definition}

The following definition of Fell bundles over groupoids is the same
given by Alex Kumjian in \cite{Kumjian:fell.bundles.over.groupoids}
(see also \cite{YamagamiFellBundles}),
except that Kumjian only works with Hausdorff groupoids and continuous
Banach bundles. Continuous Banach bundles would be enough for us, but
our main results require non-Hausdorff groupoids.
Given our interest in \'etale groupoids, we shall restrict ourselves
to this situation although the concept below actually makes sense for
any topological groupoid.

From now on we fix an \'etale groupoid $\G$.

\begin{definition}\label{D:definition of Fell bundle over groupoid}
A \emph{Fell bundle} over $\G$ is an {\usc} Banach bundle $\A=\{\A_\g\}_{\g\in \G}$ over $\G$
together with a \emph{multiplication}
$$\cdot\colon\At=\{(a,b)\in \A\times\A\colon (\pA(a),\pA(b))\in \Gt\}\to \A,\quad (a,b)\mapsto a\cdot b,$$
where $\pA\colon\A\to \G$ is the bundle projection, and an \emph{involution}
$$^{*}\colon\A\to \A,\quad a\mapsto a^*$$
satisfying the following properties:
\begin{enumerate}
\item[\textup{(i)}] $\pA(\g_1\g_2)=\pA(\g_1)\pA(\g_2)$, that is, $\A_{\g_1}\cdot\A_{\g_2}\sbe\A_{\g_1\g_2}$ whenever $(\g_1,\g_2)\in\Gt$,
                        and the multiplication map $\A_{\g_1}\times\A_{\g_2}\to \A_{\g_1\g_2}$, $(a,b)\mapsto a\cdot b$ is bilinear;
\item[\textup{(ii)}] the multiplication on $\A$ is associative, that is, $(a\cdot b)\cdot c=a\cdot(b\cdot c)$
                        whenever $a,b,c\in \A$ and this makes sense;
\item[\textup{(iii)}] the multiplication $\cdot \colon\At\to \A$ is continuous, where $\At$ carries the topology
                        induced from the product topology on $\A\times\A$;
\item[\textup{(iv)}] $\|a\cdot b\|\leq\|a\|\|b\|$ for all $a,b\in \A$;
\item[\textup{(v)}] $\pA(a^*)=\pA(a)^*$ for all $a\in \A$, that is, $\A_\g^*\sbe\A_{\g\inv}$ for all $\g\in \G$,
                    and the involution map $^{*}\colon\A_\g\to \A_{\g\inv}$ is conjugate linear;
\item[\textup{(vi)}] $(a^*)^*=a$, $\|a^*\|=\|a\|$ and $(a\cdot b)^*=b^*\cdot a^*$ for all $a,b\in \A$;
\item[\textup{(vii)}] the involution $^{*}\colon\A\mapsto \A$ is continuous;
\item[\textup{(viii)}] $\|a^*a\|=\|a\|^2$ for all $a\in \A$; and
\item[\textup{(ix)}] $a^*a$ is a positive element of the \cstar{algebra} $\A_{\s(\g)}$ for all $\g\in \G$ and $a\in \A_\g$.
\end{enumerate}
We say that $\A$ is \emph{saturated} if the closed linear span of
$\A_{\g_1}\cdot\A_{\g_2}$ equals $\A_{\g_1\g_2}$ for all
$(\g_1,\g_2)\in \Gt$.
\end{definition}

Note that for a unit $x\in \Gz$, the fiber $\A_x$ is in fact a \cstar{algebra} with respect to the restricted
multiplication and involution from $\A$. This follows from (i)-(viii) above, so that (ix) makes sense.
Moreover, the restriction $\Az=\A\rest{\Gz}$ is an {\usc} \cstar{bundle}
\cite{Nilsen:Bundles}. Conversely, if $X$ is any locally compact Hausdorff space,
and if $\B$ is an {\usc} \cstar{bundle} over $X$, then $\B$ is a Fell bundle in the sense above over
$X$ considered as a groupoid in the trivial way.

\begin{proposition}\label{prop:continuity of multiplication and involution with local section}
Let $\A$ be an {\usc} Banach bundle over $\G$,  and let $\locs$ be a set of local
sections of $\A$ satisfying the properties (i) and (ii) in
Proposition~\textup{\ref{T:topology on Banach bundles from predefined
sections}}. Suppose that $\cdot\colon\At\to\A$ is a multiplication and
$^{*}\colon\A\to \A$ is an involution on $\A$ satisfying all the axioms
(i)-(ix) in Definition~\textup{\ref{D:definition of Fell bundle over
groupoid}} except, possibly, for (iii) and (vii). Then
\begin{enumerate}
\item[\textup{(i)}] the multiplication
$\cdot\colon\At\to\A$ is continuous if and only if the map
$$\Gt\cap\big(\dom(\xi)\times\dom(\eta)\big)\ni(\g_1,\g_2)\mapsto
\xi(\g_1)\cdot\eta(\g_2)\in \A$$ is continuous for all $\xi,\eta\in\locs$;
\item[\textup{(ii)}] the involution $^{*}\colon\A\to \A$ is
continuous if and only if the map $$\G\supseteq\dom(\xi)\ni\g\mapsto
\xi(\g)^*\in \A$$ is continuous for all $\xi \in \locs$.
\end{enumerate} \end{proposition} \begin{proof} Although we do not
have the same set of hypothesis, the same proof of
\cite[VIII.2.4]{fell_doran} can be applied to our situation in order
to prove (i). And (ii) is also proved in a similar way (see also
\cite[VIII.3.2]{fell_doran}).  \end{proof}

Next, we recall the definition of Fell bundles over inverse semigroups and compare with the one over groupoids defined above.
This concept was first introduced by Nánbor Sieben in 1998 (see \cite{SiebenFellBundles}),
and we borrow some of his ideas. Unfortunately, his results were not published, but the main ideas are contained in the preprint
\cite{Exel:noncomm.cartan} by the second named author.

\begin{definition}\label{def:Fell bundles over ISG}
Let $S$ be an inverse semigroup. A \emph{Fell bundle} over $S$ is a
collection $\A=\{\A_s\}_{s\in S}$ of Banach spaces $\A_s$
together with a \emph{multiplication} $\cdot\colon\A\times\A\to \A$, an
\emph{involution} $^{*}\colon\A\to \A$, and linear maps
$j_{t,s}\colon\A_s\to \A_t$ whenever $s\leq t$, satisfying the following properties\textup:
\begin{enumerate}
\item[\textup{(i)}] $\A_s\cdot\A_t\sbe\A_{st}$ and the multiplication
is bilinear from $\A_s\times\A_t$ to $\A_{st}$ for all $s,t\in S$;
\item[\textup{(ii)}] the multiplication is associative, that is, $a\cdot(b\cdot c)=(a\cdot b)\cdot c$ for all $a,b,c\in \A$;
\item[\textup{(iii)}] $\|a\cdot b\|\leq \|a\|\|b\|$ for all $a,b\in \A$;
\item[\textup{(iii)}] $\A_s^*\sbe\A_{s^*}$ and the involution is conjugate linear from $\A_s$ to $\A_{s^*}$;
\item[\textup{(iv)}] $(a^*)^*=a$, $\|a^*\|=\|a\|$ and $(a\cdot b)^*=b^*\cdot a^*$;
\item[\textup{(v)}] $\|a^*a\|=\|a\|^2$ and $a^*a$ is a positive element of the \cstar{algebra} $\A_{s^*s}$ for all $s\in S$ and $a\in \A_s$;
\item[\textup{(vi)}] $j_{t,s}\colon\A_s\to \A_t$ is an isometric linear map for all $s\leq t$ in $S$;
\item[\textup{(vii)}] if $r\leq s\leq t$ in $S$, then $j_{t,r}=j_{t,s}\circ j_{s,r}$;
\item[\textup{(viii)}] if $s\leq t$ and $u\leq v$ in $S$, then $j_{t,s}(a)\cdot j_{v,u}(b)=j_{tv,su}(a\cdot b)$
                        for all $a\in \A_s$ and $b\in \A_u$. In other words, the following diagram commutes:
$$\xymatrix{
\A_s\times\A_u \ar[r]^{\mu_{s,u}}\ar[d]_{j_{t,s}\times j_{v,u}} & \A_{su}\ar[d]^{j_{tv,su}}\\
\A_t\times\A_v \ar[r]_{\mu_{t,v}} & \A_{tv}}
$$
where $\mu_{s,u}$ and $\mu_{t,v}$ denote the multiplication maps.
\item[\textup{(ix)}] if $s\leq t$ in $S$, then $j_{t,s}(a)^*=j_{t^*,s^*}(a^*)$ for all $a\in \A_s$, that is, the diagram
$$\xymatrix{
\A_s\ar[r]^{^{*}}\ar[d]_{j_{t,s}} & \A_{s^*}\ar[d]^{j_{t^*,s^*}}\\
\A_t\ar[r]_{^{*}} & \A_{t^*}
}
$$
commutes. If $\A_s\cdot\A_t$ spans a dense subspace of $\A_{st}$ for all $s,t\in S$, we say that $\A$ is \emph{saturated}.
\end{enumerate}
\end{definition}

\begin{example}\label{ex:Fell bundle over semigroups associated to Fell bundles over groupoids}
Let $\G$ be an étale groupoid, and let $\B=\{\B_\g\}_{\g\in \G}$ be a Fell bundle over $\G$
(as in Definition~\ref{D:definition of Fell bundle over groupoid}).
Let $S(\G)$ be the inverse semigroup of all bissections
in $\G$. Recall that an open subset $U\sbe\G$ is a \emph{bissection} if the
restrictions $\s_U\colon U\to \s(U)$ and $\r_U\colon U\to \r(U)$ of the source and range maps
$\s\colon\G\to \G^{(0)}$ and $\r\colon\G\to \G^{(0)}$ are homeomorphisms. Let $S\sbe S(\G)$ be an inverse subsemigroup.
Given $U\in S$, define $\A_U$ to be the space $\contz(\B_U)$ of continuous sections
vanishing at infinity of the restriction $\B_U$ of $\B$ to $U\sbe\G$.
Then the collection $\A=\{\A_U\}_{U\in S}$ (disjoint union of the $\A_U$'s) is a Fell bundle over $S$ with respect to the following structure:
\begin{itemize}
\item the multiplication $\A_U\times\A_V\to \A_{UV}$ is defined by
        $$(\xi\cdot\eta)(\g)\defeq \xi(\r_U\inv(\r(\g)))\cdot \eta(\s_V\inv(\s(\g)))$$
        whenever $\xi\in \contz(\B_U)$, $\eta\in \contz(\B_V)$ and $\g\in UV$ (recall that $UV$ is defined as the
        set of all products $\g_1\g_2$ with $\g_1\in U$, $\g_2\in V$ and $\s(\g_1)=\r(\g_2)$);
\item the involution $\A_U\to\A_{U^*}$ is defined by
        $$\xi^*(\g)\defeq \xi(\g\inv)^*$$
        whenever $\xi\in \contz(\B_U)$ and $\g\in U^*\defeq \{\g\inv\colon\g\in U\}$;
\item the inclusion maps $j_{V,U}\colon\A_U\into \A_V$ are defined in the canonical way: if $U\leq V$, that is, $U\sbe V$, then
      we may extend a section $\xi\in \contz(\B_U)$ by zero outside $U$ and view it as a section $\tilde\xi\in \contz(\B_V)$.
      Thus $j_{V,U}(\xi)\defeq \tilde\xi$, where $\tilde\xi$ denotes the extension of $\xi$ by zero.
\end{itemize}
The proof that $\A$ is in fact a Fell bundle is not difficult and is left to the reader. Let us just remark that the multiplication
above is well-defined. Moreover, since $U,V$ are bissections, there is a unique way to write $\g\in UV$ as a product $\g=\g_1\g_2$ with
$\g_1\in U$ and $\g_2\in V$. Indeed, we must have $\g_1=\r_U\inv(\r(\g))$ and $\g_2=\s_V\inv(\s(\g))$.
Note that the multiplication we defined in $\A$ uses the multiplication of $\B$. Thus
$(\xi\cdot\eta)(\g)=\xi(\g_1)\cdot\eta(\g_2)\in \B_{\g_1}\cdot\B_{\g_2}\sbe\B_{\g_1\g_2}=\B_\g$, so $\xi\cdot\eta$ is in fact a section.
Note that the associativity of the multiplication in $\A$ now follows
easily from the associativity of the products in $\G$ and $\B$.
Moreover, it also follows that the multiplication in $\A$ is the usual convolution:
$$(\xi\cdot\eta)(\g)=(\xi*\eta)(\g)=\sum\limits_{\g=\g_1\g_2}\xi(\g_1)\cdot\eta(\g_2).$$
\end{example}

We want to prove that the Fell bundles described in Example~\ref{ex:Fell bundle over semigroups associated to Fell bundles over groupoids}
have same universal \cstar{algebras}. We consider two cases. First, if
the groupoid $\G$ is Hausdorff, we shall follow an idea appearing in
\cite[Theorem 7.1]{Quigg.Sieben.C.star.actions.r.discrete.groupoids.and.inverse.semigroups} which uses a partition-of-unit argument.
This idea does not seem to work in the non-Hausdorff case. However, under suitable separability conditions, we adapt an idea appearing
in \cite{Exel:inverse.semigroups.comb.C-algebras} together with a disintegration result of \cite{Muhly.Williams.Equivalence.and.Disintegration}
to solve the non-Hausdorff case as well.

We need to work with bundles with incomplete fibers, and for this we shall need the following technical result:

\begin{lemma}\label{lem:algebraic Fell bundle with same enveloping algebra}
Let $\A=\{\A_s\}_{s\in S}$ be a Fell bundle over an inverse semigroup $S$, and suppose that $\A^0=\{\A_s^0\}_{s\in S}$ is a sub-bundle
of $\A$ satisfying the following properties:
\begin{enumerate}
\item[\textup{(i)}] $\A_s^0$ is a dense subspace of $\A_s$ for all $s\in S$;
\item[\textup{(ii)}] $\A_s^0\cdot\A_t^0\sbe\A_{st}^0$ and $(\A_s^0)^*\sbe\A_{s^*}^0$ for all $s,t\in S$;
\item[\textup{(iii)}] $\A_{ss^*}^0\A_s\sbe \A_s^0$ and $\A_s\A_{s^*s}^0\sbe\A_s^0$ for all $s\in S$; and
\item[\textup{(iv)}] $j_{t,s}(\A_s^0)\sbe\A_t^0$ whenever $s\leq t$.
\end{enumerate}
Then $C^*(\A)$ is the enveloping \cstar{algebra} of the quotient $\contc(\A^0)/\I_\A^0$, where $\I_\A^0$ is the ideal of $\contc(\A^0)$ defined by
$$\I_\A^0\defeq \spn\{a\delta_s-j_{t,s}(a)\delta_t\colon a\in \A_s^0,\,s\leq t\}.$$
\end{lemma}
\begin{proof}
Recall that $\I_\A=\spn\{a\delta_s-j_{t,s}(a)\delta_t\colon a\in \A_s,\,s\leq t\}$
is an ideal of $\contc(\A)$ and $C^*(\A)$ is the enveloping \cstar{algebra} of the quotient $\contc(\A)/\I_\A$.
By definition, a pre-representation of $\A$, once linearly extended to $\contc(\A)=\spn\{a\delta_s\colon a\in \A_s\}$,
is a representation if and only if it vanishes on the ideal $\I_\A$. Let $\pi$ be a pre-representation
of $\A^0$ on some \cstar{algebra} $C$, that is, $\pi$ is a map $\A^0\to C$
which respects the multiplication and involution of $\A^0$ (which are well-defined by (ii)) and is linear when
restricted to the fibers $\A_s^0$. Then $\pi$ extends to a pre-representation $\tilde\pi$ of $\A$ into $C$. In fact,
$(ii)$ implies that $\A_e^0$ is an ideal of $\A_e$ for all $e\in E(S)$. It follows from
\cite[Lemma 2.3]{Quigg.Sieben.C.star.actions.r.discrete.groupoids.and.inverse.semigroups} that the restriction $\pi_e\colon\A_e^0\to C$
of $\pi$ to $\A_e^0$ (which is a \Star{homomorphism}) is contractive, that is, $\|\pi_e(a)\|\leq \|a\|$ for all $a\in \A_e^0$.
Using the \cstar{identity} and the fact that $\pi$ respects multiplication and involution, this implies that all
restrictions $\pi_s\colon\A_s^0\to C$ are contractive and therefore extend to $\tilde\pi_s\colon\A_s\to C$.
Of course, $\tilde\pi=\{\tilde\pi_s\}_{s\in S}$ is a pre-representation since $\pi$ is. Now assume that
$\pi$ vanishes on the ideal $\I_\A^0$, that is, $\pi$ is \emph{coherent} in the sense that $\pi(a)=\pi(j_{t,s}(a))$
whenever $a\in \A_s^0$ and $s\leq t$. Then the extension $\tilde\pi$ is also coherent. Indeed, take a bounded approximate unit $(e_i)$
for $\A_{s^*s}$ contained in $\A_{s^*s}^0$. If $a\in \A_s$, then $ae_i\in \A_s^0$ by (iii), so that
$$\tilde\pi(a)=\tilde\pi(\lim_iae_i)=\lim_i\pi(ae_i)=\lim_i\pi(j_{t,s}(ae_i))=\lim_i\tilde\pi(j_{t,s}(ae_i))=\tilde\pi(j_{t,s}(a)).$$
Conversely, if $\tilde\pi$ is coherent, so is its restriction $\pi$. We conclude that representations
(that is, coherent pre-representations) of $\A^0$ correspond bijectively to representations of $\A$, whence the result follows.
\end{proof}

As in Example~\ref{ex:Fell bundle over semigroups associated to Fell bundles over groupoids}, let $\B=\{\B_\g\}_{\g\in \G}$
be a Fell bundle over an étale groupoid $\G$, and let $\A=\{\contz(\B_U)\}_{U\in S}$ be the associated Fell bundle over
an inverse subsemigroup $S\sbe S(\G)$. Given $U\in S$, we define $\A_U^0\defeq \contc(\B_U)$. Then the sub-bundle $\A^0=\{\A_U^0\}_{U\in S}$
of $\A$ satisfies the properties (i)-(iv) in statement of Lemma~\ref{lem:algebraic Fell bundle with same enveloping algebra}.
The only non-trivial property to be checked is (iii). But if $\xi\in \contz(\B_U)$, $\eta\in\contc(\B_{U^*U})$
and $K=\supp(\eta)\sbe U^*U=\s(U)$, then $(\xi\cdot\eta)(\g)=\xi(\g)\eta(\s(\g))$ for all $\g\in U$, so that
$\supp(\xi\cdot\eta)\sbe \s_U\inv(K)$ is a compact subset of $U$,
and hence $\xi\cdot\eta\in \contc(\B_U)$. Analogously, $\contc(\B_{UU^*})\cdot \contz(\B_U)\sbe\contc(\B_U)$.
As a consequence of Lemma~\ref{lem:algebraic Fell bundle with same enveloping algebra}, $C^*(\A)$ is the enveloping
\cstar{algebra} of $\contc(\A^0)/\I_\A^0$. Under some mild conditions, we shall prove in what follows that $C^*(\A)$ is
isomorphic to the full \cstar{algebra} $C^*(\B)$ of the Fell bundle $\B$. First, let us recall the definition of $C^*(\B)$.

\begin{definition}\label{def:definition of C_c(B)}
Given a Fell bundle $\B=\{\B_\g\}_{\g\in \G}$ over a locally compact étale groupoid $\G$,
we write $\contc(\B)$ for the vector space of sections $\xi$ of $\B$ which can be written as a finite sum of the form
$\xi=\sum\limits_{i=1}^n\xi_i,$ where each $\xi_i\colon U_i\to \B$ is a compactly supported, continuous local section of $\B$
over some Hausdorff open subset $U_i\sbe\G$, extended by zero outside $U_i$ and viewed as a global section $\xi_i\colon\G\to \B$.
\end{definition}

Alternatively, since the bissections in $\G$ form a basis for its topology \cite[Proposition 3.5]{Exel:inverse.semigroups.comb.C-algebras},
we may restrict to local sections $\xi_i\colon U_i\to \B$ supported on bissections $U_i$ in the above definition.
Notice that bissections are open and Hausdorff by definition. Also note that any open Hausdorff subset $U\sbe \G$ is
locally compact with respect to the induced topology \cite[Proposition 3.7]{Exel:inverse.semigroups.comb.C-algebras}.

Let us warn the reader that, in general, if $\G$ is not Hausdorff, sections in $\contc(\B)$ are not continuous with respect to the global topology.
Of course, if $\G$ is Hausdorff, $\contc(\B)$ coincides with the usual space of compactly supported, continuous
(global) sections of $\B$. In any case, the vector space $\contc(\B)$ always has a canonical \Star{algebra} structure.
The multiplication on $\contc(\B)$ is the convolution product
$$(\xi*\eta)(\g)=\sum\limits_{\g=\g_1\g_2}\xi(\g_1)\eta(\g_2)$$
and the involution is defined by $\xi^*(\g)=\xi(\g\inv)^*$ for all $\xi,\eta\in \contc(\B)$.
As observed in Example~\ref{ex:Fell bundle over semigroups associated to Fell bundles over groupoids},
if $\xi$ is supported in a bissection $U\sbe\G$ and $\eta$ is supported in a bissection $V\sbe\G$, then
$\xi*\eta$ is supported in the bissection $UV$ and $(\xi*\eta)(\g)=\xi(\g_1)\eta(\g_2)$ whenever $\g\in UV$
is (uniquely) written in the form $\g=\g_1\g_2$ with $\g_1\in U$ and $\g_2\in V$. In particular,
this shows that the convolution product is well-defined on $\contc(\B)$.

By definition, $C^*(\B)$ is the enveloping \cstar{algebra} of $\contc(\B)$.
The same argument presented in \cite[Proposition 3.14]{Exel:inverse.semigroups.comb.C-algebras} implies
that $\|\pi(\xi)\|\leq\|\xi\|_\infty$ for any \Star{representation} $\pi$ of $\contc(\B)$
whenever $\xi\in \contc(\B)$ is supported in some bissection of $\G$.
Therefore the enveloping \cstar{algebra} of $\contc(\B)$ in fact exists.

\begin{definition} (Compare \cite[Proposition 5.4]{Exel:inverse.semigroups.comb.C-algebras}) \label{def:wide inv. semigroup}
Let $\G$ be an étale groupoid.
We say that an inverse subsemigroup $S\sbe S(\G)$ is \emph{wide} if the following properties hold:
\begin{enumerate}
\item[{(i)}] $S$ is a covering for $\G$, that is, $\G=\bigcup\limits_{U\in S}U$, and
\item[{(ii)}] given $U,V\in S$ and $\g\in \G$, there is $W\in S$ such that $\g\in W\sbe U\cap V$.
\end{enumerate}
\end{definition}

\begin{theorem}\label{teor:fell bundles over groupoids and ISG isomorphic}
Let $\B=\{\B_\g\}_{\g\in \G}$ be a Fell bundle over an étale groupoid,
and let $\A=\{\contz(\B_U)\}_{U\in S}$ be the associated Fell bundle over a wide inverse subsemigroup $S\sbe S(\G)$.
Consider the sub-bundle $\A^0=\{\contc(\B_U)\}_{U\in S}$ of $\A$ as above. If $\G$ is either Hausdorff, or $\G$ is second countable
and the section algebras $\contz(\B_U)$ are separable for all $U\in S$, then the canonical map
$\Psi\colon\contc(\A^0)\to \contc(\B)$ defined by
$$\Psi\left(\sum\limits_{U\in F}\xi_U\delta_U\right)=\sum\limits_{U\in F}\xi_U$$
whenever $F$ is a finite subset of $S$ and $\xi_U\in \contc(\B_U)$ for all $U\in F$,
induces an isomorphism of \cstar{algebras} $C^*(\A)\cong C^*(\B)$.
\end{theorem}
\begin{proof}
By linearity, to show that $\Psi$ is a \Star{homomorphism}, it suffices to check the equalities $$\Psi\bigl((\xi\delta_U)\cdot(\eta\delta_V)\bigr)=\xi*\eta\quad \mbox{and}\quad \Psi\bigl((\xi\delta_U)^*\bigr)=\xi^*$$
for all $\xi\in \contc(\B_U)$ and $\eta\in \contc(\B_V)$, where $U,V\in S$. By definition of the \Star{}algebra structure of $\contc(\A^0)$,
we have $(\xi\delta_U)\cdot(\eta\delta_V)=(\xi\cdot\eta)\delta_{UV}$ and $(\xi\delta_U)^*=\xi^*\delta_{U^*}$.
And we have already observed in Example~\ref{ex:Fell bundle over semigroups associated to Fell bundles over groupoids}
that $\xi\cdot\eta=\xi*\eta$. Moreover, $\contc(\B)$ is generated by sums of the form
$\sum_{U\in F}\xi_U$ as above because $S$ covers $\G$.
Therefore, $\Psi$ is a surjective \Star{homomorphism} as stated.
Note that $\Psi$ vanishes on the ideal $\I_\A^0\sbe\contc(\A^0)$ defined in Lemma~\ref{lem:algebraic Fell bundle with same enveloping algebra}
and hence induces a \Star{}homomorphism $\tilde\Psi\colon C^*(\A)\to C^*(\B)$ by the same lemma.
Since $\Psi$ is surjective, so is $\tilde\Psi$. The most difficult part of the proof is to show that $\tilde\Psi$
is injective. Here we consider two cases:

\pf{Case 1: $G$ is Hausdorff} In this case we are going to follow the same idea as in the proof of Theorem 7.1 in
\cite{Quigg.Sieben.C.star.actions.r.discrete.groupoids.and.inverse.semigroups}.
We first show that the kernel of $\Psi$ is equal to
$$\I=\spn\left\{\xi\delta_U-\xi\delta_V\colon\xi\in \contc(\B) \mbox{ and }\supp(\xi)\sbe U\cap V \right\}.$$
Of course, $\I\sbe \ker(\Psi)$. For the opposite inclusion, take a finite sum of the form $\sum_{i=1}^n\xi_i\delta_{U_i}$,
where $U_i\in S$ and $\xi_i\in \contc(\B_{U_i})$, and suppose
$$\Psi\left(\sum_{i=1}^n\xi_i\delta_{U_i}\right)=\sum_{i=1}^n\xi_i=0.$$
If $n=1$, then $\xi_1=0$ so that $\xi_1\delta_{U_1}=0\in \ker(\Psi)$. Thus we may assume $n>1$.
Consider the set $\mfI$ of all subsets of $\{1,\ldots,n\}$ with at least two elements. Given $\mfs\in \mfI$,
we define
$$V_{\mfs}\defeq \left(\bigcap\limits_{i\in \mfs} U_i\right)\setminus\left(\bigcup\limits_{i\notin\mfs}\supp(\xi_i)\right).$$
Then $\V=\{V_\mfs\}_{\mfs\in\mfI}$ is an open cover of the compact subset $K=\cup_{i=1}^n\supp(\xi_i)\sbe\G$.
Let $\{\psi_\mfs\}_{\mfs\in \mfI}$ be a partition of unit subordinate to the cover $\V$.
Note that $\psi_\mfs\xi_i=0$ whenever $i\notin\mfs$, so that
$$\sum\limits_{i\in\mfs}\psi_\mfs\xi_i=\sum\limits_{i=1}^n\psi_\mfs\xi_i=\psi_\mfs\cdot 0=0.$$
Moreover, we have
\begin{align*}
\sum\limits_{i=1}^n\xi_i\delta_{U_i} & =\sum\limits_{i=1}^n\left(\sum\limits_{\mfs\in \mfI}\psi_\mfs\xi_i\delta_{U_i}\right)\\
                                   & =\sum\limits_{\mfs\in \mfI}\sum\limits_{i=1}^n\psi_\mfs\xi_i\delta_{U_i}
                                     =\sum\limits_{\mfs\in \mfI}\left(\sum\limits_{i\in \mfs}\psi_\mfs\xi_i\delta_{U_i}\right).
\end{align*}
Thus, it suffices to show that $\sum\limits_{i\in \mfs}\psi_\mfs\xi_i\delta_{U_i}\in \I$ for all $\mfs\in \mfI$.
Now, given distinct elements $j,k\in \mfs$, we have
\begin{align*}
\sum\limits_{i\in \mfs}\psi_\mfs\xi_i\delta_{U_i}= &\,\psi_\mfs\xi_j\delta_{U_j}+\psi_\mfs\xi_k\delta_{U_k}+
                                                                        \sum\limits_{i\in \mfs\setminus\{j,k\}}\psi_\mfs\xi_i\delta_{U_i}\\
                   =&\, \psi_\mfs\xi_j\delta_{U_j}+\psi_\mfs\xi_k\delta_{U_k}+
                                                                        \sum\limits_{i\in \mfs\setminus\{j,k\}}\psi_\mfs\xi_i\delta_{U_i}\\
                   & + \sum\limits_{i\in \mfs\setminus\{j,k\}}\psi_\mfs\xi_i\delta_{U_k} -
                                                                \sum\limits_{i\in \mfs\setminus\{j,k\}}\psi_\mfs\xi_i\delta_{U_k}\\
                   =&\, \psi_\mfs\xi_j\delta_{U_j}+\sum\limits_{i\in \mfs\setminus\{j\}}\psi_\mfs\xi_i\delta_{U_k}
                    +\sum\limits_{i\in \mfs\setminus\{j,k\}}\bigl(\psi_\mfs\xi_i\delta_{U_i}-\psi_\mfs\xi_i\delta_{U_k}\bigr)\\
                   =&\,  \psi_\mfs\xi_j\delta_{U_j}-\psi_\mfs\xi_j\delta_{U_k}
                    +\sum\limits_{i\in \mfs\setminus\{j,k\}}\bigl(\psi_\mfs\xi_i\delta_{U_i}-\psi_\mfs\xi_i\delta_{U_k}\bigr),
\end{align*}
where the last equality follows from $\psi_\mfs\xi_j+\sum\limits_{i\in \mfI\setminus\{j\}}\psi_\mfs\xi_i=\sum\limits_{i\in \mfI}\psi_\mfs\xi_i=0$.
Since $\supp(\psi_\mfs\xi_i)\sbe U_l$ for all $i,l\in \mfs$, we get $\sum\limits_{i\in \mfs}\psi_\mfs\xi_i\delta_{U_i}\in \I$ as desired.
Next, we show that $\I=\I_\A^0$. Of course, $\I_\A^0\sbe\ker(\Psi)=\I$. On the other hand, suppose $\xi\in \contc(\B)$ and $\supp(\xi)\sbe U\cap V$
with $U,V\in S$. Since $S$ is wide and $\supp(\xi)$ is compact, there are $W_1,\ldots,W_m\in S$ such that
$$\supp(\xi)\sbe\bigcup\limits_{i=1}^m W_i\sbe U\cap V.$$
Let $\{\phi_i\}_{i=1}^m$ be a partition of unit subordinate to the cover $\{W_1,\ldots,W_m\}$ of $\supp(\xi)$.
Then
\begin{align*}
\xi\delta_U-\xi\delta_V & =  \sum\limits_{i=1}^m\phi_i\xi\delta_U-\sum\limits_{i=1}^m\phi_i\xi\delta_V\\
                        & =  \sum\limits_{i=1}^m\phi_i\xi\delta_U-\sum\limits_{i=1}^m\phi_i\xi\delta_{W_i}
                        +\sum\limits_{i=1}^m\phi_i\xi\delta_{W_i}-\sum\limits_{i=1}^m\phi_i\xi\delta_V\\
& =  \sum\limits_{i=1}^m(\phi_i\xi\delta_U-\phi_i\xi\delta_{W_i})
                        +\sum\limits_{i=1}^m(\phi_i\xi\delta_{W_i}-\phi_i\xi\delta_V)\\
\end{align*}
is an element of $\I_\A^0$ because $\supp(\phi_i\xi)\sbe W_i\cap U\cap V$ for all $i=1,\ldots,m$.
We conclude that $\ker(\Psi)=\I_\A^0$. Hence $\Psi$ induces an isomorphism $\contc(\A^0)/\I_\A^0\cong\contc(\B)$ of \Star{algebras}
and therefore also between their corresponding enveloping \cstar{algebras} $C^*(\A)\cong C^*(\B)$.

\pf{Case 2: $\G$ is second countable and $\contz(\B_U)$ is separable for all $U\in S$}
In this case, we shall adapt an idea appearing in \cite[Lemma 8.4]{Exel:inverse.semigroups.comb.C-algebras} to our situation, supported
by \cite[Proposition 3.3]{Muhly.Williams.Equivalence.and.Disintegration}, which is a kind of disintegration result for linear functionals
on Banach bundles.
To show that $\tilde\Psi\colon C^*(\A)\to C^*(\B)$ is injective, it is enough to prove that any representation $\pi$ of
$C^*(\A)$ factors through a representation $\tilde\pi$ of $C^*(\B)$ in the sense that $\tilde\pi\circ\tilde\Psi=\pi$.
By Lemma~\ref{lem:algebraic Fell bundle with same enveloping algebra}, $C^*(\A)$ is the enveloping \cstar{algebra} of $\contc(\A^0)/\I_\A^0$,
so it is enough to work with representations of the latter.
Now, every representation $\pi$ of $\contc(\A^0)/\I_\A^0$ on a Hilbert space $\hils$,
once composed with the quotient homomorphism $\contc(\A^0)\to\contc(\A^0)/\I_\A^0$ and hence viewed as a representation of $\contc(\A^0)$,
has the form $$\pi\left(\sum\limits_{U\in F}\xi_U\delta_U\right)=\sum\limits_{U\in F}\pi_U(\xi_U)$$
where $\{\pi_U\}_{U\in S}$ is a representation of $\A^0=\{\contc(\B_U)\}_{U\in S}$ on $\hils$. In fact, $\pi_U$ is just the restriction
of $\pi$ to the copy of $\A_U^0=\contc(\B_U)$ inside $\contc(\A^0)$. Given $\eta,\zeta\in \hils$, we define the linear functional
$$\omega_U\colon\contc(\B_U)\to\C,\quad \omega_U(\xi)\defeq \braket{\pi_U(\xi)\eta}{\zeta}\mbox{ for all }\xi\in \contc(\B_U).$$
Since $\pi_U$ is norm contractive, that is, $\|\pi_U(\xi)\|\leq\|\xi\|_\infty$
(see proof of Lemma~\ref{lem:algebraic Fell bundle with same enveloping algebra}), it follows that $\omega_U$ is continuous with
respect to the inductive limit topology. In other words, $\omega_U$ is a \emph{generalized Radon measure} in the sense of
\cite{Muhly.Williams.Equivalence.and.Disintegration}. By Proposition 3.3 in \cite{Muhly.Williams.Equivalence.and.Disintegration},
there are bounded linear functionals $\ep_{U,\g}\in \B_\g^*$ with norm at most one for each $\g\in U$,
such that $\g\mapsto \ep_{U,\g}(\xi(\g))$ is a bounded measurable function on $U$ for all $\xi\in \contc(\B_U)$, and
$$\omega_U(\xi)=\int_U\ep_{U,\g}(\xi(\g))\dd|\omega_U|(\g),$$
where $|\omega_U|$ is the total variation measure associated to $\omega_U$ as
in \cite[Lemma 3.1]{Muhly.Williams.Equivalence.and.Disintegration}:
\begin{equation}\label{eq:definition of total variation measure}
|\omega_U|(\varphi)\defeq \sup\{|\omega_U(\xi)\colon\|\xi\|_\infty\leq \varphi\}\quad\mbox{for all }\varphi\in \contc^+(U).
\end{equation}
As we have emphasize in the notation, all these objects depend \emph{a priori} on the bissection $U\in S$. However,
we are now going to show that they are compatible on intersections of bissections. First, given $U,V\in S$,
we claim that
\begin{equation}\label{eq:functional do not depend on the bissections}
\omega_U(\xi)=\omega_V(\xi)\quad\mbox{whenever }\xi\in \contc(\B)\mbox{ and }\supp(\xi)\sbe U\cap V.
\end{equation}
In fact, since $S$ is wide and $\supp(\xi)$ is compact, we may find a cover $\{W_i\}_{i=1}^m$ of $\supp(\xi)$ consisting of bissections
$W_i\in S$ contained in $U\cap V$. Since $U\cap V$ is Hausdorff, there is a partition of unit $\{\psi_i\}_{i=1}^m$ subordinate to the cover $\{W_i\}_{i=1}^m$ of $\supp(\xi)$. Since $\supp(\psi_i\cdot \xi)\sbe W_i\sbe U\cap V$ and $\pi$ is a representation, we have
$$\pi_U(\psi_i\cdot\xi)=\pi_{W_i}(\psi_i\cdot\xi)=\pi_V(\psi\cdot\xi)\quad\mbox{for all }i=1,\ldots,m.$$
Thus,
\begin{multline*}
\pi_U(\xi)=\pi_U\left(\sum\limits_{i=1}^m\psi_i\cdot \xi\right)=\sum\limits_{i=1}^m\pi_U(\psi_i\cdot \xi)
\\=\sum\limits_{i=1}^m\pi_V(\psi_i\cdot \xi)=\pi_V\left(\sum\limits_{i=1}^m\psi_i\cdot \xi\right)=\pi_V(\xi).
\end{multline*}
Therefore, $\omega_U(\xi)=\omega_V(\xi)$ as claimed. It follows directly from~\eqref{eq:definition of total variation measure} that
$|\omega_U|(\varphi)=|\omega_V|(\varphi)$ whenever $\supp(\varphi)\sbe U\cap V$ or, equivalently, $|\omega_U|(A)=|\omega_V|(A)$
whenever $A$ is a Borel measurable subset of $U\cap V$. Since $S$ countably covers $\G$,
there is a positive Radon measure $\mu$ on $\G$ whose restriction to $U$ equals $|\omega_U|$ for all $U\in S$
(see \cite[Lemma A.1]{Muhly.Williams.Renault's.Equivalence.Theorem}).
Moreover, following the construction of $\ep_{U,\g}$ in \cite[Proposition 3.3]{Muhly.Williams.Equivalence.and.Disintegration}
and using Equation~\eqref{eq:functional do not depend on the bissections},
it follows that $\ep_{U,\g}=\ep_{V,\g}$ whenever $\g\in U\cap V$. Therefore, we may define a linear functional
$$\omega\colon\contc(\B)\to \C,\quad \omega(\xi)\defeq \int_\G\ep_\g(\xi(\g))\dd{\mu}(\g),$$
where $\ep_\g\defeq \ep_{U,\g}$ whenever $\g\in U$. By definition, the restriction of $\omega$ to $\contc(\B_U)$ coincides with $\omega_U$
for all $U\in S$.

All this implies that the map
$$\tilde\pi\colon\contc(\B)\to \bound(\hils),\quad \tilde\pi\left(\sum\limits_{U\in F}\xi_U\right)\defeq \sum\limits_{U\in F}\pi_U(\xi_U)
                                                                                            =\pi\left(\sum\limits_{U\in F}\xi_U\delta_U\right)$$
is well-defined. In fact, if $\sum\limits_{U\in F}\xi_U=0$, then
\begin{multline*}
\braket{\sum\limits_{U\in F}\pi_U(\xi_U)\eta}{\zeta}=\sum\limits_{U\in F}\omega_U(\xi_U)
=\sum\limits_{U\in F}\omega(\xi_U)=\omega\left(\sum\limits_{U\in F}\xi_U\right)=0.
\end{multline*}
Since $\eta,\zeta\in \hils$ are arbitrary, we conclude that $\sum\limits_{U\in F}\pi_U(\xi_U)=0$.
Therefore $\tilde\pi$ is a well-defined map. Obviously it is linear, and since $\pi$ is a representation of $\contc(\A^0)$,
it is easy to see that $\tilde\pi$ is a representation of $\contc(\B)$. Moreover, by construction we have
$\tilde\pi(\Psi(x))=\pi(x)$ for all $x\in \contc(\A^0)$, and this concludes the proof.
\end{proof}

\subsection{Twisted étale groupoids and Fell line bundles}\label{sec:twisted groupoids and Fell line bundles}

If $\G$ is a locally compact groupoid it is well known
\cite{Deaconi_Kumjian_Ramazan:Fell.Bundles} that,  at least in the
Hausdorff case,  there is a one-to-one
correspondence between \emph{twists} over $\G$ (namely exact sequences
  \[
  \xymatrix{\Torus\times \Gz  \ar[r]^{\iota} & \Sigma \ar[r]^{\pi} & \G},
  \]
where $\Sigma$ is a locally compact groupoid, $\iota$ is a
homeomorphism onto its image, and $\pi$ is a continuous open
surjection) and Fell line bundles (namely Fell bundles with
one-dimensional fibers) over $\G$.  In particular, the so called
full (resp. reduced) twisted groupoid \cstar{algebra} of $(\G, \Sigma)$ turns out to be
precisely the full (resp. reduced) cross-sectional \cstar{algebra} of the corresponding Fell
line bundle.

Our techniques are specially well adapted to deal with Fell bundles
and hence we have decided to emphasize these as opposed to twists.  Of
course, should one be interested in the underlying twist, it is
readily available by considering unitary elements as explained in
\cite{Deaconi_Kumjian_Ramazan:Fell.Bundles}.

Now consider a Fell line bundle $L$ over an étale locally compact
groupoid $\G$.
As a special case of
Example~\ref{ex:Fell bundle over semigroups associated to Fell bundles over groupoids},
we may consider the Fell bundle $\A=\{\contz(L_U)\}_{U\in S}$ associated to $L$, where $S$ is an inverse subsemigroup of $S(\G)$.

\begin{proposition}\label{prop:C*-algebra of L and A isomorphic}
Let notation be as above. If $\G$ is Hausdorff or second countable, and if
$S$ is a wide inverse subsemigroup in $S(\G)$, then there is a canonical isomorphism $C^*(L)\cong C^*(\A)$.
\end{proposition}
\begin{proof}
This is a direct consequence of Theorem~\ref{teor:fell bundles over groupoids and ISG isomorphic}.
\end{proof}

\section{Semi-abelian Fell bundles and twisted étale groupoids}
\label{sec:Semi-abelian Fell bundles and twisted étale groupoids}

This section contains the main result of this work. Given a semi-abelian Fell bundle over an inverse semigroup,
we are going to construct a twisted étale groupoid in a canonical way. Later we are going to show that our construction
preserves the associated \cstar{}algebras. Our techniques are inspired by those of Kumjian and Renault
in \cite{Kumjian:cstar.diagonals,RenaultCartan}.

\subsection{The canonical action and the groupoid of germs}\label{sec:the construction of G}

\begin{definition}
Let $\A=\{\A_s\}_{s\in S}$ be a Fell bundle over an inverse semigroup $S$. We say that $\A$ is \emph{semi-abelian} if
for each idempotent $e\in E=E(S)$, the fiber $\A_e$ is a commutative \cstar{algebra}.
\end{definition}

Let $\E$ be the restriction of $\A$ to the idempotent semilattice $E$
of $S$. Note that $\E$ is a Fell bundle over $E$.
Moreover, $\A$ is semi-abelian if and only if $\E$ is an \emph{abelian} Fell bundle
in the sense that $ab=ba$ for all $a,b\in \E$ (see proof of Lemma~\ref{lem:FellBundleSemiAbelianIFFC*(E)Commutative} below).
Of course, we could also say that an arbitrary Fell bundle $\A=\{\A_s\}_{s\in S}$ is \emph{abelian} if $ab=ba$ for all $a,b\in \A$,
but we actually are not going to consider this more restrictive notion in this work.\footnote{Note that in this case the
underlying inverse semigroup $S$ has to be commutative. Moreover, it is easy to see that $\A$
is an abelian Fell bundle if and only if $C^*(\A)$ is a commutative \cstar{}algebra.}
This terminology is compatible with the one appearing in \cite[Chapter X]{fell_doran},
where the authors use the synonymous word "commutative" instead. However, we should mention
that our terminology conflicts with the one introduced in \cite{Deaconi_Kumjian_Ramazan:Fell.Bundles},
where they call a Fell bundle $\B=\{\B_\g\}_{\g\in \G}$ over a groupoid
$\G$ \emph{abelian} if the fibers $\B_x$ over the units $x\in \Gz$ are commutative \cstar{}algebras.

\begin{lemma}\label{lem:FellBundleSemiAbelianIFFC*(E)Commutative}
A Fell bundle $\A$ is semi-abelian if and only if $C^*(\E)$ is a commutative \cstar{algebra}.
\end{lemma}
\begin{proof}
If $C^*(\E)$ is commutative, then so is each $\A_e$ since $\A_e$ is an
ideal of $C^*(\E)$ by \cite[Corollary 4.6]{Exel:noncomm.cartan}. Conversely, assume
that $\A_e$ is commutative for all $e\in E$. To prove that $C^*(\E)$
is commutative it suffices to show that $ab=ba$ for every $a\in \A_e$
and $b\in \A_f$, where $e,f\in E$. Let $(e_i)$ be an approximate unit
of $\A_e$. Since $\A_e$ is an ideal of $C^*(\E)$ we have
$ab,e_ib,ba\in \A_e$.  Using that $\A_e$ is commutative, we get the result
$$ab=\lim\limits_{i}e_iab=\lim\limits_i ae_ib=\lim\limits_i e_iba=ba.$$
\vskip-14pt
\end{proof}

Let us assume from now on that $\A$ is a semi-abelian saturated Fell bundle.
Suppose that $X$ is the spectrum of $C^*(\E)$ so that (we may identify) $C^*(\E)=\contz(X)$.

\begin{lemma}\label{lem:construction of theta_a}
Given $a\in \A$, we write $\dom(a)\defeq \{x\in X\colon (a^*a)(x)>0\}$ and $\ran(a)\defeq \{x\in X\colon (aa^*)(x)>0\}=\dom(a^*)$.
Then there is a unique homeomorphism $\t_a\colon\dom(a)\to\ran(a)$ satisfying
$$(a^*ba)(x)=(a^*a)(x)b(\t_a(x))\quad\mbox{for all }x\in \dom(a)\mbox{ and } b\in \contz(X).$$
If $a=u|a|$ is the polar decomposition of $a$ in $A''$, the enveloping von Neumann algebra of $A=C^*(\A)$,
where we view each fiber of $\A$ as a subspace of $A$,
then $$(u^*bu)(x)=b(\t_a(x))\quad\mbox{for all }x\in \dom(a)\mbox{ and }b\in\contz(X).$$
Moreover, the following properties hold\textup:
\begin{enumerate}
\item[\textup{(i)}] If $a\in \A_e$, where $e\in E$, then $\t_a=\id_{\dom(a)}$.
\item[\textup{(ii)}] If $a,b\in \A$, then $\t_{ab}=\t_a\circ\t_b$ \textup(composition of partial homeomorphisms\textup).
\item[\textup{(ii)}] If $a\in\A$, then $\t_{a^*}=\t_a\inv$.
\end{enumerate}
\end{lemma}
\begin{proof}
Essentially this follows from Proposition 6 and Corollary 7 in \cite{Kumjian:cstar.diagonals}.
In fact, it is enough to observe that each fiber $\A_s\sbe A$ is contained in the \emph{normalizer} of $B=C^*(\E)$ in $A$. Recall that
$a\in A$ normalizes $B$ if $a^*Ba\sbe B$ and $aBa^*\sbe B$.
\end{proof}

\begin{lemma}\label{lem:theta_s well-defined}
Suppose that $s\in S$, $a_1,a_2\in \A_s$ and $x\in \dom(a_1)\cap\dom(a_2)$. Then we have $\t_{a_1}(x)=\t_{a_2}(x)$.
\end{lemma}
\begin{proof}
Since $a_1^*a_2\in \A_{s^*s}$, Lemma~\ref{lem:construction of theta_a} yields
$$\t_{a_1}\inv\circ\t_{a_2}=\t_{a_1^*}\circ\t_{a_2}=\t_{a_1^*a_2}=\id_{\dom(a_1^*a_2)}.$$
So, it is enough to show that $x\in \dom(a_1^*a_2)$, that is,
$$((a_1^*a_2)^*(a_1^*a_2))(x)=(a_2^*a_1a_1^*a_2)(x)>0.$$
It is equivalent to show that $(a_2^*a_2)(x)(a_2^*a_1a_1^*a_2)(x)>0$ because $x\in \dom(a_2)$.
Since the fibers over the idempotents commute with each other, we have
$$(a_2^*a_2)(a_2^*a_1a_1^*a_2)=a_2^*a_1a_2^*a_2a_1^*a_2=(a_2^*a_2)(a_1^*a_1)(a_2^*a_2).$$
The result now follows.
\end{proof}

Let $e\in E$. Since $\A_e$ is an ideal in $C^*(\E)$, there is an open subset $\U_e\sbe X$ such that $\A_e=\contz(\U_e)$.

\begin{proposition}\label{prop:theta is an action of S}
Given $s\in S$, there is a homeomorphism $\t_s\colon\U_{s^*s}\to \U_{ss^*}$ such that $\t_s|_{\dom(a)}=\t_a$ for all $a\in \A_s$.
Moreover, we have $\t_s\circ\t_t=\t_{st}$ for all $s,t\in S$. In other words, $\t$ is an action of $S$ on $X$.
\end{proposition}
\begin{proof}
Given $x\in \U_{s^*s}$, we define $\t_s(x)\defeq \t_a(x)$, where $a$ is any element in $\A_s$ with $(a^*a)(x)>0$.
Note that such an element exists because $\A$ is saturated. By Lemma~\ref{lem:theta_s well-defined}, $\t_s$ is
a well-defined map $\U_{s^*s}\to \U_{ss^*}$ and, by definition,
the restriction of $\t_s$ to $\dom(a)$ is equal to $\t_a$. Since each
$\t_a$ is a homeomorphism, we deduce that $\t_s$ is continuous.

It remains to prove that $\t_s\circ\t_t=\t_{st}$. Now, since $\t_a\circ\t_b=\t_{ab}$ for all $a\in \A_s$ and $b\in \A_t$,
it is enough to show that the domains of $\t_s\circ\t_t$ and $\t_{st}$ coincide. If $x\in \dom(\t_s\circ\t_t)$,
that is, if $x\in \U_{t^*t}=\dom(\t_t)$ and $\t_t(x)\in \U_{s^*s}=\dom(\t_s)$, then there is $a\in \A_s$ and $b\in \A_t$ with
$x\in \dom(b)$ and $\t_b(x)\in \dom(a)$, that is,
$$x\in \dom(\t_a\circ\t_b)=\dom(\t_{ab})=\dom(ab)\sbe\U_{(st)^*(st)}=\dom(\t_{st}).$$
Conversely, if $x\in \dom(\t_{st})$, there is $c\in \A_{st}$ such that $x\in \dom(c)$, that is, $(c^*c)(x)>0$.
We claim that there is $a\in \A_s$ and $b\in \A_t$ with $x\in \dom(ab)$. In fact, suppose this is not the case, so that $(ab)^*(ab)(x)=0$
for all $a\in \A_s$ and $b\in \A_t$. Polarization in $a$ and $b$ yields
$$(a_1b_1)^*(a_2b_2)(x)=0$$
for all $a_1,a_2\in \A_s$ and $b_1,b_2\in \A_t$. This implies that
$(c_1^*c_2)(x)=0$ for all $c_1, c_2\in \cspn\A_s\A_t=\A_{st}$, which is a
contradiction. This proves our claim. Therefore, there is $a\in\A_s$
and $b\in \A_t$ with $$x\in
\dom(ab)=\dom(\t_{ab})=\dom(\t_a\circ\t_b)\sbe\dom(\t_s\circ\t_t).$$
\vskip-15pt \end{proof}

We are ready to define the groupoid $\G=\G(\A)$ associated to the semi-abelian Fell bundle $\A$.
Indeed, we define the groupoid $\G$ to be the groupoid of germs of the action $\t$ of $S$ on $X$ constructed above:
$$\G\defeq \{\germ{s}{x}\colon s\in S,\; x\in \dom(\t_s)=\U_{s^*s}\},$$
where, by definition, $\germ{s}{x}=\germ{t}{y}$ if and only if $x=y$ and there is $e\in E$ with $x\in \U_e$ and $se=te$.
The source and range maps are defined by $\s(\germ{s}{x})=x$ and $\r(\germ{s}{x})=\theta_s(x)$, respectively.
Multiplication and inversion are given by
$$\germ{s}{x}\cdot \germ{t}{y}\defeq [st,y]\quad\mbox{whenever }\t_t(y)=x,\mbox{ and } \germ{s}{x}\inv\defeq \germ{s^*}{\t_s(x)}.$$
The topology on $\G$ is, by definition, generated by the basic open sets
\begin{equation}\label{eq:open sets in G}
\O(s,U)\defeq \{\germ{s}{x}\colon x\in U\}
\end{equation}
where $s\in S$ and $U\sbe \U_{s^*s}$ is an open subset.
In particular, $\O_s\defeq \O(s,\U_{s^*s})$ is an open subset of $\G$. Moreover, the restriction of $\s$ defines a homeomorphism
$\s_s\colon\O_s\to \U_{s^*s}$.

See \cite[Section~4]{Exel:inverse.semigroups.comb.C-algebras} for more
details on the construction of groupoids of germs, but please notice
that our notion of germs differs from the one described in
\cite[Section~3]{RenaultCartan}.

With this structure, $\G$ is an étale groupoid and
the unit space of $\G$ may be identified with $X$ through the map
$$X\ni x\mapsto \germ{e_x}{x}\in \Gz$$
where $e_x\in E(S)$ is any idempotent with $x\in \U_{e_x}$.

\subsection{The construction of the Fell line bundle}\label{sec:construction of L}

Let $A$ and $B$ be \cstar{algebras}, and
let $_A\X_B$ be an imprimitivity Hilbert $A$-$B$-bimodule.
If $I$ is a closed  ideal of $A$, then $\F(I)=I\cdot \X$ is a closed submodule of $\X$, and $\Ind_\X(I)=\cspn\braket{\F(I)}{\F(I)}_B$
is a closed ideal of $B$. Moreover, the maps $I\mapsto \F(I)$ and $I\mapsto \Ind_\X(I)$ define bijective correspondences between closed
ideals of $A$, closed submodules of $\X$ and closed ideals of $B$. Given a closed ideal $J$ in $B$, the corresponding submodule is
$\X\cdot J$, and the corresponding ideal in $A$ is $\cspn_A\!\braket{\X\cdot J}{\X\cdot J}$. This fact is known as the Rieffel correspondence.

If $I$ is a closed ideal of $A$, then using an approximate unit for $I$, it is easy to see that
\begin{equation}\label{Eq:Charact. of IX}
\F(I)=I\cdot \X=\{\xi\in \X\colon {_A}\!\braket{\xi}{\xi}\in I\}.
\end{equation}
And similarly if $J$ is a closed ideal of $B$, then
\begin{equation}\label{Eq:Charact. of XJ}
\X\cdot J=\{\xi\in \X\colon\braket{\xi}{\xi}_B\in J\}.
\end{equation}

Given a saturated Fell bundle $\A=\{\A_s\}_{s\in S}$ over an inverse semigroup $S$, note that each fiber $\A_s$ is an imprimitivity Hilbert
$\A_{ss^*}$-$\A_{s^*s}$-bimodule in the canonical way. For instance, the inner products are defined by
$$\braket{a}{b}_{\A_{s^*s}}\defeq a^*b,\quad _{\A_{ss^*}}\!\braket{a}{b}\defeq ab^*\quad \mbox{for all }a,b\in \A_s.$$

Now assume that $\A=\{\A_s\}_{s\in S}$ is a fixed saturated, semi-abelian Fell bundle. Let $X$ be the spectrum of the commutative \cstar{algebra}
$C^*(\E_\A)$. Recall that $\U_e$ denotes the open subset of $X$ corresponding to the spectrum of the ideal $\A_e$ in $C^*(\E_\A)\cong \contz(X)$.

\begin{definition}
Given $s\in S$ and $x\in X$, we define $\A_{(s,x)}$ to be the closed submodule of $\A_s$ corresponding to the ideal
$\{b\in \A_{s^*s}\colon b(x)=0\}$ in  $\A_{s^*s}$ under the Rieffel correspondence.
\end{definition}

\begin{lemma}\label{L:properties of A_(s,x)}
With the notations above, we have the following properties for all $s,t\in S$, $a\in \A_s$ and $x,y\in X$\textup:
\begin{enumerate}
\item[\textup{(i)}] $\A_{(s,x)}=\A_{(ss^*,\t_s(x))}\cdot\A_s=\A_s\cdot \A_{(s^*s,x)}$ whenever $x\in \U_{s^*s}$.\\
Moreover, if $x\notin \U_{s^*s}$, then $\A_{(s,x)}=\A_s$;
\item[\textup{(ii)}] $a\in \A_{(s,x)} \Leftrightarrow (a^*a)(x)=0 \Leftrightarrow (aa^*)(\t_s(x))=0$;
\item[\textup{(iii)}] $\A_s\cdot\A_{(t,y)}=\A_{(st,y)}$ and, if $x\in \U_{tt^*}$, then $\A_{(s,x)}\A_t=\A_{(st,\theta_{t^*}(x))}$;
\item[\textup{(iv)}] $\A_{(s,x)}\cdot\A_{(t,y)}=\A_{(st,y)}$ whenever $x\in \U_{s^*s}$, $y\in \U_{t^*t}$ and $\theta_t(y)=x$;
\item[\textup{(v)}] $\A_{(s,x)}^*=\A_{(s^*,\theta_s(x))}$ whenever $x\in \U_{s^*s}$.
\end{enumerate}
\end{lemma}
\begin{proof}
Note that $\A_{(s^*s,x)}=\{b\in \A_{s^*s}\colon b(x)=0\}$. Thus, by definition, $\A_{(s,x)}=\A_s\cdot\A_{(s^*s,x)}$.
Now, by Equation~\eqref{Eq:Charact. of XJ}, we have
$$\A_{(s,x)}=\{a\in \A_s\colon (a^*a)(x)=0\}.$$
We have $(aa^*)(\theta_s(x))=(a^*a)(x)$ for all $a\in \A_s$ and $x\in \U_{s^*s}$.
Thus, if $x\in\U_{s^*s}$, then
$$\A_{(s,x)}=\{a\in \A_s\colon (aa^*)(\theta_s(x))=0\}.$$
Again, by Equation~\eqref{Eq:Charact. of IX} this is equal to $\A_{(ss^*,\theta_s(x))}\cdot\A_s$ because
$$\A_{(ss^*,\theta_s(x))}=\{b\in \A_{ss^*}\colon b(\theta_s(x))=0\}.$$
If $x\notin\U_{s^*s}$, then $\A_{(s^*s,x)}=\A_{s^*s}$ because, by definition, $\U_{s^*s}$ is the spectrum of $\A_{s^*s}$.
This proves (i) and (ii). To prove (iii), we use (i) to conclude that
$$\A_s\cdot\A_{(t,y)}=\A_s\cdot\A_t\cdot\A_{(t^*t,y)}=\A_{st}\cdot\A_{(t^*t,y)}=\A_{st}\cdot\A_{t^*s^*st}\cdot\A_{(t^*t,y)}$$
and
$$\A_{(st,y)}=\A_{st}\cdot\A_{(t^*s^*st,y)}.$$
Since $t^*s^*st\leq t^*t$, we have $\A_{t^*s^*st}\sbe\A_{t^*t}$ and hence also $\A_{(t^*s^*st,y)}\sbe\A_{(t^*t,y)}$.
Thus, $\A_{(st,y)}\sbe \A_s\cdot \A_{(t,y)}$. On the other hand, both $\A_{t^*s^*st}$ and $\A_{(t^*t,y)}$
are ideals in $\contz(X)$, so that
$$\A_{t^*s^*st}\cdot\A_{(t^*t,y)}=\A_{t^*s^*st}\cap\A_{(t^*t,y)}=\{b\in \A_{t^*s^*st}\colon b(y)=0\}=\A_{(t^*s^*st,y)}.$$
This concludes the proof of the first assertion in (iii). To prove the second assertion in (iii) we use a similar argument.
First, if $x\notin\U_{s^*s}$, then $\theta_{t^*}(x)\notin\theta_{t^*}(\U_{s^*s}\cap\U_{tt^*})=\U_{t^*s^*st}$
(see \cite[Proposition 4.5]{Exel:inverse.semigroups.comb.C-algebras}), so that $\A_{(s,x)}\A_t=\A_s\A_t=\A_{st}=\A_{(st,\theta_{t^*}(x))}$
by (i). Now assume that $x\in \U_{s^*s}$. Using (i) again, we get
$$\A_{(s,x)}\cdot\A_t=\A_{(ss^*,\theta_s(x))}\cdot\A_{st}=\A_{(ss^*,\theta_s(x))}\cdot \A_{stt^*s^*}\cdot\A_{st}$$
and, on the other hand,
$$\A_{(st,\theta_t^*(x))}=\A_{(stt^*s^*,\theta_{st}(\theta_{t^*}(x)))}\cdot\A_{st}=\A_{(stt^*s^*,\theta_s(x))}\cdot\A_{st}.$$
As above, it is easy to see that $\A_{(stt^*s^*,\theta_s(x))}=\A_{(ss^*,\theta_s(x))}\cdot \A_{stt^*s^*}$.
To prove (iv), we use (i), (iii) and the hypothesis $\theta_t(y)=x$, to get
\begin{align*}
\A_{(s,x)}\cdot\A_{(t,y)}&\overeq{(i)}\A_{(ss^*,\theta_s(x))}\cdot\A_{st}\cdot\A_{(t^*t,y)}\\
               &\!\!\!\!\!\!\overeq{\theta_t(y)=x}\A_{(ss^*,\theta_{st}(y))}\cdot \A_{stt^*s^*}\cdot \A_{st}\cdot \A_{t^*s^*st}\cdot \A_{(t^*t,y)}\\
                        &\!\!\overeq{(iii)}\A_{(stt^*s^*,\theta_{st}(y))}\cdot\A_{st}\cdot \A_{(t^*s^*st,y)}\overeq{(i)}\A_{(st,y)}
\end{align*}
Finally, to prove (v), we use (i) and conclude that
\begin{multline*}
\A_{(s^*,\theta_s(x))}=\A_{(s^*s,\theta_{s^*}(\theta_s(x)))}\cdot\A_{s^*}=\A_{(s^*s,x)}\cdot\A_{s^*}
=\bigl(\A_{s}\cdot\A_{(s^*s,x)}\bigr)^*=\A_{(s,x)}^*.
\end{multline*}
\vskip-13pt
\end{proof}

Having fixed our semi-abelian Fell bundle $\A = \{\A_s\}_{s\in
S}$ above,  we henceforth let $\G$ be the associated groupoid of germs constructed
in Section~\ref{sec:the construction of G}.

\begin{definition}\label{D:definition a eqx b}
  Let $s\in S$, and $a, b\in \A_s$.  Given $x\in X$, we shall say that
  $$ a\eq x b,$$
  if $\big((a-b)^*(a-b)\big)(x)=0$. Hence $a\eq x b$ if and only if $a-b\in \A_{(s,x)}$.
\end{definition}

In the following we present some elementary properties of the relation defined above.

\begin{lemma}\label{L:Equivinh}
  Let $s \in S$, let $a, b, c\in \A_s$, and let $x\in X$.
\begin{enumerate}
\item[\textup{(i)}]\label{ZeroCase} If $(a^*a)(x)=0$ \textup(in particular if
    $x\notin \U_{s^*s}$\textup), then $a\eq x 0_s$ \textup(where $``0_s\kern-4pt"$
    stands for the zero element of $\A_s$\textup).
\item[\textup{(ii)}]\label{LeftInv} If $a\eq x b$, then for every $t\in S$,
    and $c\in \A_t$, one has $ca\eq xcb$.
\item[\textup{(iii)}]\label{RightInvIdemp} If $a\eq x b$, then for every $e\in E(S)$,
    and every $c\in \A_e$, one has $ac\eq x bc$.
\item[\textup{(iv)}]\label{RightInv} If $a\eq x b$, then for every $t\in S$
    such that $x\in \U_{tt^*}$, and every $c\in \A_t$, one has $ac\eq {\t_t\inv(x)} bc.$
\item[\textup{(v)}]\label{Trans} If $a\eq x b$ and $b\eq xc$, then $a\eq xc$.
\item[\textup{(vi)}] If $a\eq x b$, then $a^*\eq{\t_s(x)}b^*$.
\end{enumerate}
\end{lemma}
\begin{proof} Item (i) is  obvious.  To prove (ii), suppose $a\eq x b$. Then $a-b\in \A_{(s,x)}$, so that $ca-cb=c(a-b)\in\A_t\A_{(s,x)}=\A_{(ts,x)}$
  by Lemma~\ref{L:properties of A_(s,x)}(iii). Thus $ca\eq x cb$.

To prove (iii), we compute
$$
  \big((ac-bc)^* (ac-bc)\big)(x) =
  \overline{c(x)}\;\big((a-b)^* (a-b)\big)(x)\; c(x) = 0.
$$
  It is also possible to prove (iii) by showing that $\A_{(s,x)}\cdot\A_e=\A_{(se,x)}$. Note that this follows from
  Lemma~\ref{L:properties of A_(s,x)}(iii) if $x\in \U_e$.
Item (iv) follows from Lemma~\ref{L:properties of A_(s,x)}(iii).
Finally, if $a-b\in \A_{(s,x)}$ and $b-c\in \A_{(s,x)}$, then $a-c=(a-b)+(b-c)\in \A_{(s,x)}$. This proves (v).
Finally, (vi) follows from the relation $(cc^*)(\t_s(x))=(c^*c)(x)$ applied to $c=a-b\in \A_s$.
\end{proof}

\begin{definition} \label{DefineEquivRel}
Consider the set $\F$ of triples
$\trip asx$ such that $a\in \A_s$, and $x\in \U_{s^*s}$.  If one is
given $\trip asx, \trip {a'}{s'}{x'}\in \F$, we will say that
  $$
  \trip asx\sim \trip {a'}{s'}{x'}
  $$
  if there exists some $e\in E(S)$,  and some $b\in \A_e$,
with
\begin{enumerate}
\item[\textup{(i)}] $x=x'$,
\item[\textup{(ii)}] $b(x)\neq 0$,
\item[\textup{(iii)}] $se=s'e$,
\item[\textup{(iv)}] $ab\eq x a'b$.
\end{enumerate}
\end{definition}

\begin{remark}
{\bf (1)} Under the conditions of the above definition, the
coordinate "$s$" in $\trip asx$ could just as well be dropped, since
one usually assumes that the fibers are pairwise disjoint, and hence
there is only one $s$ such that $a\in \A_s$.  Nevertheless we believe
it is convenient to mention $s$ explicitly.

{\bf (2)} Since $b$ identifies with a function on $X$, which is
supported on $\U_e$, the fact that $b(x)\neq0$ implies $x\in \U_e$.

{\bf (3)}
Observe that $ab\in \A_s\A_e \sbe \A_{se}$,  and $a'b\in \A_{s'}\A_e \sbe
\A_{s'e}$.  Since $se=s'e$, we see that both $ab$ and $a'b$ lie in the
same fiber of $\A$,  and hence (iv) is meaningful.
\end{remark}

\begin{proposition}
  The relation "$\sim$" defined in~\textup{\ref{DefineEquivRel}} is an
equivalence relation.
\end{proposition}
\begin{proof} Given $\trip asx\in\F$, one has $\trip asx \sim \trip asx$
by taking $e=s^*s$, and any $b\in \A_e$, with $b(x)\neq0$.  That our
relation is symmetric is obvious.

With respect to transitivity suppose that $\trip asx \sim \trip
{a'}{s'}{x'} \sim \trip {a''}{s''}{x''}$.  Take $e$ and $f$ in $E(S)$,
$b\in \A_e$, and $c\in \A_f$, satisfying the conditions of Definition~\ref{DefineEquivRel} with respect to each one of the two
occurrences of "$\sim$" above, respectively.  Noticing that
$x=x'=x''$, let $g=ef$, and $d=bc$. Clearly
  $$
  d\in \A_e\A_f\sbe \A_{ef} = \A_g.
  $$
  Moreover
  $d(x) = b(x)c(x)\neq 0$, and
  $$
  sg = sef = s'ef = s'fe = s''fe = s''g.
  $$
  By Lemma~\ref{L:Equivinh}(iii), we have
  $$
  ad  =
  abc \eq x
  a'bc =
  a'cb \eq x
  a''cb =
  a''d,
  $$
so the conclusion follows from Lemma~\eqref{L:Equivinh}(v).
\end{proof}

Recall that $\t\colon S\to \pb{X}$, $s\mapsto\t_s$ is an action of $S$ on $X=\Gz$. Here $\pb{X}$ denotes the inverse
semigroup of all partial bijections of $X$. Moreover, it induces an action $\tilde\t\colon S\to\pb{C_0(X)}$ on the commutative
\cstar{algebra} $\contz(X)$ in the canonical way: $\tilde\t_s\colon\contz(\U_{s^*s})\to \contz(\U_{ss^*})$ is defined by
$\tilde\t_s(f)\defeq f\circ\t_s\inv$ for all $s\in S$ and $f\in \contz(\U_{s^*s})$.
We are tacitly identifying $\contz(X)$ with $C^*(\E)$, where $\E=\A\rest E$. Under this identification, $\A_e\sbe C^*(\E)$ corresponds to
the ideal $\contz(\U_{s^*s})\sbe\contz(X)$. Thus, we may view $\tilde\t$ as an action of $S$ on $C^*(\E)$.

\begin{lemma}\label{NiceRelation}
If $s\in S$, $a\in \A_s$ and $d\in \A_{s^*s}$, then $ad = \tilde\t_s(d)a$.
Moreover,
$$ab^*c=cb^*a\quad\mbox{for all }a,b,c\in \A_s.$$
\end{lemma}
\begin{proof}  We first prove that $ab^*b=bb^*a$ for all $b\in \A_s$. We have
\begin{align*}
  (ab^*b - bb^*a)&(ab^*b - bb^*a)^*  = (ab^*b - bb^*a) (b^*ba^* - a^*bb^*) \\
  & = ab^*bb^*ba^* - ab^*(ba^*)(bb^*) - b(b^*a)(b^*b)a^* + bb^*aa^*bb^* \\
  & = ab^*bb^*ba^* - ab^*bb^*ba^* - bb^*bb^*aa^* + bb^*aa^*bb^* = 0.
\end{align*}
Therefore, $ab^*b = bb^*a$ for all $b\in \A_s$. By polarization, $ab^*c=cb^*a$ for all $b,c\in \A_s$.
Note that $\tilde\t_s(b^*c)=cb^*$ because $(b^*c)(x)=(cb^*)(\t_s(x))$ for all $x\in \U_{s^*s}$.
We conclude that $ab^*c=\tilde\t_s(b^*c)a$ for all $b,c\in \A_s$. Since $\A_s^*\A_s$ spans a dense subspace of $\A_{s^*s}$
the assertion follows.
\end{proof}

Now, we can show that the equivalence relation defined in~\ref{DefineEquivRel} also has a left hand side version in the following sense.

\begin{lemma}\label{L:left hand side version of equivalence in L}
Given $\trip asx, \trip{a'}{s'}{x'}\in \F$, we have $\trip asx\sim \trip{a'}{s'}{x'}$ if and only if
there is $f\in E(S)$ and $c\in \A_f$ satisfying
\begin{enumerate}
\item[\textup{(i)}] $x=x'$,
\item[\textup{(ii)}] $c(y)\neq 0$, where $y=\theta_s(x)=\theta_{s'}(x)$,
\item[\textup{(iii)}] $fs=fs'$, and
\item[\textup{(iv)}] $ca\eq x ca'$.
\end{enumerate}
\end{lemma}
\begin{proof}
Suppose that $\trip asx\sim \trip{a'}{s'}{x'}$, so that $x=x'$ and there is $e\in E(S)$ and $b\in \A_e$ with $b(x)\neq 0$, $se=s'e$ and $ab\eq x a'b$.
Replacing $e$ by $es^*ss'^*s'$ and using Lemma~\ref{L:Equivinh}(iii) to multiply the equation $ab\eq x a'b$ on the right by a function in
$d\in \A_{(s^*s)(s'^*s')}$ with $d(x)\neq 0$, we may assume that $e\leq s^*s$ and $e\leq s'^*s'$. Thus, we may assume
$b\in \A_e\sbe\A_{s^*s}\cap\A_{s'^*s'}$. By Lemma~\ref{NiceRelation}, we have $ab=\tilde\t_s(b)a$ and $a'b=\tilde\t_{s'}(b)a$.
Note that $\tilde\t_s(b)=\tilde\t_{se}(b)=\tilde\t_{s'e}(b)=\tilde\t_{s'}(b)$. Define $c$ to be this common value.
We have $c=\tilde\t_s(b)\in \tilde\t_s(\A_e\cap\A_{s^*s})=\A_{ses^*}$ by \cite[Proposition 4.5]{Exel:inverse.semigroups.comb.C-algebras},
and similarly $c=\tilde\t_{s'}(b)\in \A_{s'es'^*}$. Defining  $f$ to be $gg'$, where $g\defeq ses^*$ and $g'\defeq s'es'^*$,
we therefore have $c\in \A_f$. Moreover, $gs=ses^*s=se=s'e=s'es'^*s'=g's'$,
so that $fs=fs'$. We conclude that $c\in \A_f$ and $ca\eq x ca'$. Finally, note that $c(y)=\tilde\t_{s}(b)(y)=b(\t_s\inv(y))=b(x)\neq 0$.
Thus we have shown (i)-(iv) provided $\trip asx\sim \trip{a'}{s'}{x'}$. Conversely, if (i)-(iv) hold, then we may use a similar argument to
show that $\trip asx\sim \trip{a'}{s'}{x'}$. In fact, as before, we may assume $f\leq ss^*$ and $f\leq s's'^*$.
Then we use Lemma~\ref{NiceRelation} again to conclude that $ca=a\tilde\t_s\inv(c)$ and $ca'=a\tilde\t_{s'}\inv(c)$. Defining $b\defeq \tilde\t_{s}\inv(c)=\tilde\t_{s'}\inv(c)$ and $e=(s^*fs)(s'^*fs)$, and proceeding as before, we conclude that $b\in \A_e$, $b(x)\neq 0$, and $ab\eq x a'b$.
Therefore, $\trip asx\sim \trip{a'}{s'}{x'}$, as desired.
\end{proof}

\begin{definition}
The equivalence class of each $\trip asx$ in $\F$ will be denoted by
$\qtrip asx$. The quotient of $\F$ by this equivalence relation will
be denoted by $L$.
\end{definition}

In what follows we shall give $L$ the structure of a Fell line bundle
over the groupoid $\G$ constructed in Section~\ref{sec:the construction of G}.

\begin{proposition}
Given $\qtrip asx,\qtrip{a'}{s'}{x'}$ in $L$, we have $\qtrip asx=\qtrip{a'}{s'}{x'}$ if and only if
$\germ sx=\germ{s'}{x'}$ in $\G$, and $ac\eq x a'c$ for all $e\in E(S)$ with $se=s'e$ and for all $c\in \A_e$.
In particular, the correspondence
  $$
  \pi\colon
  \qtrip asx \mapsto \germ sx
  $$
  is a well-defined surjection from $L$ to $\G$.
\end{proposition}
\begin{proof}
If $\qtrip asx=\qtrip{a'}{s'}{x'}$, then, by definition, we have $x=x'$ and there is $f\in E(S)$ and $d\in \A_f$ with $sf=s'f$,
$d(x)\not=0$ and $ad\eq x a'd$. In particular, this implies $\germ sx=\germ{s'}{x'}$.
Now, take any $e\in E(S)$ with $se=s'e$ and any element $c\in \A_e$. By Lemma~\ref{L:Equivinh}(iii), we have
$adc\eq xa'dc$. Since $dc=cd$ and $d(x)\not=0$, this implies $ac\eq x a'c$.
Conversely, assume that $\germ sx=\germ{s'}{x'}$ and $ac\eq xa'c$ for all $e\in E(S)$ with $se=s'e$ and $c\in \A_e$.
The equality $\germ sx=\germ{s'}{x'}$ means that $x=x'$ and there is $f\in E(S)$ with $x\in \U_f$ and $sf=s'f$.
Since $\U_f$ is the spectrum of the commutative \cstar{algebra} $\A_f$, there is $d\in \A_e$ with $d(x)\not=0$.
And by hypothesis, $ad\eq x a'd$, so that $\qtrip asx=\qtrip{a'}{s'}{x'}$.
\end{proof}

From now on, for each $\gamma\in\G$ we  let
$L_\gamma$ denote $\pi^{-1}(\gamma)$, and call it the \emph{fiber of
$L$ over $\gamma$}.
The next result provides a linear structure on each fiber of $L$.

\begin{proposition}\label{def:linear structure on the fibers of L}
Given $\g\in \G$ and elements $\qtrip asx, \qtrip btx\in L$ which are in the same fiber $L_\g$, we define
$$\qtrip asx +\qtrip btx\defeq  \qtrip{ac+bc}{se}{x},$$
where $e\in E(S)$ is any idempotent satisfying $se=te$ and $x\in \U_e$,
and $c\in \A_e$ is any function with $c(x)=1$. Then this is a
well-defined addition operation on $L_\g$.
Moreover, if $s=t$ \textup(so that $a,b$ belong to the same fiber $\A_s$\textup), then
$$\qtrip asx +\qtrip bsx= \qtrip{a+b}{s}{x}.$$
Given $\lambda\in \C$, we define
$$\lambda\cdot\qtrip asx\defeq \qtrip{\lambda a}{s}{x}.$$
Then, this is a well-defined scalar product on $L_\g$. With this structure, $L_\g$ is a complex vector space.
Moreover, the assignment
$$\qtrip asx\mapsto \big\|\qtrip asx\big\|\defeq \sqrt{(a^*a)(x)}$$
is a well-defined norm on the fiber $L_\g$. Hence $L_\g$ is a normed vector space.
\end{proposition}
\begin{proof}
First, note that if $\qtrip asx$ and $\qtrip btx$ are in the same fiber $L_\g$, then $\germ sx=\pi(\qtrip asx)=\pi(\qtrip btx)=\germ tx$,
so that there is an idempotent $e\in E(S)$ with $se=te$ and $x\in \U_e$.
To show that the definition of the sum above does not depend
on the choices of $e\in E(S)$ and $c\in \A_e$, take another idempotent $e'\in E(S)$ with $se'=te'$ and $x\in \U_{e'}$, and
another function $c'\in\A_{e'}$ with $c'(x)=1$. We have to show that
\begin{equation}\label{eq:equality of elements defining the sum in L}
\qtrip{ac+bc}{se}{x}=\qtrip{ac'+bc'}{se'}{x}.
\end{equation}
Define $f\defeq ee'\in E(S)$, and take any function $d\in \A_f$ with $d(x)\not=0$. Note that $x\in \U_f$ and $r\defeq (se)f=sf=(se')f=(te')f=tf=(te)f$.
Since $c(x)=c'(x)=1$, it is easy to check that $\xi c\eq x \xi c'$ for all $\xi\in \A_r$. In particular,
$(ad+bd)c\eq x (ad+bd)c'$, and hence (using that $cd=dc$ and $c'd=dc'$)
$$(ac+bc)d=(ad+bd)c\eq x (ad+bd)c'=(ac'+bc')d.$$
This verifies Equation~\eqref{eq:equality of elements defining the sum in L}. Next, we show that the sum on $L$ does not depend on representatives:
suppose that $\qtrip{a'}{s'}{x}=\qtrip asx$ and $\qtrip{b'}{t'}{x}=\qtrip btx$. All these elements belong to the same fiber $L_\g$.
Hence, there is $e\in E(S)$ with $x\in \U_e$ and $r\defeq se=s'e=t'e=te$. The equality $\qtrip{a'}{s'}{x}=\qtrip asx$ yields $f\in E(S)$
and $c\in \A_f$ with $c(x)\not=0$, $s'f=sf$ and $a'c\eq x ac$. And the equality $\qtrip{b'}{t'}{x}=\qtrip btx$ yields $g\in E(S)$ and
$d\in \A_g$ with $d(x)\not=0$, $t'g=tg$ and $b'd\eq x bd$. Rescaling $c$ and $d$, we may assume that $c(x)=d(x)=1$. Moreover,
replacing the idempotents $e,f,g$ by the product $efg$, and using Lemma~\ref{L:Equivinh}(iii) to replace
the functions $c$ and $d$ by a function of the form $hcd\in \A_{efg}$,
where $h$ is any function in $\A_e$ with $h(x)=1$, we may further assume that $e=f=g$ and $c=d$.
Thus, all the elements $ac,bc,a'c,b'c$ belong to the same fiber $\A_r$, and we have $ac\eq x a'c$ and $bc\eq x b'c$,
that is $ac-a'c\in \A_{(r,x)}$ and $bc-b'c\in \A_{(r,x)}$. Therefore,
$$(ac+bc)-(a'c+b'c)=(ac-a'c)+(bc-b'c)\in \A_{(r,x)},$$
that is, $ac+bc\eq x a'c+b'c$. This shows that the sum on $L_\g$ is well-defined.
If $a,b$ belong to the same fiber $\A_s$ and if $c\in \A_e$ is such that $c(x)=1$, where $e\in E(S)$ and $x\in \U_e$,
then $\qtrip{ac+bc}{se}{x}=\qtrip{(a+b)c}{se}{x}=\qtrip{a+b}{s}{x}$, so that
$$\qtrip asx +\qtrip btx= \qtrip{a+b}{s}{x}.$$
It is easy to see that the scalar product
is also well-defined and that $L_\g$ is a complex vector space with this structure.
To see that the map $\qtrip asx\mapsto (a^*a)(x)\half$ is well-defined, assume that $\qtrip asx=\qtrip btx$.
First suppose that $s=t$, so that $a,b$ are in the same fiber $\A_s=\A_t$. In this case, the equality
$\qtrip asx=\qtrip btx$ means that $a\eq x b$. Multiplying this equation on the left by $a^*$ and using Lemma~\ref{L:Equivinh}(ii)
we get $a^*a\eq x a^*b$, that is, $(a^*a)(x)=(a^*b)(x)$. Similarly, $(b^*b)(x)=(b^*a)(x)$. Since $(a^*b)(x)=\overline{(b^*a)(x)}$,
we get $(a^*a)(x)=(b^*b)(x)$. In the general case, if $a$ and $b$ are in different fibers $\A_s$ and $\A_t$,
the equality $\qtrip asx=\qtrip btx$ yields $e\in E(S)$ and $c\in \A_e$ with $c(x)\not=0$, $se=te$ and $ac\eq x bc$.
Now, $ac$ and $bc$ are in the same fiber $\A_{se}=\A_{te}$. The previous argument implies $(ac)^*(ac)(x)=(bc)^*(bc)(x)$.
But $(ac)^*(ac)=(c^*a^*ac)(x)=(c^*c)(x)(a^*a)(x)$ and similarly $(bc)^*(bc)(x)=(c^*c)(x)(b^*b)(x)$. Since $(c^*c)(x)>0$,
we conclude that $(a^*a)(x)=(b^*b)(x)$. Therefore $\qtrip asx\mapsto (a^*a)(x)\half$ is a well-defined map,
which is easily seem to be a norm on $L_\g$.
\end{proof}

Next, we prove that $L$ has one-dimensional fibers:

\begin{proposition}
Given $\g=\germ sx\in \G$, take $a\in \A_s$ with $(a^*a)(x)>0$. Then, for any element $\qtrip bty$ in the fiber $L_\g$,
there is a unique $\lambda\in \C$ such that $\qtrip bty= \lambda\cdot\qtrip asx$.
Hence, the singleton formed by the element $\qtrip asx$ is a basis for $L_\g$ and we have $L_\g\cong\C$ as complex vector spaces.
Moreover, the map
$$\lambda\mapsto \frac{\lambda\cdot\qtrip asx}{\sqrt{(a^*a)(x)}}$$
defines an isomorphism $\C\congto L_\g$ of normed vector spaces. In particular, each fiber $L_\g$ of $L$ is a Banach space.
\end{proposition}
\begin{proof}
Since $\qtrip bty\in L_\g$, we have $\germ ty=\pi\big(\qtrip bty\big)=\g=\germ sx$. Thus, $x=y$ and there is $e\in E(S)$ such that $x\in \U_e$ and $te=se$. If $c\in \A_e=\contz(\U_e)$ is any function with $c(x)=1$, then we have $\qtrip btx=\qtrip{bc}{te}{x}$ and $\qtrip asx=\qtrip{ac}{se}{x}$. In fact, if $d\in \A_e$ is any function with $d(x)\not=0$, then it is easy to see that $bcd\eq x bd$ and $acd\eq x ad$.
Hence, replacing $b$ by $bc$, and $a$ by $ac$, we may assume that both $a,b$ are in the same fiber, say $\A_s$. Thus, we want to show that
there is a unique $\lambda\in \C$ satisfying $\qtrip bsx=\lambda\cdot \qtrip asx=\qtrip{\lambda a}{s}{x}$. Since both $a,b$ belong
to the same fiber $\A_s$, this is the same as to show that $b\eq x \lambda a$. Multiplying this equation on the left by $a^*$ (and using Lemma~\ref{L:Equivinh}(ii)), we see that if $\lambda$ exists, it has to be $\frac{(a^*b)(x)}{(a^*a)(x)}$. Now, to see that this $\lambda$ works,
we compute
\begin{equation}\label{eq:computation to prove b =x lambda a}
(b-\lambda a)^*(b-\lambda a)(x)=(b^*b)(x)-\lambda(b^*a)(x)-\overline\lambda(a^*b)(x)+|\lambda|^2(a^*a)(x).
\end{equation}
Note that
$$\lambda (b^*a)(x)=\frac{(a^*b)(x)}{(a^*a)(x)}(b^*a)(x)=\frac{(a^*bb^*a)(x)}{(a^*a)(x)}=\frac{(a^*a)(x)(bb^*)(\t_s(x))}{(a^*a)(x)}=(b^*b)(x).$$
Similarly, one proves that $\overline\lambda(a^*b)(x)=|\lambda|^2(a^*a)(x)=(b^*b)(x)$, so that Equation~\eqref{eq:computation to prove b =x lambda a}
equals zero. Therefore, $b\eq x\lambda a$ for $\lambda=\frac{(a^*b)(x)}{(a^*a)(x)}$. Finally, it is easy to see
that the map $\C\ni\lambda\mapsto \frac{\lambda\cdot\qtrip asx}{\sqrt{(a^*a)(x)}}\in L_\g$ is an isometric isomorphism of complex vector spaces.
Its inverse is the map $\qtrip bsx\mapsto \frac{(a^*b)(x)}{\sqrt{(a^*a)(x)}}$.
The element $\qtrip asx$ is therefore a basis vector for $L_\g$ and $L_\g\cong\C$ as complex normed vector spaces.
\end{proof}

In order to define a topology on $L$ we shall use Proposition~\ref{T:topology on Banach bundles from predefined sections}.
Given $a\in \A_s$, we define the local section $\hat{a}$ of $L$ by the formula
\begin{equation}\label{eq:definition of local sections hat a}
\hat{a}\big(\germ sx\big)\defeq \qtrip asx\quad\mbox{for all }\germ sx\in \O_s.
\end{equation}
Thus, by definition, the domain of $\hat{a}$ is the open subset $\dom(\hat{a})\defeq \O_s\sbe\G$.
Recall that $\O_s=\O(s,\U_{s^*s})=\{\germ sx\colon x\in \U_{s^*s}\}$.
\begin{proposition}
There is a unique topology on $L$ making it a continuous Banach bundle and making all the local sections $\hat{a}$ with $a\in \A$ continuous
for this topology. Moreover, with this topology, $L$ is a complex line bundle, that is, a locally trivial one-dimensional complex vector bundle.
\end{proposition}
\begin{proof}
We are going to prove (i) and (ii) in Proposition~\ref{T:topology on Banach bundles from predefined sections} in order to find the required topology.
Property (i) is obvious since every element of $L$ has the form $\qtrip asx$ for some $a\in \A_s$ and $\germ sx\in \O_s$.
To prove (ii), suppose we have finitely many elements $s_i\in S$, $a_i\in \A_{s_i}$ and $\lambda_i\in \C$ for $i=1,\ldots,n$.
We want to show that the set
$$\V=\left\{\g\in \bigcap\limits_{i=1}^n\dom(\hat{a}_i)\colon\left\|\sum\limits_{i=1}^n\lambda_i\hat{a}_i(\g)\right\|<\alpha\right\}$$
is open in $\G$ for all $\alpha>0$. Given $\g_0$ in $\V$, it belongs to $\dom(\hat{a}_i)=\O_{s_i}$ for all $i=1,\ldots,n$, so it
has equivalent representations of the form $\g_0=[s_i,x_0]$ for all $i=1,\ldots,n$.
Thus, there is $e\in E(S)$ such that $x_0\in \U_e$ and $t\defeq s_1e=\ldots=s_ne$.
Replacing $e$ by the product $(s_1^*s_1)\ldots(s_n^*s_n)e$, we may assume $e\leq s_i^*s_i$ and hence $\U_e\sbe\U_{s_i^*s_i}$  for all $i$.
Take a function $b\in \contz(\U_e)\cong\A_e$ which is identically $1$
on a neighborhood $U_0\sbe\U_e$ of $x_0$. It is easy to see that $b(x)a_ic\eq x a_ibc$ for all $x\in \U_e$ and all $c\in \A_e$.
In particular, $a_ic\eq x a_ibc$ and hence $\qtrip{a_i}{s_i}{x}=\qtrip{a_ib}{s_ie}{x}$ for all $x\in U_0$.
Hence, $\O(t,U_0)=\O(s_i,U_0)\sbe\cap_{i}\O_{s_i}=\cap_{i}\dom(\hat{a}_i)$, and if $\g=\germ tx=\germ{s_i}{x}\in \O(t,U_0)$, we have
\begin{align*}
\lambda_i\hat{a_i}(\gamma)&=
\qtrip{\lambda_ia_i}{s_i}{x}\\
    &=\qtrip{\lambda_ia_ib}{s_ie}{x}=\widehat{\lambda_ia_ib}(\germ{s_ie}{x})=\widehat{\lambda_ia_ib}(\g)
\end{align*}
Note that $\lambda_ia_ib\in \A_{s_ie}=\A_t$ for all $i$. Defining $a\defeq \sum\limits_{i=1}^n\lambda_ia_ib\in \A_t$, we conclude that
$$\sum\limits_{i=1}^n\lambda_i\hat{a}_i(\g)=\hat{a}(\g)\quad\mbox{for all }\g\in \O(t,U_0).$$
Note that $\|\hat{a}(\g)\|=\sqrt{(a^*a)(x)}$ for all $\g=\germ tx\in \O(t,U_0)$. Since $\sqrt{(a^*a)(x_0)}=\|\hat{a}(\g_0)\|<\alpha$
and $a^*a$ is continuous, there is a neighborhood $U$ of $x_0$ contained in $U_0$ such that $\|\hat{a}(\g)\|<\alpha$ for all $g$
in the open subset $\U\defeq \O(t,U)$ of $\G$. All this implies that $\g_0\in \U\sbe \V$ and therefore $\V$ is an open subset of $\G$.
By Proposition~\ref{T:topology on Banach bundles from predefined sections}, there is a unique topology on $L$ turning it into an
{\usc} Banach bundle and making the local section $\hat{a}\colon\O_s\to L$ continuous for all $a\in A_s$, $s\in S$.
As we have already noted, $\|\hat{a}(\g)\|=\sqrt{(a^*a)(\s(\g))}$ for all $\g\in \dom(\hat{a})$. Since the map $\g\mapsto \sqrt{(a^*a)(\s(\g))}$
is continuous from $\G$ to $\R^+$, we conclude that the norm on $L$ is continuous
(again by Proposition~\ref{T:topology on Banach bundles from predefined sections}). Therefore, $L$ is in fact a continuous Banach bundle, as desired.
It remains to show that $L$ is locally trivial. However, the local sections $\hat{a}$ are continuous and do not vanish on $\O(s,\dom(a))\sbe\O_s$.
Since ($\O_s$ and hence also) $\O(s,\dom(a))$ is a locally compact Hausdorff space, we may apply the usual methods to conclude that $L$ is
trivializable over $\O(s,\dom(a))$. A local trivialization is provided by the map $(\lambda,\g)\mapsto \lambda\cdot \hat{a}(\g)$
from $\C\times\O(s,\dom(a))$ to $L$. See also \cite[Remark~II.13.9]{fell_doran}.
\end{proof}

Given $v, w\in L$, with $\big(\pi(v),\pi(w)\big)\in \Gt$, we shall
define the product $vw$ as follows.  Write $v=\qtrip asx$ and
$w=\qtrip bty$, so our assumption translates into
  $$
  \big(\germ sx, \germ ty\big)\in \Gt,
  $$
  and hence $x=\t_t(y)$.

\begin{proposition}\label{prop:definition of the multiplication on L}
  With notation as above,  putting
  $$
  vw=
  \qtrip{ab}{st}{y},
  $$
we get a well-defined operation on $L$, that is, the right-hand side
does not depend on the choice of representatives for $v$ and $w$.
\end{proposition}
\begin{proof} Suppose that $v$ has another representation as
$v=\qtrip {a'}{s'}{x'}$, in which case $x=x'$, and there exists $e\in
E(S)$ and $c\in \A_e$, such that
  $$
  c(x)\neq 0,\quad se=s'e \quad\mbox{and}\quad ac \eq x a'c.
  $$
  To show that
\begin{equation}\label{ProdchangeLeft}
  \qtrip{ab}{st}{y}=\qtrip{a'b}{s't}{y},
\end{equation}
  we will check that the conditions of~\ref{DefineEquivRel} are
satisfied for the idempotent $f=t^*et$, and the element $d=b^*cb\in
\A_f$, under the special case in which $(b^*b)(y)\neq0$.  With respect
to condition~\ref{DefineEquivRel}.(ii), Lemma~\ref{lem:construction of theta_a} yields
  $$
  d(y) =
  (b^*cb)(y) =
  (b^*b)(y)\; c(\t_t(y)) =
  (b^*b)(y)\; c(x) \neq 0.
  $$
  Checking~\ref{DefineEquivRel}.(iii) we have
  $$
  stf =
  st t^*et =
  set t^*t =
  s'et t^*t =
  s't t^*et =
  s'tf.
  $$
  As for~\ref{DefineEquivRel}.(iv), recall that the fibers over idempotent
elements are commutative, so
  $$
  abd = ab b^*cb =
  acb b^*b \eq y
  a'cb b^*b =
  a'b b^*cb =
  a'b d,
  $$
  where the crucial middle step is a consequence of
Lemma~\ref{L:Equivinh}(iv), given that $b b^*b\in \A_t$, and
$\t_t\inv(x)=y$.  The verification of Equation~\eqref{ProdchangeLeft} is thus
complete when $(b^*b)(y)\neq0$.

Suppose now that $(b^*b)(y)=0$.  Then it follows from
Lemma~\ref{L:Equivinh}(i) that $b\eq y 0_t$, and hence from
Lemma~\ref{L:Equivinh}(ii) we have
  $$
  ab \eq y 0_{st} \quad\mbox{and}\quad a'b \eq y 0_{s't}.
  $$
  It follows that
  $$
  \qtrip{ab}{st}{y} =   \qtrip{0}{st}{y} \quad\mbox{and}\quad
  \qtrip{a'b}{s't}{y} =   \qtrip{0}{s't}{y}.
  $$
  and it suffices to show that
  $
  \qtrip{0}{st}{y} = \qtrip{0}{s't}{y}.
  $
  Let  $f=t^*et$, as above,  and notice that since $x\in\U_e\cap\U_{tt^*}$, one has
  $$
  y = \t_{t^*}(x)\in \t_{t^*}(\U_e\cap\U_{tt^*}) = \U_{t^*et}=\U_f.
  $$
  Choosing any $c\in \A_f$, such that $c(x)\neq0$, one verifies the
conditions in~\ref{DefineEquivRel}, hence proving
Equation~\eqref{ProdchangeLeft}.

Next suppose that one is given another representation of $w$ as
$w=\qtrip {b'}{t'}{y'}$, in which case $y=y'$, and there exists $e\in E(S)$ and $c\in \A_e$, such that
$$
  c(x)\neq 0,\quad te=t'e \quad\mbox{and}\quad bc \eq y b'c.
$$
It follows that
$$
  ste=st'e \quad\mbox{and}\quad abc \eq y ab'c,
$$
the last relation being a consequence of
Lemma~\ref{L:Equivinh}(ii).  Therefore
  $$
  \qtrip{ab}{st}{y} = \qtrip{ab'}{st'}{y}.
  $$
\vskip-16pt
\end{proof}

\begin{proposition}\label{prop: definition of the involution on L}
The assignment
$$\qtrip asx\mapsto \qtrip asx^*\defeq \qtrip{a^*}{s^*}{\theta_s(x)}$$
is a well-defined operation on $L$.
\end{proposition}
\begin{proof}
Suppose $\qtrip asx=\qtrip bty$ in $L$, that is, $x=y$ and there is $e\in E(S)$ and $c\in \A_e$ such that $se=te$, $c(x)\not=0$
and $ac\eq x bc$. Then $\t_s(x)=\t_t(y)$, $es^*=et^*$ and $c^*a^*\eq{\t_s(x)} c^*b^*$ by Lemma~\ref{L:Equivinh}(vi). And by
Lemma~\ref{L:left hand side version of equivalence in L} this implies that $\qtrip{a^*}{s^*}{\t_s(x)}=\qtrip{b^*}{t^*}{\t_t(y)}$.
\end{proof}

\begin{theorem}\label{T:L is a Fell line bundle}
With the multiplication defined in Proposition~\textup{\ref{prop:definition of the multiplication on L}} and the involution defined in
Proposition~\textup{\ref{prop: definition of the involution on L}}, $L$ is a Fell line bundle,
that is, a one-dimensional, locally trivial \textup(continuous\textup) Fell bundle over the étale groupoid $\G$.
\end{theorem}
\begin{proof}
We already know that $L$ is a one-dimensional, locally trivial
continuous Banach bundle with the unique topology making the local sections $\hat{a}$ continuous for all $a\in\A$.
Since the algebraic operations on $L$ are essentially inherited from $\A$, it is easy to see that all the algebraic properties
required in Definition~\ref{D:definition of Fell bundle over groupoid} are indeed satisfied.
Let us check axioms (iv), (viii) and (ix) in Definition~\ref{D:definition of Fell bundle over groupoid}. Given $\qtrip asx,\qtrip bty\in L$
with $\t_t(y)=x$, Lemma~\ref{lem:construction of theta_a} yields
\begin{align*}
    \big\|\qtrip asx\cdot \qtrip bty\big\|^2&=\big\|\qtrip{ab}{st}{y}\big\|^2=(b^*a^*ab)(y)=(b^*b)(y)(a^*a)(\t_t(y))\\
    &=(b^*b)(y)(a^*a)(x)=\big\|\qtrip asx\big\|^2\big\|\qtrip bty\big\|^2.
\end{align*}
Thus $\big\|\qtrip asx\cdot \qtrip bty\big\|=\big\|\qtrip asx\big\|\cdot\big\|\qtrip bty\big\|$. This, of course, proves (iv).
To prove (viii), we compute
\begin{align*}
\big\|\qtrip asx^*\cdot\qtrip asx\big\|&=\big\|\qtrip{a^*a}{s^*s}{x}\big\|=\sqrt{\big((a^*a)^*(a^*a)\big)(x)}\\
                                        &=(a^*a)(x)=\big\|\qtrip asx\big\|^2.
\end{align*}
To check (ix), it is enough to observe that
$$\qtrip asx^*\cdot\qtrip asx=\qtrip{a^*a}{s^*s}{x}=\qtrip{(a^*a)\half}{s^*s}{x}^*\cdot \qtrip{(a^*a)\half}{s^*s}{x}$$
which is an element of the \cstar{algebra} $L_{\germ{s^*s}{x}}$ of the form $w^*w$, with $w\in L_{\germ{s^*s}{x}}$, and therefore positive.
The relation $(a^*a)(x)=(aa^*)(\t_s(x))$ for all $a\in \A_s$ and $x\in \U_{s^*s}$ implies that the involution is isometric:
$\big\|\qtrip asx^*\big\|=\big\|\qtrip asx\big\|$. Finally, we show that the multiplication and the involution on $L$ are continuous.
By Proposition~\ref{prop:continuity of multiplication and involution with local section} it is enough to prove the relations
\begin{equation}\label{E:relations of multiplication and involution for the Gelfand maps}
\hat{a}(\g)\cdot\hat{b}(\g')=\widehat{ab}(\g\g')\quad\mbox{and}\quad\widehat{a}(\g)^*=\widehat{a^*}(\g\inv)
\end{equation}
for all $a,b\in \A$, $\g\in \dom(\hat{a})$ and $\g'\in\dom(\hat{b})$ with $\r(\g')=\s(\g)$. Suppose $a\in \A_s$ and $b\in \A_t$.
By definition, we have
$$\dom(\hat{a})=\O_s=\{\germ sx\colon x\in \U_{s^*s}\}\quad\mbox{and}\quad \dom(\hat{b})=\O_t=\{\germ ty\colon y\in \U_{t^*t}\}.$$
Suppose $\g=\germ sx$ and $\g'=\germ ty$ with $\r(\g')=\t_t(y)=x=\s(\g)$. Then
$$\hat{a}(\g)\cdot\hat{b}(\g')=\qtrip asx\cdot \qtrip bty=\qtrip{ab}{st}{y}=\widehat{ab}(\germ{st}{y})=\widehat{ab}(\g\g'),$$
and
$$\widehat{a}(\g)^*=\qtrip asx^*=\qtrip{a^*}{s^*}{\t_s(x)}=\widehat{a^*}(\germ{s^*}{\t_s(x)}=\widehat{a^*}(\g\inv).$$
\vskip-14pt
\end{proof}

\begin{definition}
The Fell line bundle $L=L(\A)$ over $\G=\G(\A)$ constructed above from the Fell bundle $\A=\{\A_s\}_{s\in S}$ over $S$, will be called
the \emph{Fell line bundle associated to $\A$}.
\end{definition}

As already observed in Section~\ref{sec:twisted groupoids and Fell line bundles}, Fell line bundles over $\G$
correspond bijectively to twists over $\G$. The twisted groupoid $(\G(\A),\Sigma(\A))$ corresponding to $L(\A)$ will
be called the \emph{twisted groupoid associated to $\A$}. Thus $\G(\A)$ is the étale groupoid constructed in
Section~\ref{sec:the construction of G} and the extension groupoid $\Sigma(\A)$ is the space of unitary elements of
the Fell line bundle $L(\A)$ constructed above. The algebraic and topological structure on $\Sigma(\A)$ is canonically induced from $L(\A)$.
For instance, the multiplication on $\Sigma(\A)$ is just the multiplication of $L(\A)$ restricted to $\Sigma(\A)$, and the inversion on
$\Sigma(\A)$ is the restricted involution from $L(\A)$. The projection $\Sigma(\A)\onto \G(\A)$ is the restriction of the bundle projection
$L(\A)\onto \G(\A)$. Finally, the inclusion $\Torus\times X\into \Sigma(\A)$ is defined by $(z,x)\mapsto z\cdot 1_x$ for all $z\in \Torus$ and
$x\in X=\G(\A)^{(0)}$, where $1_x$ denotes the unit element of the \cstar{algebra} $L(\A)_x\cong\C$ (the fiber over $x$).

It is also possible to construct the twist groupoid $\Sigma=\Sigma(\A)$ directly from $\A$ following the ideas appearing in \cite{RenaultCartan}.
For convenience, we outline here the main steps into this procedure.  We define $\Sigma$ as the set
$$\Sigma\defeq \{\qstrip{a}{s}{x}\colon a\in \A_s, x\in \dom(a)\},$$
of equivalence classes $\qstrip asx$, where, by definition, $\qstrip{a}{s}{x}=\qstrip{a'}{s'}{x'}$
if and only if $x=x'$, and there is $b,b'\in \E$ such that $b(y),b'(y)>0$ and $ab=a'b'$.
The surjection $\pi\colon\Sigma\to\G$ is defined by $\pi(\qstrip{a}{s}{x})\defeq \germ sx$. It is easy to see
that $\pi$ is well-defined.

The groupoid structure on $\Sigma$ is defined in the same fashion as for $\G$: the source and range maps are $\s(\qstrip{a}{s}{x})=x$ and $\r(\qstrip{a}{s}{x})=\t_s(x)$ and the operations are
$$\qstrip{a}{s}{x}\cdot\qstrip{a'}{s'}{x'}=\qstrip{aa'}{ss'}{x'}\quad\mbox{whenever }\t_{s'}(x')=x \quad\mbox{and}$$
$$\quad\mbox{and}\quad\qstrip{a}{s}{x}\inv=\qstrip{a^*}{s^*}{\t_s(x)}.$$
With this structure, it is not difficult to see that $\Sigma$ is in fact a groupoid.
Before we define the appropriate topology on $\Sigma$,
let us define the inclusion map $\iota\colon\Torus\times X\to \Sigma$:
Given $(z,x)\in \Torus\times X$, define $\iota(z,x)\defeq \qstrip bex\in \Sigma$, where $e\in E(S)$ and $b$ is any element of $\A_e$
with $b(x)\not=0$ and $\frac{b(x)}{|b(x)|}=z$. Then $\iota$ is a well-defined injective morphism of groupoids.

To specify a topology on $\Sigma$
it is enough to define a system of open neighborhoods of a point $\qstrip{a}{s}{x}\in\Sigma$. This is given by the sets
\begin{equation}\label{eq:basic open sets in Sigma}
\O(a,U,V)\defeq \{\qstrip{za}{s}{y}\colon y\in U\mbox{ and }z\in V\},
\end{equation}
where $U$ is an open subset of $\dom(a)$ containing $x$ and $V\sbe\Torus$ is an open subset containing $1$.
With this topology it is not difficult to see that $\Sigma$ is a topological groupoid,
that $\pi\colon\Sigma\onto \G$ is an open continuous map
and that $\iota\colon\Torus\times X\into \Sigma$ is a homeomorphism onto its image:
$$\I=\{\qstrip{b}{e}{x}\colon x\in X, e\in E, b\in \A_e\mbox{ and } b(x)\not=0\}.$$
Moreover,
\[
\xymatrix{\Torus\times X  \ar[r]^{\iota} & \Sigma \ar[r]^{\pi} & \G}
\]
is an exact sequence of topological groupoids and hence $(\G,\Sigma)$ is a twisted groupoid.

Note that $\Sigma$ has a canonical action of $\Torus$:
$$z\cdot\qstrip{a}{s}{x}\defeq \qstrip{za}{s}{x}\quad\mbox{for all }z\in \Torus\mbox{ and }\qstrip{a}{s}{x}\in \Sigma.$$
It is easy to see that this is a well-defined free action of $\Torus$ on $\Sigma$ and the surjection
$\pi\colon\Sigma\to\G$ induces an isomorphism from the
orbit space $\Sigma/\Torus$ onto $\G$. In other words, $\Sigma$ is a principal $\Torus$-bundle over $\G$.
It is also possible to exhibit local trivializations for the $\Torus$-bundle $\Sigma$. In fact, given $a\in \A_s$, the map
$$\psi\colon\Torus\times \dom(a)\to \Sigma|_U,\quad \psi(z,x)\defeq \qstrip{za}{s}{x}=z\cdot\qstrip{a}{s}{x}$$
defines a homeomorphism, where $\Sigma\rest{U}\defeq \pi\inv(U)$ is the restriction of $\Sigma$ to
the open subset $U=\O(s,dom(a))\sbe\G$.
Note that $\psi$ is \emph{$\Torus$-equivariant} in the sense that $\psi(z,\g)=z\cdot\psi(1,\g)$.
Recall that $U$ is a bissection, so we have a canonical homeomorphism $U\cong\dom(a)$ (which is given by the restriction of $\s$ to $U$).
Through this homeomorphism we may also obtain homeomorphisms $\Sigma_U\cong\Torus\times U$. 
These homeomorphisms are compatible with the projections onto $U$, so they are in fact isomorphisms of bundles over $U$.

\begin{proposition}
Let $(\G,\Sigma)$ be a twisted étale groupoid and let $L=(\C\times\Sigma)/\Torus$ be the associated Fell line bundle.
If $S$ is a wide inverse subsemigroup of $S(\G)$ and if $\A=\{\contz(L_U)\}_{U\in S}$ is the Fell bundle over
$S$ defined from $L$ as in Example~\textup{\ref{ex:Fell bundle over semigroups associated to Fell bundles over groupoids}},
then the twisted groupoid $(\G(\A),\Sigma(\A))$ associated to $\A$ is isomorphic to $(\G,\Sigma)$.
\end{proposition}
\begin{proof}
By definition, the inclusion $\iota\colon\Torus\times X\into \Sigma$ is a homeomorphism onto its image $\iota(\Torus\times X)=\pi\inv(X)$,
where $\pi$ is the surjection $\Sigma\to \G$. Thus the restriction $\Sigma\rest{X}=\pi\inv(X)$ is trivializable. It follows that
$L$ is trivializable over $X$, and hence also over any open subset $U\sbe X$. Thus $L_U\cong \C\times U$, so that
$$\A_U=\contz(L_U)\cong \contz(U)\sbe\contz(X)\quad\mbox{for every }U\in E(S).$$
This gives a faithful representation of $\E_\A$ into $\contz(X)$. Since $S$ covers $\G$, the idempotent semilattice $E(S)$ covers $\G^{(0)}=X$.
Indeed, given $x\in X$, there is $U\in S$ with $x\in U$ and hence $x=\s(x)\in \s(U)=U^*U\in E(S)$.
Thus the family $\{\contz(U)\colon U\in E(S)\}$ spans a dense subspace of $\contz(X)$. It follows from \cite[Proposition 4.3]{Exel:noncomm.cartan}
that the representation $\E_\A\to \contz(X)$ integrates to an isomorphism $C^*(\E_\A)\cong \contz(X)$.
So, the spectrum of $C^*(\E_\A)$ may be identified with $X$, and through this identification, the spectrum $\widehat{\A_U}$
of $\A_U$ is $U$ for all $U\in E(S)$. Let us now describe the action $\t$ associated to $\A$.
Given $U\in S$, $\t_U$ is a homeomorphism from $\widehat{\A_{U^*U}}\cong U^*U=\s(U)$
onto $\widehat{\A_{UU^*}}\cong UU^*=\r(U)$ which satisfies
$$(a^*ba)(x)=(a^*a)(x)b(\t_U(x))\quad\mbox{for all }a\in \A_U=\contz(L_U),\,b\in \contz(X)\mbox{ and }x\in \s(U).$$
It is enough to consider $b\in \contz(\r(U))$ in order to characterize $\t_U$. Now, if $\g\in U$ and $x=\s(\g)$,
note that $x=\g\inv \g=\g\inv\r(\g)\g$. By definition of the multiplication on $\A$
(see Example~\ref{ex:Fell bundle over semigroups associated to Fell bundles over groupoids}), we have
$$(a^*a)(x)=a^*(\g\inv)a(\g)=\overline{a(\g)}a(\g)=|a(\g)|^2$$
and
$$(a^*ba)(x)=a^*(\g\inv)b(\r(\g))a(\g)=|a(\g)|^{2}b(\r(\g)).$$
As a consequence, $\t_U(x)=\t_U(\s(\g))=\t_U(\r(\g))$. In other words,
$\t_U$ is the homeomorphism $\tilde\t_U\colon\s(U)\to \r(U)$ given by $\tilde\t_U(x)=\r_U(\s_U\inv(x))$.
The maps $\tilde\t_U$ always give an action of $S$ on $X$. And by \cite[Proposition 5.4]{Exel:inverse.semigroups.comb.C-algebras},
the map $\phi\colon\G(\A)\to \G$ defined by $\phi(\germ{U}{y})=\s_U\inv(y)$ for all $\germ{U}{y}\in \G(\A)$ is an isomorphism
of étale groupoids provided $S$ is wide, which is our case here. Next, we are going to find an isomorphism $\Sigma(\A)\cong \Sigma$.
Recall that $\Sigma$ may be identified with the set of unitary elements in the Fell line bundle $L$ through the map $\sigma\mapsto [1,\sigma]$.
In this way, we define $\psi\colon\Sigma(\A)\to \Sigma$ by
$$\psi(\qstrip{a}{U}{y})\defeq \frac{a(\g)}{|a(\g)|},\mbox{ where }\g=\s_U\inv(y)\in \G,$$
whenever $a\in \A_U=\contz(L_U)$, $y\in \dom(a)\sbe\s(U)$ and $x=\t_U(y)=\r_U(\s_U\inv(y))$.
Notice that $|a(\g)|^2=|a(\g)^*a(\g)|=|a^*(\g\inv)a(\g)|=(a^*a)(y)>0$.
To show that $\psi$ is well-defined, assume $\qstrip{a}{U}{y}=\qstrip{a}{U'}{y}$ in $\Sigma(\A)$, so there are
$b,b'\in \E_{\A}$ such that $ab=a'b'$ and $b(y),b'(y)>0$. Suppose $a\in \A_U$, $a'\in \A_{U'}$, $b\in \A_V$ and $b'\in \A_{V'}$,
where $U,U',V,V'\in S$. Since $b,b'\in \E_\A$, we have $V,V'\in E(S)$, so that $V,V'\sbe X$. If $\g=\s_U\inv(y)$,
then $\g=\g\s(\g)=\g y\in UV\cap U'V'$. Thus $(ab)(\g)=a(\g)b(y)$ and $(a'b')(\g)=a'(\g)b'(y)$. Since $b(y)$ and $b'(y)$ are positive numbers, we get
$$\frac{a(\g)}{|a(\g)|}=\frac{a(\g)b(y)}{|a(\g)b(y)|}=\frac{(ab)(\g)}{|(ab)(\g)|}=\frac{(a'b')(\g)}{|(a'b')(\g)|}=
\frac{a'(\g)b'(y)}{|a'(\g)b'(y)|}=\frac{a'(\g)}{|a'(\g)|}.$$
Let us check that $\psi$ is a groupoid homomorphism. Take $\qstrip{a}{U}{y},\qstrip{b}{V}{z}\in \Sigma(\A)$ with $a\in \A_U$, $b\in \A_V$,
and let $\g_1\in U$ and $\g_2\in V$ such that $\s(\g_1)=y$ and $\s(\g_2)=z$. Then $\g=\g_1\g_2\in UV$ and $\s(\g)=\s(\g_2)=z$,
so that $(ab)(\g)=a(\g_1)b(\g_2)$. Hence,
\begin{align*}
\psi(\qstrip{a}{U}{y}\qstrip{b}{V}{z})&=\psi(\qstrip{ab}{UV}{z})=\frac{(ab)(\g)}{|(ab)(\g)|}\\
&=\frac{a(\g_1)}{|a(\g_1)|}\frac{b(\g_2)}{|b(\g_2)|}=\psi(\qstrip{a}{U}{y})\psi(\qstrip{b}{V}{z}).
\end{align*}
This shows that $\psi$ respects multiplication.
Here we have used that the multiplication in the Fell line bundle $L$ restricts to the
multiplication in the groupoid $\Sigma\sbe L$. Similarly, since the involution in $L$ restricts to the inverse in $\Sigma\sbe L$,
we get that $\psi$ preserves inversion:
$$\psi\big(\qstrip{a}{U}{y}\inv\big)=\psi\big(\qstrip{a^*}{U^{-1}}{\t_U(y)}\big)
        =\frac{a^*(\g_1\inv)}{|a^*(\g_1\inv)|}=\frac{a(\g_1)^*}{|a(\g_1)|}=\psi\big(\qstrip{a}{U}{y}\big)\inv.$$
The pair $(\psi,\phi)$ of groupoid homomorphisms $\psi\colon\Sigma(\A)\to \Sigma$ and $\phi\colon\G(\A)\to \G$ we have defined is a
morphism of extensions, that is, the following diagram commutes:
$$
\xymatrix{
\Torus\times X \ar[d]_{\id}\ar[r]^{\iota_\A} &  \Sigma(\A) \ar[d]_{\psi}\ar[r]^{\pi_\A} &    \G(\A)  \ar[d]_{\phi} \\
\Torus\times X \ar[r]_{\iota}             &  \Sigma \ar[r]_{\pi}              &    \G
}
$$
In fact, recall that $\iota_\A\colon\Torus\times X\to \Sigma(\A)$ is defined by $\iota_\A(z,x)=\qstrip{b}{V}{x}$,
where $V\in E(S)$ and $b\in \A_V$ is such that
$b(x)\not=0$ and $\frac{b(x)}{|b(x)|}=z$.
And the surjection $\pi_\A\colon\Sigma(\A)\to \G(\A)$ is defined by $\pi_\A(\qstrip{a}{U}{y})=\germ{U}{y}$ whenever $a\in \A_U$.
By definition, $\psi(\iota_\A(z,x))=\psi(\qstrip{b}{V}{x})=\frac{b(x)}{|b(x)|}=z$. Here we view $z\in \Torus$ as the element
$z[1,x]=[z,x]$ of the fiber $L_x$. On the other hand, $\iota(z,x)=z\cdot x\in \Sigma$ is identified with the element $[1,z\cdot x]=[z,x]\in L_x$.
This says that the left hand side of the diagram above commutes. To see that the right hand side also commutes,
take $\qstrip{a}{U}{y}\in \Sigma$. If $a\in \A_U$ and $\g=\s_U\inv(y)\in U$, then $\phi(\pi_\A(\qstrip{a}{U}{y}))=\phi(\germ{U}{y})=\s_U\inv(y)=\g$.
On the other hand, $\psi(\qstrip{a}{U}{y})=\frac{a(\g)}{|a(\g)|}\in \Sigma\sbe L$ is also sent to $\g$ via $\pi\colon\Sigma\to\G$ because
$\frac{a(\g)}{|a(\g)|}$ belongs to the fiber $L_\g$ and the restriction of the bundle projection $L\to\G$ to $\Sigma$ equals $\pi$.
Therefore the diagram commutes, as desired. Since $\phi$ is bijective, the commutativity of the diagram forces $\psi$ to be bijective
as well. For example, to prove the injectivity of $\psi$, assume that $\sigma_1,\sigma_2\in \Sigma(\A)$ and
$\psi(\sigma_1)=\psi(\sigma_2)$. Applying $\pi$ and using the commutativity of the right hand side, we get $\phi(\pi_\A(\sigma_1))=\phi(\pi_\A(\sigma_2))$. The injectivity of $\phi$ yields $\pi_\A(\sigma_1)=\pi_\A(\sigma_2)$,
so there is a unique $z\in \Torus$ with $\sigma_2=z\cdot\sigma_1$. By the commutativity of the left hand side $\psi$ must be $\Torus$-equivariant,
so that $\psi(\sigma_1)=\psi(\sigma_2)=z\cdot\psi(\sigma_1)$. Since the $\Torus$-action is free, we conclude that $z=1$ and hence $\sigma_1=\sigma_2$.
Analogously one proves that $\psi$ is surjective. It remains to check that $\psi$ is a homeomorphism.
Since this is a local issue, we may restrict to local trivializations. As we have seem above
each $a\in\A_U$ yields a local trivialization $\Torus\times\dom(a)\cong\Sigma_\U$ through the map $(z,x)\mapsto \qstrip{za}{U}{x}$,
where $\U=\O(U,\dom(a))\sbe\G(\A)$. On the other hand, since $|a(\s_U\inv(x))|^2=(a^*a)(x)>0$, $a$
is a non-vanishing continuous section on $\V=\phi(\U)=\{\s_U\inv(x)\colon x\in \dom(a)\}$. This yields a local trivialization
$\C\times \V\cong L_\V$ through the map $(\lambda,\g)\mapsto \lambda a(\g)$, which induces a local trivialization
$\Torus\times \V\cong\Sigma_\V$ through the map $(z,\g)\mapsto z \frac{a(\g)}{|a(\g)|}$.
Once composed with these trivializations, $\psi\colon\Sigma_\U\to \Sigma_\V$ gives the map $(z,x)\mapsto (z,\s_U\inv(x))$
from $\Torus\times\dom(a)$ to $\Torus\times \V$, which is a homeomorphism because $\s_U$ is. Therefore, $\psi$ is a homeomorphism.
\end{proof}

\begin{remark}
Let $\A=\{\A_s\}_{s\in S}$ be a saturated semi-abelian Fell bundle over an inverse semigroup $S$,
and let $\G$ be the étale groupoid of germs associated to $\A$ as in Section~\ref{sec:the construction of G}.
Then the inverse subsemigroup $T=\{\O_s\colon s\in S\}\sbe S(\G)$ is wide.
In fact, the first property in Definition~\ref{def:wide inv. semigroup}
is obvious because every element $\g\in \G$ has the form $\g=[s,x]$ with $x\in \U_{s^*s}$,
so that $\g\in \O_s$ for some $s\in S$. To prove the second property, take $s,t\in S$ and suppose that $\g\in\O_s\cap\O_t$.
Then $\g=\germ sx=\germ tx$ for some $x\in \U_{s^*s}\cap\U_{t^*t}$. Thus, there is $e\in E(S)$ such that
$x\in \U_e$ and $se=te$. Define $r\defeq se=te\in S$ and note that $r^*r=s^*se=t^*te$, so that $\U_{r^*r}=\U_{s^*s}\cap\U_{t^*t}\cap\U_e$.
In particular $x\in \U_{r^*r}$. Since $re=se=te$, it follows that $\germ rx=\germ sx=\germ tx=\g$ belongs to $\O_r$.
And if $\germ ry$, $y\in \U_{r^*r}$, is an arbitrary element of $\O_r$, it belongs to $\O_s\cap\O_t$ because
$\germ ry=\germ sy=\germ ty$. In fact, $y\in \U_e$ and $re=se=te$.
\end{remark}

\subsection{Characterization of semi-abelian Fell bundles}

Let $\A$ be a semi-abelian, saturated Fell bundle over $S$, let $\G$ be the étale groupoid of germs constructed
in Section~\ref{sec:the construction of G} with unit space $\Gz=X$, the spectrum
of the commutative \cstar{algebra} $C^*(\E_\A)$. Consider the Fell line bundle $L$ over $\G$ associated to $\A$
as in Section~\ref{sec:construction of L}.

Given $s\in S$, we define $L_s=L\rest{\O_s}$ to be the restriction of $L$ to the open subset $\O_s=\O(s,\U_{s^*s})\sbe\G$.
We shall write $\CC_s=\contz(L_s)$ for the space of continuous sections of $L_s$ vanishing at infinity.
Recall that each $\O_s=\{\germ sx\colon x\in \U_{s^*s}\}$ is a bissection of $\G$ and we have
(see \cite[Proposition 7.4]{Exel:inverse.semigroups.comb.C-algebras})
$$\O_s\cdot\O_t=\O_{st}\quad\mbox{and}\quad\O_s\inv=\O_{s^*}\quad\mbox{for all }s,t\in S.$$
This says that the map $s\mapsto \O_s$ is a homomorphism from $S$ to the inverse semigroup $S(\G)$ of all bissections in $\G$.
The restrictions of $\s$ and $\r$ to $\O_s$ will be denoted by $\s_s$ and $\r_s$. Since $\O_s$ is
a bissection, $\s_s\colon\O_s\to \U_{s^*s}$ and $\r_s\colon\O_s\to \U_{ss^*}$ are homeomorphisms. Moreover, from the definitions
of $\s$ and $\r$, it follows that $\r_s\circ\s_s\inv=\t_s$.

\begin{proposition}\label{prop:definition of the fell bundle of continuous sections}
With notations as above, the family of Banach spaces $\CC=\{\CC_s\}_{s\in S}$
is a Fell bundle over $S$ with respect to the following algebraic operations\textup:
\begin{itemize}
\item the multiplication $\CC_s\times\CC_t\to \CC_{st}$
is defined by
$$(\xi\cdot\eta)(\g)\defeq \xi\bigl(\r_s\inv(\r(\g))\bigr)\eta\bigl(\s_t\inv(\s(\g))\bigr)\quad
                                    \mbox{for all }\g\in \O_{st},\xi\in \CC_s,\eta\in\CC_t;$$
\item and the involution $\CC_s\to\CC_{s^*}$ is defined by
$$\xi^*(\g)=\overline{\xi(\g\inv)}\quad\mbox{for all }\g\in \O_{s^*}\mbox{ and }\xi\in \CC_s.$$
\end{itemize}
The inclusion maps are defined in the canonical way: if $s\leq t$ in $S$, then $\O_s\leq\O_t$ in $S(\G)$, that is,
$\O_s\sbe\O_t$. Thus each section $\xi$ of $L_s$ may be viewed as a section of $L_t$ extending it by zero outside $\O_s$.
Hence we define the inclusion map $j_{t,s}\colon\CC_s\to\CC_t$ by $j_{t,s}(\xi)=\tilde{\xi}$ for all $\xi\in \CC_s$,
where $\tilde{\xi}$ denotes the extension of $\xi$ by zero.
\end{proposition}
\begin{proof}
The proof consists of straightforward calculations and is left to the reader.
We just remark that the multiplication is well-defined. In fact, if $\g\in \O_{st}=\O_s\cdot\O_t$,
there is a unique way to write $\g=\g_1\cdot \g_2$ with $\g_1\in \O_s$ and $\g_2\in \O_t$ because $\O_s$ and $\O_t$ are bissections.
Moreover, this unique way is given by $\g_1=\r_s\inv(\r(\g))$ and $\g_2=\s_t\inv(\s(\g))$.
Note that the multiplication we defined on $\cont$ uses the multiplication of $L$ as a Fell line bundle.
Thus $(\xi\cdot\eta)(\g)=\xi(\g_1)\eta(\g_2)\in L_{\g_1}L_{\g_2}\sbe L_{\g_1\g_2}=L_\g$, so that $\xi\cdot\eta$ is a section of $L$.
Since all maps involved in the multiplication are continuous, $\xi\cdot\eta$ is a continuous section and it vanishes at infinity because
$\xi$ and $\eta$ do.
\end{proof}

\begin{theorem}\label{teo:characterization of saturated semi-abelian Fell bundles}
Let $\A=\{\A_s\}_{s\in S}$ be a semi-abelian, saturated Fell bundle and let $\CC=\{\CC_s\}_{s\in S}=\{\contz(L_s)\}_{s\in S}$ be the
\textup(semi-abelian, saturated\textup) Fell bundle constructed above. Given $a\in \A_s$, we define the function $\hat{a}\colon\O_s\to L$ by
\begin{equation}\label{eq:definition of Gelfand map}
\hat{a}\big(\germ sx\big)\defeq \qtrip asx\quad\mbox{for all }\germ sx\in \O_s.
\end{equation}
Then $\hat{a}$ is a section of $L_s$ and belongs to $\contz(L_s)$.
Moreover, the map $a\mapsto \hat{a}$ from $\A$ to $\cont$,
which shall henceforth be called the \emph{Gelfand map}, is an isomorphism of Fell bundles $\A\cong\cont$.
In particular, we have isomorphisms of imprimitivity Hilbert bimodules
$$_{\A_{ss^*}}{\A_s}_{\,\A_{s^*s}}\cong\,  _{\contz(\U_{ss^*})}\contz(L_s)_{\contz(\U_{s^*s})}\quad \mbox{for all }s\in S.$$
\end{theorem}
\begin{proof}
It is clear that $\hat{a}$ is a section of $L_s$ since the bundle projection $p\colon L\to \G$ is given by $p(\qtrip asx)=\germ sx$.
Moreover, by definition of the topology on $L$, all the sections $\hat{a}$ are continuous.
Note that
\begin{equation}\label{eq:Gelfand map is isometric}
\left\|\hat{a}\bigl(\germ{s}{x}\bigr)\right\|^2=\big\|\qtrip asx\big\|^2=(a^*a)(x).
\end{equation}
Since $a^*a\in \contz(\U_{s^*s})$, this implies that $\hat{a}$ vanishes at infinity, that is, $\hat{a}\in \contz(L_s)$ and
$$\|\hat{a}\|=\sup_{x\in \U_{s^*s}}\sqrt{(a^*a)(x)}=\|a^*a\|\half=\|a\|.$$
Thus the map $\A_s\ni a\mapsto \hat{a}\in\contz(L_s)$ is isometric.
It is obviously linear by definition of the linear structure on the fibers of $L$ (see Proposition~\ref{def:linear structure on the fibers of L}).
To prove its surjectivity, first take a section $\xi\in \contc(L_s)$ with support contained in $\O(s,\dom(a))$,
where $a$ is a fixed element of $\A_s$. Since $\hat{a}$ does not vanish on $\O(s,\dom(a))$ by Equation~\eqref{eq:Gelfand map is isometric},
the line bundle $L$ is trivializable over $\O(s,\dom(a))\cong\dom(a)$. More precisely, the restriction $L\rest{\O(s,\dom(a))}$
is isomorphic to $\C\times\O(s,\dom(a))\cong\C\times\dom(a)$ through the map
$\C\times\dom(a)\ni(\lambda,x)\mapsto \lambda\hat{a}(\germ sx)\in L$.
Therefore, there is a continuous function $h\colon\dom(a)\to \C$ such that
\begin{equation}\label{eq:section is trivializable}
\xi(\germ sx)=h(x)\hat{a}(\germ sx)\quad\mbox{for all }x\in \dom(a).
\end{equation}
Since $\xi$ is supported in $\O(s,\dom(a))$, $h$ belongs to $\contc(\dom(a))\sbe\contc(\U_{s^*s})$ and hence may
be viewed as an element of $\A_{s^*s}\cong\contz(\U_{s^*s})$. Moreover, in this way Equation~\ref{eq:section is trivializable}
holds for all $x\in\U_{s^*s}$ and from the (easily verified) relation $ah\eq x h(x) a$, we obtain
$$\widehat{ah}(\germ sx)=\qtrip{ah}{s}{x}= h(x)\qtrip asx= h(x)\hat{a}(\germ sx)=\xi(\germ sx).$$
Thus $\widehat{ah}=\xi$ and therefore the image of the Gelfand map
$\A_s\to \contz(L_s)$ contains all the functions with compact support contained in $\O(s,\dom(a))$.
Since the open subsets $\O(s,\dom(a))$ with $a\in \A_s$ cover $\O_s$, a partition-of-unit argument shows that any function
in $\contc(L_s)$ is in the image of the Gelfand map and therefore it is surjective.

We have already seen in the proof of Theorem~\ref{T:L is a Fell line bundle}
(see Equation~\eqref{E:relations of multiplication and involution for the Gelfand maps}) that
the Gelfand map preserves the Fell bundle multiplications and involutions, that is,
$$\widehat{a\cdot b}=\widehat{a}\cdot\widehat{b}\quad\mbox{and}\quad\widehat{a^*}
                                                    =\widehat{a}^*\quad\mbox{for all }s,t\in S, a\in \A_s, b\in \A_t.$$
Finally, we show that the Gelfand map preserves the inclusion maps $\A_s\into\A_t$ and $\contz(L_s)\into\contz(L_t)$ whenever $s\leq t$.
For this all we have to check is the following: if $a\in \A_s$ and we consider it as an element of $\A_t$ (so we are in fact identifying $\A_s\sbe\A_t$), then the function $\hat{a}$ vanishes outside $\O_s$. But as we have already observed above,
Equation~\ref{eq:Gelfand map is isometric} implies that $\hat{a}$ is supported in $\O(s,\dom(a))\sbe\O_s$.
\end{proof}

The Fell bundle $\CC=\{\contz(L_s)\}_{s\in S}$ over $S$ constructed in Proposition~\ref{prop:definition of the fell bundle of continuous sections}
may be also considered as a Fell bundle over the inverse semigroup
$$T=\{\O_s\colon s\in S\}\sbe S(\G).$$
The structure is basically the same. To avoid confusion, let us write $\B=\{\B_t\}_{t\in T}$ for
the Fell bundle $\cont$ considered over $T$. Thus, if $t=\O_s$, then
the fiber $\B_t$ is by definition $\contz(L_s)$, and the algebraic operations and inclusion maps are defined as in
Proposition~\ref{prop:definition of the fell bundle of continuous sections}. Note that $\cont$ is the pullback of $\B$ along the
map $\O\colon S\to T$, $s\mapsto \O_s$. The map $\O$ is a surjective homomorphism of inverse semigroups, but it is not injective in general.
The most trivial example is when the Fell bundle $\A=\{\A_s\}_{s\in S}$ is the \emph{zero} Fell bundle, that is, $\A_s=\{0\}$ for all $s\in S$.
In this case, the associated twisted groupoid $(\G,\Sigma)$ is the \emph{empty} groupoid, that is, $\G=\Sigma=\emptyset$. Thus
$S(\G)$ (and hence also $T$) is the inverse semigroup with just one element, the zero element (empty set): $S(\G)=T=\{0\}$.
Of course, the map $\O\colon S\to \{0\}$ is not injective since $S$ might be an arbitrary inverse semigroup.

If $\O$ is injective, then $\B$ and $\cont$ are isomorphic Fell bundles. Thus, by
Theorem~\ref{teo:characterization of saturated semi-abelian Fell bundles}, $\B$ is also isomorphic to the original Fell bundle $\A$
in this case. So, it is interesting to give conditions on $\A$ that imply the injectivity of $\O\colon S\to S(\G)$.
Before we go into this problem, let us prove that the \cstar{algebras} of $\A$ and $\B$ are always isomorphic:

\begin{proposition}
Let notation be as above. Then the Gelfand map from $\A$ to $\B$ induces an isomorphism $C^*(\A)\cong C^*(\B)$ which
restricts to an isomorphism $C^*(\E_\A)\cong C^*(\E_\B)\cong \contz(X)$, where $\E_\A=\A|_{E(S)}$ and $\E_\B=\B|_{E(T)}$.
\end{proposition}
\begin{proof}
Since the Gelfand map from $\A$ to $\CC=\{\contz(L_s)\}_{s\in S}$ is an isomorphism of Fell bundles,
it is enough to show that the canonical morphism from $\CC=\{\contz(L_s)\}_{s\in S}$ to $\B=\{\contz(L_t)\}_{t\in T}$
(consisting of the homomorphism $\O\colon S\to T=\{\O_s\colon s\in S\}$ and the identity maps $\contz(L_s)\to \contz(L_t)$
between the fibers whenever $t=\O_s$) induces an isomorphism $C^*(\cont)\cong C^*(\B)$ which restricts to an isomorphism
$C^*(\E_\cont)\cong C^*(\E_\B)$. The canonical morphism $\cont\to \B$ induces the map $\Rep(\B)\to \Rep(\cont)$
which assigns to a representation
  (see \cite[Definition 3.1]{Exel:noncomm.cartan})
$\pi=\{\pi_t\}_{t\in T}$ of $\B$, the representation $\tilde\pi=\{\tilde\pi_s\}_{s\in S}$
with $\tilde\pi_s=\pi_{\O_s}$ for all $s\in S$. The induced \Star{homomorphism} $C^*(\cont)\to C^*(\B)$ is automatically surjective,
and to prove its injectivity, we show that every representation $\rho=\{\rho_s\}_{s\in S}$ of $\cont$ is equal to $\tilde\pi$
for some (necessarily unique) representation $\pi\in \Rep(\B)$. In fact, all we have to show is that, for all $r,s\in S$,
\begin{equation}\label{eq:pi does not depend on s, only on Os}
\pi_r(f)=\pi_s(f)\mbox{ whenever }\O_r=\O_s\mbox{ and }f\in \contz(L_r)=\contz(L_s).
\end{equation}
The equality $\O_r=\O_s$ implies that $\U_{r^*r}=\U_{s^*s}$,
$\U_{rr^*}=\U_{ss^*}$ and $\t_r=\t_s$.
Moreover, given $x\in \U_{r^*r}=\U_{s^*s}$, we have $\germ rx=\germ sx$ because $\O_r=\O_s$.
Hence, there is $e\in E(S)$ with $x\in \U_e$ and $re=se$. Multiplying $e$ by $(r^*r)(s^*s)$, we may assume that $e\leq (r^*r)(s^*s)$,
that is, $e\leq r^*r$ and $e\leq s^*s$. Defining $E_{r,s}=\{e\in E(S)\colon re=se, e\leq (r^*r)(s^*s)\}$, we conclude that
$$\U_{r^*r}=\U_{s^*s}=\bigcup\limits_{e\in E_{r,s}}\U_e.$$
As a consequence, we get
\begin{equation}\label{eq:equality of continuous functions}
\contz(\U_{r^*r})=\contz(\U_{s^*s})=\cspn_{e\in E_{r,s}}\contz(\U_e).
\end{equation}
Now, given $e\in E_{r,s}$, $f\in \contz(L_r)=\contz(L_s)$ and $h\in \contz(\U_e)$, we have
$$\rho_r(f)\rho_e(h)=\rho_{re}(fh)=\rho_{se}(fh)=\rho_s(f)\rho_e(h).$$
Since $e\leq r^*r$ and $\rho$ is a representation, we have $\rho_e(h)=\rho_{r^*r}(h)$, so that
$$\rho_r(f)\rho_e(h)=\rho_r(f)\rho_{r^*r}(h)=\rho_r(fh).$$
Analogously, $\rho_s(f)\rho_e(h)=\rho_s(fh)$. We conclude that
$\rho_r(fh)=\rho_s(fh)$ for all $h\in \contz(\U_e)$ with $e\in E_{r,s}$, and by Equation~\eqref{eq:equality of continuous functions}
this also holds for every $h$ in $\contz(\U_{r^*r})=\contz(\U_{s^*s})$. This is enough to prove~\eqref{eq:pi does not depend on s, only on Os}
because by Cohen's Factorization Theorem, every element of $\contz(L_r)=\contz(L_s)$ is a product of the form $fh$ with
$f\in\contz(L_r)=\contz(L_s)$ and $h\in \contz(\U_{r^*r})=\contz(\U_{s^*s})$. Therefore we get an isomorphism $C^*(\cont)\cong C^*(\B)$.
Its restriction to the idempotent parts gives an injective \Star{homomorphism} $C^*(\E_\cont)\to C^*(\E_\B)$ (which is
the identity map on the fibers). To see that it is surjective, suppose $s\in S$ and $f=\O_s$ is idempotent in $T$.
Although $s$ is not necessarily idempotent in $S$, we must have $f=\O_s=\O_s^*\O_s=\O_{s^*s}$.
Since the morphism $\E_\cont\to \E_\B$ maps $\CC_{s^*s}=\contz(L_{s^*s})$ onto $\B_f=\contz(L_s)=\contz(L_{s^*s})$
(it is just the identity map), the surjectivity of the induced map $C^*(\E_\cont)\to C^*(\E_\B)$ follows, and therefore
$C^*(\E_\cont)\cong C^*(\E_\B)\cong\contz(X)$, as desired.
\end{proof}

Let us now return to the injectivity problem of the map $\O\colon S\to S(\G)$.

\begin{definition}
Let $\A=\{\A_s\}_{s\in S}$ be a Fell bundle over an inverse semigroup $S$. Let $C^*(\A)$ be the (full) cross-sectional \cstar{algebra} of $\A$,
and let $\pi_u\colon\A\to C^*(\A)$ be the universal representation of $\A$. We say that $\A$ is \emph{faithful} if the
map $s\mapsto \pi_u(\A_s)$ is injective, that is, $\pi_u(\A_s)=\pi_u(\A_t)$ if and only if $s=t$. We say that $\A$ is \emph{semi-faithful}
if the restriction $\E_\A=\A\rest{E}$ of $\A$ to the semilattice of idempotents $E=E(S)$ is faithful.
\end{definition}

Recall from \cite{Exel:tight.representations} that an inverse semigroup $S$ with zero is said to be \emph{continuous}
if $s\equiv t$ implies $s=t$, where $\equiv$ is the following equivalence relation:
\begin{multline*}
s\equiv t\Longleftrightarrow s^*s=t^*t\mbox{ and for any nonzero idempotent }f\leq s^*s,\\
\mbox{ there is a nonzero idempotent }e\leq f\mbox{ with }se=te.
\end{multline*}

\begin{proposition}\label{prop: semi-faithfulness and continuity => injectivity of s->Os}
Let $\A=\{\A_s\}_{s\in S}$ be a saturated, semi-abelian Fell bundle over an inverse semigroup $S$ with zero element $0$
such that $\A_0=\{0\}$, and let $L$ be the associated Fell line bundle. If $S$ is continuous and $\A$ is semi-faithful,
then the map $s\mapsto\O_s$ from $S$ to the inverse semigroup $T=\{\O_s\colon s\in S\}\sbe S(\G)$ is injective.
Hence $\A=\{\A_s\}_{s\in S}$ and $\B=\{\contz(L_t)\}_{t\in T}$ are isomorphic Fell bundles.
\end{proposition}
\begin{proof}
Recall that $\O_s=\{\germ sx\colon x\in \U_{s^*s}\}$ for all $s\in S$.
If $s,t\in S$ are such that $\O_s=\O_t$, then $\U_{s^*s}=\U_{t^*t}$, $\U_{ss^*}=\U_{tt^*}$ and $\t_s=\t_t$.
In particular $\A_{s^*s}\cong\contz(\U_{s^*s})=\contz(\U_{t^*t})\cong A_{t^*t}$ in $\contz(X)\cong C^*(\E)$.
Since $\A$ is semi-faithful, we get $s^*s=t^*t$. Moreover, from the semi-faithfulness we have $\A_e=\{0\}$ if and only if $e=0$.
Now, take any nonzero idempotent $f\leq s^*s$. Then $\emptyset\not=\U_f\sbe\U_{s^*s}=\U_{t^*t}$.
Thus, if $y\in \U_f$, then $\germ sx=\germ tx$, and hence there is $e\in E(S)$ with $se=te$ and $x\in \U_e$.
In particular, $\U_e\not=\emptyset$ so that $\A_e\not=\{0\}$ and hence $e$ is a nonzero idempotent.
The product $g=ef$ is a nonzero idempotent because $x\in \U_{g}=\U_e\cap\U_f$, and we have $g\leq f$ and $sg=tg$.
Hence $s\equiv t$ and the continuity of $S$ implies $s=t$.
\end{proof}

\begin{remark}
The hypothesis $\A_0=\{0\}$ is not necessary for the injectivity of $s\mapsto \O_s$. As a simple example, consider
$S=\{0,1\}$ (which is a semilattice and therefore a continuous inverse semigroup). Define $\A_0=\C$ and $\A_1=\C\times\C$.
With the inclusion $\C\into \C\times\{0\}\sbe \C\times\C$ and the canonical algebraic operations (inherited from $\C\times\C$),
$\A$ is a semi-abelian, saturated Fell bundle. Note that $C^*(\A)=C^*(\E)\cong \C\times \C\cong\cont(\{x_0,x_1\})$, where $x_0,x_1$
are two distinct points. With these identifications, we may say that $\U_0=\{x_0\}$ and $\U_1=\{x_0,x_1\}=X$.
Note that $\G\cong X$ and $\Sigma\cong \Torus\times X$ in this case (this happens whenever $S$ is a semilattice).
Moreover, we may identify $\O_0\cong\U_0$ and $\O_1\cong\U_1$. Thus the map $s\mapsto \O_s$ is injective.
\end{remark}

\begin{example}
Consider a Fell bundle $\A$ over a discrete group $G$ with unit $1$.
Rigorously, $G$ is not continuous as an inverse semigroup because it has no zero element.
However, this is the only problem. One may add a zero element $0$ to $G$ turning it into a continuous inverse semigroup $S=G\cup \{0\}$ (not a group anymore, of course), and extend $\A$ to a Fell bundle $\tilde\A$ over $S$ simply defining $\tilde\A_0=\{0\}$. Note that $\A$ is saturated and
semi-abelian if and only if $\tilde\A$ is. Moreover, if $\A$ is saturated, then the following assertions are equivalent:
\begin{itemize}
\item $\A$ is the zero Fell bundle;
\item $\A_1=\{0\}$;
\item there is $g\in G$ with $\A_g=\{0\}$.
\end{itemize}
In fact, if $g\in G$ is such that $\A_g=\{0\}$, then $\A_1=\A_g^*\A_g=\{0\}$. And if $\A_1=\{0\}$, then $\A_g^*\A_g=\A_1=\{0\}$, so that
$\A_g=\{0\}$ for all $g\in G$. For groups, the zero Fell bundle is the only case where the injectivity of $s\mapsto \O_s$ fails
(unless $G=\{1\}$ is the trivial group). If fact, if $\A$ is a nonzero, semi-abelian, saturated Fell bundle over $G$,
then all the fibers $\A_g$ are nonzero by the equivalences above. It follows that $\tilde\A$ is faithful.
In particular, it is semi-faithful, so we may apply Proposition~\ref{prop: semi-faithfulness and continuity => injectivity of s->Os}
to conclude that $s\mapsto\O_s$ is injective from $S$ to $S(\G)$. In particular, its restriction from $G$ to $S(\G)$ is also injective.
One may also prove this directly in the group case: the associated groupoid $\G$ is the transformation groupoid $\G\cong G\ltimes_\t X$,
which as a topological space is just $G\times X$ and the groupoid operations are $\s(s,x)=x$, $\r(s,x)=\t_s(x)$,
$(s,x)\cdot (t,y)=(st,y)$ whenever $\t_t(y)=x$, and $(s,x)\inv=(s\inv,\t_s(x))$. Bissections of this groupoid are sets of the form $\{s\}\times U$
where $U\sbe X$ is some open subset. Moreover, the bissection $\O_s$ is just $\{s\}\times X$
(here we are assuming $\A_s\not=\{0\}$ for all $s\in G$; otherwise we have $\O_s=\emptyset$ for all $s\in G$).
Hence the map $s\mapsto \O_s$ is clearly injective.
\end{example}

\section{Isomorphism of reduced algebras}\label{sec:Isomorphism of reduced algebras}

\subsection{The regular representation} In this section we will study
reduced \cstar{algebras} of Fell bundles over inverse semigroups and their
relationship to the reduced \cstar{algebra} of the associated
Fell line bundle.

We begin by briefly recalling some facts about the reduced \cstar{algebra}
of Fell line bundles over (non necessarily Hausdorff) \'etale
groupoids, referring the reader to \cite{RenaultThesis} for more
details, such as the construction of the reduced \cstar{algebra} in the
non-\'etale case, although Renault only treats the Hausdorff case.

Throughout this section we suppose we are given an étale groupoid $\G$ and a Fell
line bundle $L$ over $\G$.

For the time being we will also fix $x\in \Gz$.  Observe that $L_x$ is a one-dimensional \cstar{algebra}
which is therefore isomorphic to $\C$,  so we will henceforth tacitly identify
$L_x$ with $\C$.

Denoting by $\G_x=\s\inv(x)$,
let $\Hx$ be the collection of all square-summable sections of $L$ over $\G_x$, that is
all functions
  $\xi\colon\G_x\to \dot{\bigcup}_{\g\in \G_x}L_\g$, such that $\xi(\g)\in L_\g$, for all
$\g\in \G_x$, and such that
\begin{equation}\label{Eq:LthoBounded}
\sum_{\g\in \G_x} \xi(\g)^*\xi(\g) <\infty.
\end{equation}
  In regards to this sum notice that $\xi(\g)^*\xi(\g) \in L_{\g^*\g}
= L_x$, which we are identifying with $\C$, as mentioned above.

It is well known that $\Hx$ becomes a Hilbert space with inner
product
\begin{equation}\label{Eq:DefInnProd}
  \braket{\xi}{\eta} = \sum_{\g\in \G_x} \xi(\g)^*\eta(\g)
  \quad\mbox{for all } \xi,\eta\in \Hx.
\end{equation}

\bigskip

Recall from Definition~\ref{def:definition of C_c(B)} that $\contc(L)$ denotes the space of all sections of the form $\sum_{i=1}^n f_i$
where each $f_i$ is a compactly supported, continuous local section $f_i\colon U_i\to L$ over some open Hausdorff subset $U_i\sbe\G$
(which can be taken to be a bissection of $\G$), extended by zero outside $U_i$ and viewed as a global section $f_i\colon\G\to L$.

\begin{proposition}\label{DefinePix}
For every $f\in \contc(L)$ there exists a bounded linear operator $\pi_x(f)$ on $\Hx $ such that
$$
  \pi_x(f)\xi\calcat \gamma =
  \sum_{\g_1\g_2=\g} f(\g_1)\xi(\g_2)
  \quad\mbox{for all } \xi\in \Hx, \g\in\G_x.
$$
\end{proposition}

\begin{proof} Before we begin observe that any given summand
``$f(\g_1)\xi(\g_2)$'' above lies in the same fiber of $L$, namely
$L_{\g_1}L_{\g_2}=L_{\g_1\g_2} = L_\g$.

Let us now argue that the sum in the statement does indeed converge
by showing that only finitely many summands are nonzero. We begin by
treating the case in which $f$ is supported on a given bissection $U$
of $\G$.

Given $\g\in \G_x$ suppose first that
$\r(\g)\notin\r(U)$.  Therefore there is no $\g_1$ in $U$ such that
$\r(\g)=\r(\g_1)$, and hence the above sum admits no nonzero summand,
and $\pi_x(f)\xi\calcat \gamma =0$.

On the other hand, if $\r(\g)\in\r(U)$, then $\r(\g)=\r(\g_1)$, for
some $\g_1\in U$, which is unique, given that $U$ is a
bissection.  Setting $\g_2=\g_1\inv\g$, we see that $(\g_1,\g_2)$ is the
unique pair satisfying $\g_1\g_2=\g$, with $\g_1\in U$, and hence
\begin{equation}\label{Eq:SumWithOneComponent}
  \pi_x(f)\xi\calcat \gamma =f(\g_1)\xi(\g_2).
\end{equation}
  This shows that the sum in the statement converges as it has at most
one nonzero summand.

Still supposing that $f$ is supported on the bissection $U$, let us
prove that $\pi_x(f)$ is well-defined and bounded.  In order to do
this we claim that the correspondence $\g \mapsto \g_2$, defined
as above
for $\g\in \G_x\cap \r(U)$, is injective. In fact, suppose
that $\g$ and $\g'$ lie in $\G_x\cap \r(U)$ and that both lead up
to the same $\g_2$.  This means that there are $\g_1$ and $\g_1'$ in
$U$ such that
  $$
  \g=\g_1\g_2 \quad\mbox{and}\quad   \g'=\g_1'\g_2.
  $$
  Therefore $\s(\g_1) = \r(\g_2) = \s(\g_1')$, which implies that
$\g_1=\g_1'$, again because $U$ is a bissection.  This obviously gives
$\g=\g'$, concluding the proof of our claim.

Viewing $\g_2$ as a function of $\g$, as above, we then
have for all $\xi\in \Hx $,
  \def\bsum{\kern-10pt\sum_{\g\in \G_x\cap \r(U)}\kern-10pt}
\begin{multline*}
  \sum_{\g\in\G_x}\big\|\big(\pi_x(f)\xi\big)(\gamma)\big\|^2 =
  \bsum\big\|\big(\pi_x(f)\xi\big)(\gamma)\big\|^2 \\=
  \bsum\big\|f(\g\g_2\inv)\xi(\g_2)\big\|^2 \leq
  \|f\|_\infty^2\bsum\big\|\xi(\g_2)\big\|^2 \leq
  \|f\|_\infty^2\|\xi\|^2,
\end{multline*}
  where the last inequality holds because $\g_2$ is an injective
function of $\g$.  This shows that $\|\pi_x(f)\xi\|\leq
\|f\|_\infty\|\xi\|$, and hence that $\pi_x(f)$  is well-defined and
bounded with
$\|\pi_x(f)\|\leq\|f\|_\infty$.

In order to treat the general case, let $f\in \contc(L)$.
Then we may write $f$ as a finite sum $f=\sum_{i=1}^n f_i$,
where each $f_i$ is supported in some bissection, in which case it
is clear that $\pi_x(f) = \sum_{i=1}^n \pi_x(f_i)$, and we see that
$\pi_x(f)$ is indeed bounded.
\end{proof}

It is now easy to see that the correspondence \ $f\mapsto \pi_x(f)$ \ is a
*-representation of $\contc(L)$, and hence extends continuously to a
representation, by abuse of language also denoted $\pi_x$, of
$C^*(L)$ on $\Hx$.

For each $\g\in\G_x$ choose a unit vector $v_\gamma$ in
$L_\gamma$ and set
  $$
  \delta_\gamma =
  \left\{\begin{array}{cc}
    v_\gamma, & \hbox{if } \g'=\g,\hfill \cr\cr
    0, & \hbox{otherwise.}
    \end{array}
  \right.
  $$

As we shall see, the random choice of $v_\gamma$ above will have
little, if any effect in what follows.  It is then easy to see that
$\{\delta_\g\}_{\g\in\G_x}$ is an orthonormal basis of $\Hx $.

Among the elements of $\G_x$ one obviously finds $x$ itself, so
$\delta_x$ is one of our basis elements.

\begin{proposition}\label{PiXCyclic}
For every $x\in \Gz$ one has $\delta_x$ is a cyclic vector for $\pi_x$.
\end{proposition}
\begin{proof} Let $\g'\in\G_x$, and let $U$ be a bissection of $\G$
containing $\g'$.  Choose $f\in \contc(L)$ supported on $U$
and such that $f(\g')\neq0$.  We claim that $\pi_x(f)\delta_x$ is a nonzero
multiple of $\delta_\g'$.
  In order to see this, suppose that $\g\in \G$ is such that
  $$
  \pi_x(f)\delta_x\calcat\g\neq0.
  $$
  By definition there exists at least one pair $(\g_1,\g_2)$ such that
$\g_1\g_2=\g$, and $f(\g_1)\delta_x(\g_2)\neq 0$.  This obviously implies
that $\g_1\in U$,  and $\g_2=x$.  In particular this says that $\g_2\in\Gz$
and hence $\g_1=\g$,  so we deduce that $\g\in U$.  In
addition
  $$
  \s(\g) = \s(\g_1)=\r(\g_2) = x.
  $$
  It follows that the source of both $\g$ and $\g'$ coincide with $x$,
and that both $\g$ and $\g'$
lie in $U$.  Since $U$ is a bissection we conclude that $\g=\g'$,
thus showing that $\pi_x(f)\delta_x\calcat\g$ vanishes whenever
$\g\neq\g'$.

In order to compute the value of $\pi_x(f)\delta_x\calcat{\g'}$ one
observes that $\g'=\g'x$, so the unique pair $(\g_1,\g_2)$ with
$\g_1\g_2=\g'$ and
$\g_1\in U$, according to Equation~\eqref{Eq:SumWithOneComponent}, is
$(\g_1,\g_2)=(\g',x)$.  We then have
  $$
  \pi_x(f)\delta_x\calcat {\g'} = f(\g')\delta_x(x).
  $$
  Since this is nonzero we conclude that $\pi_x(f)\delta_x$ is indeed a
nonzero multiple of $\delta_{\g'}$.  This shows that any $\delta_{\g'}$ lies
in the cyclic space spanned by $\delta_x$, and hence that $\delta_x$ is a
cyclic vector for $\pi_x$, thus concluding the proof.
\end{proof}

\begin{proposition}\label{PiXCyclicState}
  Given $x\in \Gz$ let $\phi_x$ be the state associated to the
representation $\pi_x$ and the cyclic vector $\delta_x$, namely
  $$
  \phi_x(f) = \braket{\pi_x(f)\delta_x}{\delta_x}
  $$
for all $f\in \contc(L)$.  Then $\phi_x(f)= f(x)$.
\end{proposition}
\begin{proof}
  It is clearly enough to prove the statement under the assumption
that $f$ is supported on a bissection $U$ of
$\G$.  We begin by claiming that
\begin{equation}\label{Eq:CalculoComLotsaX}
  \pi_x(f)\delta_x\calcat x = f(x)\delta_x(x).
\end{equation}
  Suppose first that
  $\r(x)\in\r(U)$.  So
  we have by Equation~\eqref{Eq:SumWithOneComponent} that
  $
  \pi_x(f)\delta_x\calcat x = f(\g_1)\delta_x(\g_2),
  $
  where $(\g_1,\g_2)$ is the unique pair of elements in $\G$ such that
$\g_1\g_2=x$, and $\g_1\in U$.

In order to find $\g_1$ and $\g_2$, recall that $\r(x)\in\r(U)$, so
there exists $\g\in U$ such that $\r(\g)=\r(x)=x$.  Then $\g\g\inv=x$,
and we see that $(\g_1,\g_2)=(\g,\g\inv)$.  Using brackets to indicate
boolean value we have
  $$
  \pi_x(f)\delta_x\calcat x =
  f(\g)\delta_x(\g\inv) =
  [x{=}\g\inv]\ f(\g)\delta_x(\g\inv) =
  [x{=}\g]\ f(x)\delta_x(x).
  $$
Notice that the last term above equals $f(x)\delta_x(x)$.
While this is obvious when $f(x)=0$, notice that if $f(x)\neq0$ we
must have $x\in U$, and then the unique element $\g\in U$ with
$\r(\g)=x$ is $x$ itself, so $x=\g$, or equivalently $[x{=}\g]=1$.

This proves Equation~\eqref{Eq:CalculoComLotsaX} under the assumption that
$\r(x)\in\r(U)$, so suppose now that
$\r(x)\notin\r(U)$.  In this case we have already seen that
  $
  \pi_x(f)\delta_x\calcat x
  $
  vanishes, so it is enough to check that the right hand side of
Equation~\eqref{Eq:CalculoComLotsaX} also vanishes.  But this is immediate since
otherwise $x\in U$, whence $\r(x) \in\r(U)$.

Having finished the proof of Equation~\eqref{Eq:CalculoComLotsaX} we have
  $$
  \phi_x(f)=
  \braket{\pi_x(f)\delta_x}{\delta_x} =
  \Big(\pi_x(f)\delta_x\calcat x\Big) \overline{\delta_x(x)}
\overeq{\eqref{Eq:CalculoComLotsaX}}
  f(x)\delta_x(x) \overline{\delta_x(x)} =
  f(x),
  $$
  because $\delta_x(x)$ is a unit vector by construction.
\end{proof}

\subsection{The isomorphism}
  Let $\A=\{\A_s\}_{s\in S}$ be a semi-abelian Fell bundle over the
inverse semigroup $S$ and let $\E$ be the restriction of $\A$ to the
idempotent semilattice $E(S)$.  Therefore $C^*(\E)$ is an abelian
\cstar{algebra} whose spectrum will be denoted by $X$, so $C^*(\E)$ is
isomorphic to $\contz(X)$.

For each $e\in E(S)$ we will identify $\A_e$ as a closed two-sided
ideal in $\contz(X)$, and therefore there exists an open set $\U_e\sbe X$,
such $\A_e=\contz(\U_e)$.
By Proposition~\ref{prop:theta is an action of S}, for each $s\in S$, there exists a homeomorphism
  $$
  \t_s\colon\U_{s^*s} \to \U_{ss^*},
  $$
  such that for every $f\in \contz(\U_{ss^*})$, and every $a_s\in \A_s$,
  $$
  \big(a_s^*fa_s\big)x = (a_s^*a_s)(x)\; f\big(\t_s(x)\big)
  \quad\mbox{for all } f\in \contz(\U_{ss^*}), a_s\in \A_s, x\in \U_{s^*s}.
  $$
Moreover, $s\mapsto\t_s$ gives an action of $S$ on $X$.
Let $\G$ be the groupoid of germs for this action as in Section~\ref{sec:the construction of G} and let
$L$ be the Fell line bundle over $\G$ associated to $\A$ as constructed in Section~\ref{sec:construction of L}.

It is our goal in this section to prove that the reduced \cstar{algebra} of
$\A$ is isomorphic to the reduced \cstar{algebra} of $L$.

\begin{lemma}\label{TrivialGerm}
A necessary and sufficient condition for a given element $\germ sx$
in $\G$ to lie in $\Gz$ is that there exists $e\in E(S)$ such that
$e\leq s$, and $x\in \U_e$.
\end{lemma}

\begin{proof}
  Given $e$ as in the statement notice that $se=e=ee$,
  and hence in view of the equivalence relation leading up to the notion of germs, we have
  $$
  \germ sx = \germ ex\in \Gz.
  $$
Conversely, suppose that $\germ sx\in \Gz$.  Then
  $$
  \germ sx =\germ sx\inv \germ sx = \germ{s^*}{\t_s(x)}\germ sx = \germ{s^*s}{x}.
  $$
  Therefore, there exists $f\in E(S)$ such that $x\in \U_f$, and $sf = s^*sf$. Setting $e=s^*sf$, we get
  $$
  \U_e = \U_{s^*s}\cap \U_f \ni x,
  $$
  and
  $$
  se = ss^*sf = sf = s^*sf  = e,
  $$
  so that $e\leq s$, and the proof is complete.
\end{proof}

Given a pure state $\phi$ on $C^*(\E)=\contz(X)$, we wish to identify the
state $\tilde\phi$ on $C^*(\A)$ described in \cite[Proposition~7.4]{Exel:noncomm.cartan}.
Since the pure states on $\contz(X)$ are precisely the point evaluations,
there must exist some $x_0\in X$, such that
  $$
  \phi(f) = f(x_0)
  \quad\mbox{for all } f\in \contz(X).
  $$
Recall from \cite[Equation~(7.3)]{Exel:noncomm.cartan}  that, given $e\in E(S)$, we let $\phi_e$ be the state
on $\A_e=\contz(\U_e)$ given by restriction of $\phi$.  Evidently
\begin{equation}\label{Eq:CharacSupport}
  \phi_e\neq 0 \iff x_0\in \U_e.
\end{equation}

Using \cite[Proposition~5.5]{Exel:noncomm.cartan}, the above can be used as a
characterization of when is $\phi$ supported on $\A_e$.

If $s\in S$ and $e\in E(S)$ are such that $e\leq s$ and $\phi_e\neq
0$, then $\phi_e$ has a canonical extension $\tilde\phi_e^s$ to $\A_s$
given by \cite[Proposition~6.1]{Exel:noncomm.cartan}.  In order to describe it, choose
$h\in \A_e$ such that $h(x_0)=1$.  Then it is clear that
  $$
  \phi_e(f) = \phi_e(h)\phi_e(f) = \phi_e(hf),
  $$
  so, by \cite[Proposition~6.3]{Exel:noncomm.cartan} we have
  $$
  \tilde\phi_e^s(a)= \phi_e(ha) =(ha)(x_0)
  \quad\mbox{for all } a\in \A_s.
  $$
  In order to compute the expression $(ha)(x_0)$, we use the
isomorphism $\A_s\cong\contz(L_s)$, that is, the Gelfand map constructed
in Theorem~\ref{teo:characterization of saturated semi-abelian Fell bundles},
which we have so far also used implicitly in identifying $\A_e = \contz(L_e) =
\contz(\U_e)$. Thus, we are going to identify $a\in \A_s$ with $\hat{a}\in \contz(L_s)$.
Under our identification we have
\begin{align*}
  (ha)(x_0) &= \widehat{ha}\big(\germ{e}{x_0}\big) = \hat{h}\hat{a}\big(\germ{e}{x_0}\big) \\
  &=\hat{h}\big(\germ{e}{x_0}\big)\hat{a}\big(\germ{s}{x_0}\big) = \cdots
\end{align*}
  because the only way of writing $\germ{e}{x_0}$ as a product of
elements in $\O_e$ and $\O_s$ is
  $$
  \germ{e}{x_0} =   \germ{e}{x_0}   \germ{s}{x_0}.
  $$
  Continuing with our computation of $(ha)(x_0)$ above, we have
  $$
  \cdots =
  1\cdot
  \hat{a}\big(\germ{s}{x_0}\big) =
  \hat{a}\big(\germ{e}{x_0}\big),
  $$
  where the last equality is simply due to the fact that $\germ{s}{x_0} = \germ{e}{x_0}$.  Returning with the identification between $X$
and $\Gz$, we may then write
  $$
  \tilde\phi_e^s(a) = a(x_0).
  $$

Suppose, on the other hand that $s$ is such that there is
no $e$ in $\supp(\phi)$ with $e\leq s$.  By the characterization of
$\supp(\phi)$ given in Equation~\eqref{Eq:CharacSupport}, we deduce that there is
no $e$ in $E(S)$ such that $x_0\in \U_e$ and $e\leq s$.
  Choosing any idempotent $f$ such that $x\in U_f$ we then claim that
$\germ{f}{x_0}\notin \O_s$.  In order to prove this suppose the contrary
and hence $\germ{f}{x_0}=\germ{s}{x}$, for some $x\in \U_{s^*s}$.
This would imply $x=x_0$ and the existence of $e\in E(S)$,
with $x_0\in \U_e$ and $fe=se$.  Setting $e'=fe$ we would then
have $x_0\in U_{e'}$ and $e'\leq s$, a situation which has been explicitly ruled out by our hypothesis.
Therefore $\germ{f}{x_0}\notin \O_s$ and hence any $f\in \contc(L_s)$
vanishes on $\germ{f}{x_0}$.  In particular, for any $a_s\in \A_s$, we
have
  $$
  \widehat{a_s}(\germ{f}{x_0})=0 =
  \tilde\phi(a_s).
  $$
  We have therefore proved the following:

\begin{proposition}\label{FoundPhiTilde}
  Let $x_0\in X$ and let $\phi$ be the pure state on $\contz(X)$ given by
evaluation on $x_0$.  Then the canonical extension $\tilde\phi$ of
$\phi$ to $C^*(\A)$ given by \cite[Proposition~7.4]{Exel:noncomm.cartan} is such that
  $$
  \tilde\phi(a_s) = \widehat{a_s}(x_0)
  \quad\mbox{for all } s\in S \mbox{ and } a_s\in \A_s.
  $$
\end{proposition}

Consider the canonical inclusion from $\contc(L_s)$ into $\contc(L)\sbe C^*(L)$.
An argument similar to that given in \cite[Proposition
3.14]{Exel:inverse.semigroups.comb.C-algebras} shows that this
inclusion is continuous for the sup-norm on $\contc(L_s)$, and hence
it
extends to
$\contz(L_s)\to C^*(L)$. Moreover, these maps form a representation of the Fell bundle $\{\contz(L_s)\}_{s\in S}$ into $C^*(L)$.
Using the Gelfand isomorphisms $\A_s\cong\contz(L_s)$ (see Theorem~\ref{teo:characterization of saturated semi-abelian Fell bundles}),
we get a representation of $\A$ into $C^*(L)$, which therefore integrates to a (surjective) \Star{}homomorphism
$$\Psi\colon C^*(\A)\to C^*(L).$$
In fact, this is essentially the same \Star{}homomorphism appearing in Theorem~\ref{teor:fell bundles over groupoids and ISG isomorphic}
for the case $\B=L$ (which is therefore an isomorphism if $\G$ is Hausdorff or second countable).

\begin{theorem}\label{T:A and L have same reduced C*-algebras}
The homomorphism $\Psi\colon C^*(\A) \to C^*(L)$ above factors through the corresponding
reduced \cstar{algebras} providing an isomorphism
  $$
  \Psi_r\colon\CstarRed(\A) \to \CstarRed(L).
  $$
\end{theorem}
\begin{proof} For $x_0\in X$ denote by $\pi_{x_0}$ the representation of
$\CstarRed(L)$ given by Proposition~\ref{DefinePix}.
On the other hand,
let $\tilde\phi$ be the state on $C^*(\A)$ mentioned in
Proposition~\ref{FoundPhiTilde} in terms of $x_0$, and let $\rho_{x_0}$ be the
GNS representation of $C^*(\A)$ associated to $\tilde\phi$.

We claim that $\pi_{x_0}\circ\Psi$ is a representation equivalent to
$\rho_{x_0}$.  Since $\delta_{x_0}$ is a cyclic vector for $\pi_{x_0}$ by Proposition~\ref{PiXCyclic},
we see that it is also cyclic for $\pi_{x_0}\circ\Psi$ simply because $\Psi$ is onto.  The
associated vector state is given, on any $a_s\in \A_s$, by
  $$
  \braket{\pi_{x_0}(\Psi(a_s))\delta_{x_0}}{\delta_{x_0}} \overeq{(\ref{PiXCyclicState})}
  \Psi(a_s)(x_0) \overeq{(\ref{FoundPhiTilde})}
  \tilde\phi(a_s).
  $$
  The claim therefore follows from the uniqueness of the GNS
representation.   With respect to the respective reduced norms $\|\cdot\|_r$, we
then have, for every $a\in C^*(\A)$,
  $$
  \|\Psi(a)\|_r =
  \sup_{x_0\in X} \|\pi_{x_0}(\Psi(a))\| =
  \sup_{x_0\in X} \|\rho_{x_0}(a)\| = \|a\|_r,
  $$
  from where the result readily follows.
\end{proof}

Let $(\G,\Sigma)$ be the twisted groupoid associated to $\A$
as in Section~\ref{sec:twisted groupoids and Fell line bundles}.
Since the reduced \cstar{algebra} of $(\G,\Sigma)$ is, by definition, the
reduced \cstar{algebra} of $L$,  the above result also gives a canonical isomorphism $$\CstarRed(\A)\cong \CstarRed(\G,\Sigma).$$

\section{Application: Cartan subalgebras}

In this section we will apply the results so far developed to prove
part of Renault's Theorem \cite{RenaultCartan} on the
characterization of Cartan subalgebras of \cstar{algebras}.

Based on previous work by Vershik, Feldman, Moore, and Kumjian
\cite{VershikCartan,feldman_more:Cartan.subalgebrasI,feldman_more:Cartan.subalgebrasII,Kumjian:cstar.diagonals},
Renault gave the following:

\begin{definition} \label{CartanRenault}
  {\rm (\cite[Definition~5.1]{RenaultCartan})}
  A closed \Star{subalgebra} $B$ of a separable \cstar{algebra} $A$ is a \emph{Cartan subalgebra} if
\begin{enumerate}
  \item[\textup{(i)}] $B$ contains an approximate unit of $A$,
  \item[\textup{(ii)}] $B$ is maximal abelian in $A$,
  \item[\textup{(iii)}] $B$ is regular in the sense that the normalizer of $B$ in $A$,
namely
  $$
  N(B) = \{a\in A\colon aBa^*\subseteq B,\; a^*Ba\subseteq B\},
  $$
  generates $A$, and
  \item[\textup{(iv)}] there exists a faithful conditional expectation from $A$ to $B$.
\end{enumerate}
\end{definition}
Renault has proved \cite[Theorem~5.6]{RenaultCartan} that, whenever $B$ is a
Cartan subalgebra of $A$, there exists a twisted, essentially principal,
étale, Hausdorff groupoid $(\G, \Sigma)$, such that $A$ is isomorphic
to $\CstarRed(\G,\Sigma)$ via an isomorphism which carries $B$
onto $\contz(\Gz)$.

In \cite{Exel:noncomm.cartan} the second named author has recently introduced
a generalization of the notion of Cartan subalgebras to include
situations in which the \emph {maximal abelian algebra $B$ is no
longer abelian}.  To describe this result we need to recall that,
given a closed \Star{subalgebra} $B$ of a \cstar{algebra} $A$, a \emph{virtual
commutant of $B$ in $A$} is a $B$--bimodule map
  $$
  \phi\colon J \to A,
  $$
  where $J$ is an ideal in $B$.  Virtual commutants are akin to
elements in the relative commutant $B'\cap A$, since given any
$a\in B'\cap A$, the map
  $$
  \phi_a\colon b\in B \mapsto ab\in A,
  $$
  is a virtual commutant with domain $J=B$.

\begin{definition} \label{DefineGeneralized}
  {\rm (\cite[Definition~12.1]{Exel:noncomm.cartan})}
  A closed \Star{subalgebra} $B$ of a separable \cstar{algebra} $A$ is said to be a
\emph{generalized Cartan subalgebra} if it satisfies (i), (iii) and
(iv) of Definition~\ref{CartanRenault} and, instead of (ii), it is
required that
\begin{enumerate}
  \item[\textup{(ii)'}] every virtual commutant of $B$ in $A$ has its
range contained in $B$.
\end{enumerate}
\end{definition}

In \cite[Theorem~14.5]{Exel:noncomm.cartan} it is proved that if $B$ is a generalized
Cartan subalgebra of $A$, there exists a Fell bundle
$\A=\{\A_s\}_{s\in S}$, over a countable inverse semigroup $S$, such
that $A$ is isomorphic to $\CstarRed(\A)$ via an isomorphism which
carries $B$ onto $\CstarRed(\E)$, where $\E$ is the restriction of
$\A$ to the idempotent semilattice of $S$.  We observe that by
\cite[Proposition~4.3]{Exel:noncomm.cartan} there is no difference between
$\CstarRed(\E)$ and $C^*(\E)$, but we use the former to standardize
our notation.

In the remainder of this section we will show how
\cite[Theorem~14.5]{Exel:noncomm.cartan} combines with the results of this paper to
prove most of the conclusions of Renault's Theorem.  We therefore fix,
throughout, a separable \cstar{algebra} $A$ and a Cartan subalgebra $B$,
according to Definition~\ref{CartanRenault}.

Given that $B$ is abelian it is easy to prove that property
  \ref{CartanRenault}(ii) implies
  \ref{DefineGeneralized}(ii)' (see \cite[Proposition~9.8]{Exel:noncomm.cartan}), so
$B$ is also a generalized Cartan subalgebra.  By
\cite[Theorem~14.5]{Exel:noncomm.cartan} we therefore deduce that there exists a
countable inverse semigroup $S$ and a Fell bundle $\A=\{\A_s\}_{s\in
S}$, such that $A\simeq \CstarRed(\A)$, and $B\simeq \CstarRed(\E)$,
as above.

If $e\in E(S)$ then, by \cite[Corollary~8.9]{Exel:noncomm.cartan}, $\A_e$
identifies with an ideal of $B$, and hence $\A_e$ is commutative.  In
other words, $\A$ is a semi-abelian Fell bundle.  Moreover by the
construction of the Fell bundle in \cite[Theorem~14.5]{Exel:noncomm.cartan}, it
is also saturated (see \cite[Proposition~13.3]{Exel:noncomm.cartan}).

Let $\G$ be the étale groupoid associated to $\A$ as in Section~\ref{sec:the construction of G},
and let $L$ be the Fell line bundle over $\G$ associated to $\A$ as constructed in Section~\ref{sec:construction of L}.
As observed in Section~\ref{sec:twisted groupoids and Fell line bundles} this is tantamount
to saying that we have a twisted groupoid $(\G,\Sigma)$.
We conclude from Theorem~\ref{T:A and L have same reduced C*-algebras} that
  $$
  A \simeq \CstarRed(\A) \simeq \CstarRed(L)=\CstarRed(\G,\Sigma),
  $$
  where the isomorphisms involved restrict to give
  $$
  B \simeq C^*(\E) \simeq \contz(\Gz).
  $$

In order to get the whole of Renault's result we still need to
show that $\G$ is Hausdorff and essentially principal but,
unfortunately, we cannot offer an argument entirely based on our
results which leads to the conclusion that $\G$ is Hausdorff.

Being unable to fully prove Renault's result we conclude this
section  with some remarks on the Hausdorff question.

Of course, should we know that $\G$ is Hausdorff, it would quickly
follow that $\G$ is essentially principal by
\cite[Proposition~4.2]{RenaultCartan}.  However
we remark that there are non-Hausdorff, essentially principal, étale
groupoids for which $\contz(\Gz)$ is \emph{not} maximal abelian in
$\CstarRed(\G)$ \cite[Proposition~2.4]{ExelnHausdorff}, and hence
\cite[Proposition~4.2]{RenaultCartan} is not valid for non-Hausdorff groupoids.

Returning to the question of whether the groupoid $\G$ associated to
$\A$ is Hausdorff, one might wonder if it is possible to answer it
affirmatively via some quick argument based on the postulated
existence of the conditional expectation\footnote{In
\cite[Proposition~5.4]{RenaultCartan} Renault does achieve this in a
non-trivial way.}.  An indication that lured us in this direction is
the fact that Hausdorff étale groupoids do possess a very standard
conditional expectation obtained by restricting functions to the unit
space \cite[Proposition~4.3]{RenaultCartan}, a process which is not
available in the non-Hausdorff case.

However we have found an example which proves that the existence of
the conditional expectation, by itself (without the assumption that the
groupoid be essentially principal), is not enough to guarantee that the
groupoid is Hausdorff.

\begin{proposition}\label{prop:Non-Hausdorff groupoid with conditional expectation}
There exists a non-Hausdorff, étale groupoid $\G$,
such that $\contz(\Gz)$ is the image of a faithful conditional
expectation defined on $\CstarRed(\G)$.
\end{proposition}
\begin{proof} Partly as an illustration of our methods, we will first
construct a semi-abelian Fell bundle over an inverse semigroup, which
will then give rise by Section~\ref{sec:the construction of G} to the groupoid we need.

Consider the commutative semigroup $S = \{e,1,\sigma\}$ endowed
with the multiplication operation
$$
  \begingroup				
  \offinterlineskip
  \centerline{\vbox{
  \halign{\strut\ #\ &\vrule\ #\ &\ #\ &\ #\ \cr
  $\cdot$ &	$e$ &	1 &	$\sigma$ \cr
  \noalign{\hrule}
  $e$ &	$e$ &	$e$ &	$e$ \cr
  1 &	$e$ &	1 &	$\sigma$ \cr
  $\sigma$ &	$e$ &	$\sigma$ &	1 \cr
  }}}
  \endgroup
$$
Notice that $e$ is a {zero-element} and $1$ is a
unit for $S$.  It is easy to see that $S$ is an inverse semigroup and
that all of its elements are self-adjoint.

In presenting our Fell bundle $\A = \{\A_s\}_{s\in S}$, each fiber
$\A_s$ will be taken to be a subset of the cartesian product
  $\cont([-1,1])\times S$, as follows
\begin{itemize}
  \item $\A_e = \contz[-1,0)\times\{e\}$,
  \item $\A_1 = \cont[-1,1] \times\{1\}$,
  \item $\A_\sigma = \cont[-1,1] \times\{\sigma\}$.
\end{itemize}
The Banach space structure of each $\A_s$ is that of its first coordinate, while
the multiplication and involution on $\A$ are defined
coordinatewise. The inclusions $j_{1,e}$ and $j_{\sigma,e}$ are given by
$$
  j_{1,e}(f,e) = (f,1) \quad\mbox{and}\quad
  j_{\sigma,e}(f,e) = (f,\sigma)
  \quad\mbox{for all } f\in \contz[-1,0).
$$
Clearly $\CstarRed(\E)$ identifies with $\A_1=\cont[-1,1]$, so that
$\Gz=[-1,1]$, and one may check that the groupoid of germs $\G$ consists of the following distinct
elements
\begin{itemize}
  \item $\germ ex \quad\mbox{for } x\in[-1,0)$,
  \item $\germ 1x \quad\mbox{for } x\in[0,1]$,
  \item $\germ \sigma x \quad\mbox{for } x\in[0,1]$.
\end{itemize}
 Incidentally, notice that $\germ ex= \germ 1x =\germ \sigma x$ for all $x\in [-1,0)$.
 The topology of $\G$ is such that the open bissections
  $$
  \O_1 = \big\{\germ 1x\colon x\in[-1,1]\big\}  = \Gz \quad\mbox{and}\quad
  \O_\sigma = \big\{\germ \sigma x\colon x\in[-1,1]\big\}
  $$
are each canonically homeomorphic to $[-1,1]$, but $\G$ is not
Hausdorff since it is impossible to separate the germs $\germ 10$ and
$\germ \sigma 0$ from one another.

The reduced \cstar{algebra} of $\G$, which is isomorphic to the reduced
\cstar{algebra} of $\A$, may be described as the algebra of all continuous
complex-valued functions on the topological subspace $X$ of $\R^2$
given by
$X = \big([-1,1]\times\{0\}\big) \cup \big([0,1] \times\{1\}\big)$.
$$
 \setlength{\unitlength}{1mm}
\begin{picture}(-10,35)
  \linethickness{0.01mm}
  \put(-40,10){\vector(1,0){60}}
  \put(21,9){$x$}
  \put(-10,0){\vector(0,1){30}}
  \put(-9,31){$y$}
  \linethickness{0.3mm}
  \put(-10,20){\line(1,0){15}}
  \put(-11,19){$\bullet$}
  \put(4,19){$\bullet$}
  \put(-13,19){1}
  \linethickness{0.3mm}
  \put(-25,10){\line(1,0){30}}
  \put(-26,9){$\bullet$}
  \put(4,9){$\bullet$}
  \put(-27,6){-1}
  \put(4,6){1}
\end{picture}
$$
The natural identification of fibers within $\cont(X)$,  say
  $
  \pi_s\colon A_s \to \cont(X) \mbox{ for } s\in S,
  $
  may be given as follows: for each $f\in \contz[-1,0)$, one puts
  $$
  \pi_e(f,e)\calcat{(x,y)}=
  \left\{
  \begin{array}{cc}
    f(x), & \hbox{if $x<0$, and $y=0$, } \\
    0,  & \hbox{otherwise.}\hfill
  \end{array}\right.
  $$
For $f\in \contz[-1,1]$,
  $$
  \pi_\sigma(f,\sigma)\calcat{(x,y)}=
  \left\{
  \begin{array}{cc}
    f(x), & \hbox{if $y=0$,} \hfill \cr
    -f(x), & \hbox{if $x\geq0$, and $y=1$, }
  \end{array}\right.
  $$
  while
  $$
  \pi_1(f,1)\calcat{(x,y)}= f(x)
  \quad\mbox{for all }(x, y)\in X.
  $$

The subalgebra $\contz(\Gz)$, or equivalently $\CstarRed(\E) \ (=
\pi_1(A_1))$, may be described as the subalgebra of $\cont(X)$ formed by
the functions which do not depend on the second variable $y$.

It therefore remains to show that there does indeed exists a
conditional expectation as required.  But this may be simply given,  for
every $g\in \cont(X)$,  by
  $$
  E(g)\calcat{(x,y)} =
  \left\{
  \begin{array}{cc}
    g(x, 0), & \hbox{if $x<0$, and $y=0$, }\cr
    p(x)g(x, 0) + (1-p(x)) g(x, 1), & \hbox{if $x\geq0$,} \hfill
  \end{array}\right.
  $$
  where $p\colon [0,1] \to [0,1]$ is any continuous function such that
$p(0)=1$.  As long as the set of points $x$ where $p(x)\in(0, 1)$ is
dense in $[0,1]$, one can prove that $E$ is a faithful conditional
expectation as desired.
\end{proof}

\section{Concluding remarks}

In this work we have studied a natural connection between Fell bundles over étale groupoids and inverse semigroups.
As we have seen in Section~\ref{sec:FellBundles}, Fell bundles over étale groupoids give rise to Fell bundles over inverse semigroups,
but the latter seems to be slightly more general objects. This idea essentially appears in the unpublished work \cite{SiebenFellBundles}
by Nándor Sieben, and we would like to thank him for providing access to his work. In \cite{Exel:noncomm.cartan} the second named author gives
a first application to Fell bundles over inverse semigroups proving that they provide examples of noncommutative Cartan subalgebras.
Of course, this gives a special reason to study those objects and we are strongly inspired by this work and some of his references, especially
the articles~\cite{Kumjian:cstar.diagonals,RenaultCartan} by Kumjian and Renault.

Although arbitrary Fell bundles over inverse semigroups are more general than over étale groupoids,
our main result shows that in the semi-abelian case both can be "disintegrated" and are essentially equivalent
to twisted étale groupoids or, equivalently, Fell line bundles (over étale groupoids). For groupoids,
this result is proved in \cite[Theorem~5.6]{Deaconi_Kumjian_Ramazan:Fell.Bundles}, but it can be essentially also obtained from our
main result by first "integrating" a given semi-abelian\footnote{As already mentioned, in \cite{Deaconi_Kumjian_Ramazan:Fell.Bundles}
those Fell bundles are called just "abelian".} Fell bundle over an étale groupoid as in Section~\ref{sec:FellBundles}, 
obtaining in this way a semi-abelian Fell bundle over an inverse semigroup
and then "disintegrating" (that is, applying our main result to) it as in Section~\ref{sec:Semi-abelian Fell bundles and twisted étale groupoids} yielding the desired twisted étale groupoid.



\end{document}